\documentclass[11pt]{amsart}

\usepackage[english]{babel}     
\usepackage[utf8]{inputenc}
\usepackage{amsmath}
\usepackage{amsfonts}
\usepackage{amssymb}
\usepackage{mathrsfs}
\usepackage{stmaryrd}
\usepackage{hyperref}
\usepackage{xcolor}
\usepackage{todonotes}
\usepackage{comment}
\usepackage{enumerate}
\usepackage{tikz-cd}
\usepackage[capitalize]{cleveref}
\usepackage{tikz}

\usetikzlibrary{shapes,calc,arrows,intersections,decorations.pathreplacing,decorations.markings,shapes.geometric,calc,positioning,backgrounds}

\newtheorem{theorem}{Theorem}[section]
\newtheorem{theoremx}{Theorem}[section]

\newtheorem{lemma}[theorem]{Lemma}
\newtheorem{proposition}[theorem]{Proposition}
\newtheorem{corollary}[theorem]{Corollary}

\theoremstyle{definition}
\newtheorem{definition}[theorem]{Definition}
\newtheorem{example}[theorem]{Example}
\newtheorem{remark}[theorem]{Remark}

\theoremstyle{remark}
\numberwithin{equation}{section}

\DeclareMathOperator{\emb}{emb}
\DeclareMathOperator{\Id}{Id}

\DeclareMathOperator{\Span}{Span}
\DeclareMathOperator{\Mat}{Mat}
\DeclareMathOperator{\Cstar}{C^*}

\DeclareMathOperator{\Nor}{Nor}
\DeclareMathOperator{\qNor}{qNor}

\newcommand{\compact}{\mathbb{K}}
\newcommand{\bound}{\mathbb{B}}

\newcommand{\CC}{\mathbb{C}}

\newcommand{\EE}{\mathbb{E}}
\newcommand{\FF}{\mathbb{F}}

\newcommand{\KK}{\mathbb{K}}

\newcommand{\NN}{\mathbb{N}}

\newcommand{\ZZ}{\mathbb{Z}}

\newcommand{\calC}{\mathcal{C}}

\newcommand{\calH}{\mathcal{H}}

\newcommand{\calJ}{\mathcal{J}}
\newcommand{\calK}{\mathcal{K}}
\newcommand{\calL}{\mathcal{L}}

\newcommand{\calQ}{\mathcal{Q}}

\newcommand{\calS}{\mathcal{S}}

\newcommand{\calU}{\mathcal{U}}
\newcommand{\calV}{\mathcal{V}}
\newcommand{\calW}{\mathcal{W}}

\newcommand{\calZ}{\mathcal{Z}}

\newcommand{\boldA}{\mathbf{A}}

\newcommand{\boldM}{\mathbf{M}}

\newcommand{\uu}{\mathbf{u}}
\newcommand{\vv}{\mathbf{v}}
\newcommand{\ww}{\mathbf{w}}

\newcommand{\zz}{\mathbf{z}}

\let\epsilon\varepsilon

\DeclareMathOperator{\op}{op}

\DeclareMathOperator{\Ad}{Ad}

\DeclareMathOperator{\Dist}{Dist}
\DeclareMathOperator{\Radius}{Radius}

\newcommand{\Kompact}{\mathbb{K}}

\newcommand{\astred}{*^{\min}}

\newcommand{\Conv}{{\rm Conv}}

\newcommand{\Link}{\textrm{Link}}
\newcommand{\Star}{\textrm{Star}}
\newcommand{\Rigid}{\textrm{Rigid}}
\newcommand{\Crigid}{\mathcal{C}_{{\rm Rigid}}}
\newcommand{\Ccomplete}{\mathcal{C}_{{\rm Complete}}}
\newcommand{\Cvert}{\mathcal{C}_{{\rm Vertex}}}

\newcommand{\Cao}{\mathcal{C}_{{\rm (AO)}}}
\newcommand{\Cantifree}{\mathcal{C}_{{\rm anti-free}}}

\usepackage[date=year,backend=bibtex,doi=false,isbn=false,url=false,eprint=false, maxnames=50]{biblatex}
\title[Rigid graph products]{Rigid graph products}

\date{\noindent \today.  MSC2010 keywords:  46L51, 46L54.  MC is  supported by the NWO Vidi grant VI.Vidi.192.018 `Non-commutative harmonic analysis and rigidity of operator algebras'.}

\author[Matthijs Borst, Martijn Caspers, Enli Chen]{Matthijs Borst, Martijn Caspers, Enli Chen}

\address{TU Delft, EWI/DIAM,
	P.O.Box 5031,
	2600 GA Delft,
	The Netherlands}

\email{M.J.Borst@tudelft.nl}

\email{M.P.T.Caspers@tudelft.nl}

\email{E.Chen-1@tudelft.nl}

\begin{document}

\begin{abstract}
 We prove rigidity properties for von Neumann algebraic graph products.
 We introduce the notion of  rigid graphs and define a class of II$_1$-factors named $\calC_{\Rigid}$. For von Neumann algebras in this class we show a unique rigid graph product decomposition. In particular, we obtain unique prime factorization results and unique free product decomposition results for new classes of von Neumann algebras.
 Furthermore, we show that for many graph products of II$_1$-factors, including the hyperfinite II$_1$-factor,  we can, up to a constant 2, retrieve the radius of the graph from the graph product.
  We also prove several technical results concerning relative amenability and embeddings of (quasi)-normalizers   in graph products.
 Furthermore, we give sufficient conditions for a graph product to be nuclear and characterize   strong solidity, primeness and free-indecomposability for graph products. 
\end{abstract}

\maketitle

\section{Introduction}
The advent of Popa's deformation-rigidity theory has led to major applications to the structure of von Neumann algebras and their decomposability properties for crossed products, tensor products and free products.  For instance, in \cite{ozawaClassIIFactors2010c} Ozawa and Popa studied the notion of strongly solid von Neumann algebras (see  \cref{Dfn=StronglySolid}) and proved that the free group factors possess this property. Consequently, these von Neumann algebras do not admit certain crossed product decompositions, and they are prime factors (see \cref{Dfn=Prime}),   meaning that they can not decompose as tensor products in non-trivial way (see also \cite{ozawaSolidNeumannAlgebras2004}, \cite{popaOrthogonalPairsSubalgebras1983}).    More general prime factorization results were then obtained in e.g. \cite{ChifanKidaEtAl,  ChifanSinclairUdrea,houdayerUniquePrimeFactorization2017, isonoPrimeFactorizationResults2017,OzawaKurosh,   PetersonInventiones, SakoPrime, SizemoreWinchester,   }. In the same spirit, decompositions of von Neumann algebras in terms of free products and Kurosh type results were studied in e.g. \cite{houdayerRigidityFreeProduct2016,ioanaAmalgamatedFreeProducts2008,OzawaKurosh,    PetersonInventiones,  }.\\ 

This paper contributes to decomposability and rigidity results for von Neumann algebras that appear as graph products. The von Neumann algebraic graph product was  introduced in \cite{caspersGraphProductsOperator2017a},\cite{mlotkowskiLfreeProbability2004}. Given a simple graph $\Gamma$ and for each vertex $v \in \Gamma$ a von Neumann algebra $M_v$ with a normal faithful state $\varphi_v$, one can construct the graph product von Neumann algebra $M_{\Gamma} := *_{v,\Gamma}(M_v,\varphi_v)$, see \cref{prelim:graph-products} for the exact definition. The construction of graph products naturally generalizes the notion of free products and tensor products. Indeed, if $\Gamma$ has no edges then $M_{\Gamma}$ is simply the von Neumann algebraic free product $*_{v \in \Gamma}(M_v,\varphi_v)$, while if $\Gamma$ is a complete graph then $M_{\Gamma}$ is the tensor product $\overline{\bigotimes}_{v\in \Gamma}M_{v}$. The construction of graph products was originally introduced by Green in \cite{greenGraphProductsGroups1990} for the setting of groups. For groups, the graph product $G_{\Gamma} = *_{v,\Gamma}G_v$ generalizes both free products and direct sums. The two notions of graph products naturally correspond since for group von Neumann algebras we have $\calL(*_{v,\Gamma}G_v) = *_{v,\Gamma}\calL(G_v)$.\\ 

We will prove rigidity results for graph products of von Neumann algebras. We first discuss our main result \cref{intro:thm:rigid-graph-decomposition} which establishes unique rigid graph product decompositions. Thereafter, we give new unique prime factorization results and unique free product decomposion results.
Furthermore, we state results that characterize primeness, free indecomposability and strong solidity for graph products. Hereafter, we present other main results that are needed in the proofs. Last, we give an overview of the structure of the paper.

\subsection{Unique rigid graph product decomposition} \label{intro:sub:unique-rigid-graph-factorizaton}
 Our main result, \cref{intro:thm:rigid-graph-decomposition}, concerns the question whether from the graph product $*_{v,\Gamma}(M_v,\varphi_v)$ we can, under some conditions, retrieve the graph $\Gamma$ and the vertex von Neumann algebras $M_v$. Such questions have already been studied for graph products of groups.
In \cite[Theorem 4.12]{greenGraphProductsGroups1990}  Green showed the following rigidity result, which for graph products $*_{v,\Gamma}G_v$ of prime cycles $G_v$ asserts that the graph $\Gamma$ and the vertex groups $G_v$ can be retrieved from the graph product group.

\vspace{0.3cm}

 \noindent {\bf Theorem (Green).}  {\it Let $\Gamma,\Lambda$ be finite graphs, $G_{\Gamma} := *_{v,\Gamma}G_v$ and $H_{\Lambda} := *_{w,\Lambda}H_w$ be graph products of groups $G_v := \ZZ/p_{v}\ZZ$ and $H_w := \ZZ/q_w\ZZ$ with some prime numbers $p_v,q_w$. If  $G_{\Gamma}$ and $H_{\Lambda}$ are isomorphic, then there is a graph isomorphism $\alpha:\Gamma\to \Lambda$ such that $H_{\alpha(v)} \simeq G_v$. }

\vspace{0.3cm}

 In the current paper we prove an analogy of this result for graph products $M_{\Gamma} = *_{v,\Gamma}(M_v,\tau_v)$ of tracial von Neumann algebras $(M_v,\tau_v)$. Earlier rigidity results for von Neumann algebraic graph products have already been proven in  \cite[Theorem A and C]{chifanRigidityNeumannAlgebras2022}  for group von Neumann algebras $M_v :=\calL(G_v)$ for certain discrete property (T) groups $G_v$ and for graphs $\Gamma$ from a class called CC$_1$. In our main result, \cref{intro:thm:rigid-graph-decomposition}, we also prove rigidity results for graph products of von Neumann algebras $M_{\Gamma} = *_{v,\Gamma}(M_v,\tau_v)$. Our result compares to  \cite{chifanRigidityNeumannAlgebras2022, ChifanDaviesDrimbeII} as follows. On the one hand we cover a much richer class of graphs than CC$_1$ and our vertex von Neumann algebras $M_v$ come from a different class than \cite{chifanRigidityNeumannAlgebras2022, ChifanDaviesDrimbeII}. In our paper $M_v$ are not even necessarily group von Neumann algebras. On the other hand the type of rigidity obtained in  \cite{chifanRigidityNeumannAlgebras2022, ChifanDaviesDrimbeII}  is stronger as it recovers the groups up to isomorphism, and not just the von Neumann algebras. Furthermore, \cite{chifanRigidityNeumannAlgebras2022, ChifanDaviesDrimbeII}  obtains a so-called superrigidity result, meaning that the group can be recovered from an isomorphism of $\mathcal{L}(G)$ with any other group von Neumann algebra, whereas our rigidity results are usually for an isomorphism of two von Neumann algebras in the class $\Crigid$ introduced below.  Such a superrigidity result is simply not true in the context of the current paper as we argue in \cref{remark:need-for-rigid-graphs}.\\

  The condition we impose on the vertex von Neumann algebras $M_v$ is that they lie in the class $\Cvert$ of all non-amenable II$_1$-factors that satisfy property strong (AO) (see \cref{def:class-vertex-class-rigid}) and have separable preduals. This is a natural class of von Neumann algebras including the (interpolated) free group factors $\calL(\FF_t)$ for $1<t < \infty$, the group von Neumann algebras $\calL(G)$ of non-amenable hyperbolic icc groups $G$ \cite{HigsonGuentner}, $q$-Gaussian von Neumann algebras $M_q(H_{\mathbb{R}})$ associated with real Hilbert spaces $H_{\mathbb{R}}$ with $2 \leq \dim(H_{\mathbb{R}}) < \infty$ \cite[Remark 4.5]{borstIsomorphismClassGaussian2023}, \cite{Kuzmin}, free orthogonal quantum groups \cite{VaesVergnioux} as well as several common series of easy  quantum groups and free wreath products of quantum groups \cite[Theorem 0.5]{CaspersJussieu}.  

   \vspace{0.3cm}

  The condition we impose on the graph $\Gamma$ is that each vertex $v$ in $\Gamma$ satisfies $\Link(\Link(v)) = \{v\}$ (for the definition of $\Link$ see \cref{prelim:simple-graph}). Such graphs, which we call \textit{rigid}, form a large natural class of graphs containing for example complete graphs and cyclic graphs with at least 5 vertices. We also observe that all graphs in CC$_1$ are rigid (see \cref{remark:class-CC1}). We stress that some restrictions on the graphs need to be imposed. Indeed, for general graphs $\Gamma$, and  graph products
  $M_{\Gamma}= *_{v,\Gamma}(M_v,\tau_v)$ with $M_v\in \Cvert$, it is not possible to retrieve the graphs $\Gamma$ from $M_{\Gamma}$ (see \cref{remark:need-for-rigid-graphs}). This is due to the fact that the free product $(M_v,\tau_v) *(M_w,\tau_w)$ of factors $M_v,M_w\in \Cvert$ again lies again in the class $\Cvert$ (see \cref{remark:classes-graph-products}).\\
    
We now state our main result which shows rigidity for the class $\Crigid$ of all graph products $M_{\Gamma} = *_{v,\Gamma}(M_v,\tau_v)$ with $\Gamma$ non-empty, rigid graphs and with $M_v\in \Cvert$. 

\begin{theoremx}[\cref{thm:rigid-graph-decomposition} and \cref{thm:rigid-graph-decomposition-irreducible-components}]
\label{intro:thm:rigid-graph-decomposition}
 Let $\Gamma$ be a rigid graph and for $v \in \Gamma$ let $M_v$ be von Neumann algebras in class $\Cvert$ with faithful normal state $\tau_v$.
Let $M_\Gamma = \ast_{v, \Gamma} (M_{v},\tau_v)$ be their graph product. Suppose there is another graph product decomposition of $M_{\Gamma}$ over another rigid graph $\Lambda$ and other von Neumann algebras $N_w\in\Cvert$, $w\in\Lambda$, i.e. $M_{\Gamma}=*_{w,\Lambda}(N_w,\tau_w)$.
Then there is a graph isomorphism $\alpha: \Gamma \rightarrow \Lambda$, and for each $v\in \Gamma$ there is a unitary $u_v\in M_{\Gamma}$ and a real number $0<t_v<\infty$ such that:
\begin{align}
    M_{\Star(v)}= u_v^*N_{\Star(\alpha(v))}u_v& &and&  
 &M_v \simeq N_{\alpha(v)}^{t_v}.
\end{align}
Furthermore, for the connected component $\Gamma_v\subseteq \Gamma$ of any vertex $v\in \Gamma$, we have $M_{\Gamma_v} = u_v^*N_{\alpha(\Gamma_v)}u_v$; and for any irreducible component $\Gamma_0\subseteq \Gamma$, $\exists t_0\in(0,\infty)$ such that $M_{\Gamma_0} \simeq N_{\alpha(\Gamma_0)}^{t_0}$.
\end{theoremx}

We remark that in the setting of \cite[Theorem 7.9]{chifanRigidityNeumannAlgebras2022}, it is possible to obtain unitary conjugacy between the vertex von Neumann algebras $M_v = \calL(G_v)$. In our setting it is generally only possible to obtain isomorphisms up to amplification between the vertex von Neumann algebras. The reason is that the tensor product $M_v\overline{\otimes} M_w$ of II$_1$-factors is isomorphic to the tensor product $M_v^t \overline{\otimes} M_w^{1/t}$ for any $0<t<\infty$.  This is however the only obstruction to unitary conjugacy of the vertex von Neumann algebras in Theorem A. Indeed, for certain subgraphs $\Gamma_0\subseteq \Gamma$ we show in Theorem A unitary conjugacy of the graph products $M_{\Gamma_0}$ to $N_{\alpha(\Gamma_0)}$ inside $M_{\Gamma}$; in particular this applies to the  case $\Gamma_0 = \Star(v)$ for some vertex $v$ of $\Gamma$ (for the definition of $\Star$ see \cref{prelim:simple-graph}). We also obtain unitary conjugacy in case $\Gamma_0$ is a connected component of $\Gamma$.   Moreover, for $\Gamma_0$ an irreducible component of $\Gamma$ we are able to show that $M_{\Gamma_0}$ is isomorphic to an amplification of $N_{\alpha(\Gamma_0)}$.  \\

\subsection{Unique prime factorization}\label{intro:sub:unique-prime-factorizaton}
For classes of von Neumann algebras we are interested in unique prime factorization results. Recall that a II$_1$-factor $M$ is prime if it can not decompose as a tensor product $M=M_1\overline{\otimes} M_2$ of diffuse factors $M_1,M_2$. The first example of a prime factor was given by Popa in \cite{popaOrthogonalPairsSubalgebras1983}. Thereafter, Ge showed in \cite{gePrimeFactors1996} that  $\calL(\FF_n)$ is a prime factor for $n\geq 2$ by computing Voiculescu's free entropy. Later, in \cite{ozawaSolidNeumannAlgebras2004} Ozawa introduced a new property, called solidity, which for non-amenable factors implies primeness. He showed that all finite von Neumann algebras satisfying the Akemann-Ostrand property are solid. We note that in particular all von Neummann algebras in $\Cvert$ are prime. There are many more examples of prime factors, see e.g.  \cite{boutonnetStrongSolidityFree2018, chifanTensorProductDecompositions2018,   ChifanKidaEtAl,  ChifanSinclairUdrea, drimbePrimeII1Factors2019, PetersonInventiones, SakoPrime,  SizemoreWinchester,    }.\\

Given a class $\calC$ of von Neumann algebras, a natural question is whether any von Neumann algebra $M\in \calC$ has a tensor product decomposition $M=M_1\overline{\otimes} \cdots \overline{\otimes} M_m$ for some $m\geq 1$ and prime factors $M_1,\ldots, M_m\in \calC$, which is called prime  factorization inside $\calC$, and whether the prime factorization is unique. This is to say, given another prime factorization $M=N_1\overline{\otimes} \cdots \overline{\otimes} N_n$, with $n\geq 1$ and prime factors $N_1,\ldots, N_n\in \calC$, do we have $n=m$ and, up to permutation of the indices, any $M_i$ is isomorphic to an amplification of $N_i$. The first unique prime factorization (UPF) results were established by Ozawa and Popa in \cite{ozawaPrimeFactorizationResults2004} for tensor products of group von Neumann algebras $\calL(G_v)$ for certain groups $G_v$. The groups they considered included non-amenable, icc groups that are hyperbolic or are discrete subgroups of connected simple Lie groups of rank one. Later, in \cite{isonoPrimeFactorizationResults2017} Isono studied UPF results for free quantum group factors. Thereafter, by combining results from \cite{ozawaPrimeFactorizationResults2004} and \cite{isonoPrimeFactorizationResults2017},  Houdayer and Isono showed in \cite{houdayerUniquePrimeFactorization2017}  more general UPF results for tensor products of factors from a class called $\Cao$. We note that our class $\Cvert$ is very similar to $\Cao$ and that $\Cvert\subseteq \Cao$. In the setting of graph products, UPF results have been obtained in \cite[Theorem 6.16]{chifanTensorProductDecompositions2018} under the condition that the vertex von Neumann algebras are group von Neumann algebras. \\

We observe that we can use \cref{intro:thm:rigid-graph-decomposition} to obtain UPF results. Indeed, let $\Ccomplete$ be the class of all tensor products of von Neumann algebras in $\Cvert$. If in \cref{intro:thm:rigid-graph-decomposition} we restrict our attention to complete graphs (which are rigid) then we precisely obtain UPF results for the class $\Ccomplete$ (see \cref{corollary:rigidy-for-complete-graphs}). This partially retrieves the UPF results from \cite{houdayerUniquePrimeFactorization2017}. To obtain more general UPF results we prove the following result which characterizes primeness for graph products of II$_1$-factors (see also \cref{thm:primeness:irreducible-graph,thm:primeness:graph-product} in the case the vertex von Neumann algebras are not II$_1$-factors).
\begin{theoremx}[\cref{thm:primeness-for-graph-products-II1-factors}]\label{intro:thm:primeness-for-graph-products-II1-factors}
    Let $\Gamma$ be a finite simple graph of size $|\Gamma|\geq 2$. For any $v\in \Gamma$, let $M_v$ be a II$_1$-factor. The graph product $M_{\Gamma} = *_{v,\Gamma}(M_v,\tau_v)$ is prime if and only if $\Gamma$ is irreducible.
\end{theoremx}
We then use \cref{intro:thm:rigid-graph-decomposition} and \cref{intro:thm:primeness-for-graph-products-II1-factors} to prove the following theorem which covers UPF results for a new class of von Neumann algebras (see \cref{remark:new-UPF-results}). 
\begin{theoremx}[\cref{thm:unique-prime-factorization}]
    \label{intro:thm:unique-prime-factorization}
    Any von Neumann algebra $M\in \Crigid^f$ has a prime factorization inside $\Crigid^f$, i.e.
    \begin{align}
        M = M_1\overline{\otimes} \cdots \overline{\otimes} M_m,
    \end{align}
    for some $1\leq m< \infty$ and prime factors $M_1,\ldots, M_m\in \Crigid^f$.  
    
    Suppose $M$ has another prime factorization inside $\Crigid^f$, i.e.
    \begin{align}
        M = N_1\overline{\otimes} \cdots \overline{\otimes} N_n,
    \end{align}
    for some $1\leq n< \infty$, and prime factors $N_1,\ldots, N_n\in \Crigid^f$.
    Then $m=n$ and there is a permutation $\sigma$ of $\{1,\ldots, m\}$ such that $M_{i}$ is stably isomorphic to $N_{\sigma(i)}$ for $1\leq i\leq m$.
\end{theoremx}

\subsection{Unique free product decomposition}
\label{intro:sub:unique-free-factorizaton}
In \cite{OzawaKurosh} Ozawa extended the results \cite{ozawaPrimeFactorizationResults2004} for tensor products to the setting of free products. In particular, he showed for $M=M_1*\cdots *M_m$ a von Neumann algebraic free products of non-prime, non-amenable, semiexact finite factors $M_1,\ldots, M_m$ that if $M = N_1*\cdots *N_n$ is another free product decomposition into non-prime, non-amenable, semiexact finite factors $N_1,\ldots, N_n$, then $m=n$ and, up to permutation of the indices, $M_i$ unitarily conjugates to $N_i$ inside $M$ for each $1<i<m$.
This can be seen as a von Neumann algebraic version of the Kurosh isomorphism theorem \cite{Kurosch}, which states that any discrete group uniquely decomposes as a free product of freely-indecomposable groups. Versions of Ozawa's result were later shown for other classes of von Neumann algebras, see \cite{asherKuroshtypeTheoremType2009},\cite{ioanaAmalgamatedFreeProducts2008},\cite{PetersonInventiones}.
In \cite{houdayerRigidityFreeProduct2016} these results were then extended by Houdayer and Ueda to a single, large class of von Neumann algebras,   that includes free products of nonamenable factors that are either (i) non-prime, (ii) have property Gamma, (iii) possess a Cartan subalgebra, or (iv) are of type II$_1$ and possess a regular diffuse von Neumann subalgebra with relative property (T).   Other Kurosh type theorems have recently been obtained in \cite[Corollary 8.1]{drimbeMeasureEquivalenceRigidity2023}, \cite[Corollary 1.8]{dingStructureRelativelyBiexact2024}.\\

In the current paper we obtain unique free product decomposition results for a new class of von Neumann algebras. First, we prove the following result which characterizes precisely when a graph product $M_\Gamma = *_{v,\Gamma}(M_v,\tau_v)$ can decompose as tracial free product of II$_1$-factors.
\begin{theoremx}[\cref{theorem:free-indecompose}]\label{intro:theorem:free-indecompose}
Let $\Gamma$ be a simple graph of size $|\Gamma|\geq 2$, and for each $v\in \Gamma$ let $M_v$ be II$_1$-factor with separable predual. Then the graph product $M_{\Gamma}:= *_{v,\Gamma}(M_{v},\tau_v)$ can decompose as a tracial free product $M_{\Gamma} = (M_1,\tau_1)*(M_2,\tau_2)$ of II$_1$-factors  $M_1$,$M_2$ if and only if $\Gamma$ is not connected.
\end{theoremx}

Using  \cref{intro:thm:rigid-graph-decomposition} and \cref{intro:theorem:free-indecompose} we obtain unique free product decomposition for the class $\Crigid\setminus \Cvert$.

\begin{theoremx}[\cref{thm:free-product-decomposition}]
\label{intro:thm:free-product-decomposition}
      Any von Neumann algebra $M\in \Crigid\setminus \Cvert$ can decompose as a tracial free product inside $\Crigid\setminus\Cvert$, i.e.
     \begin{align}\label{intro:eq:free-product-decomposition}
         M = *_{i\in I}M_i,
     \end{align}
     for some index set $I$, and for every $i\in I$ a factor $M_i\in \Crigid\setminus \Cvert$ that can not decompose as any tracial free product of II$_1$-factors. 
     
     Suppose $M$ can decompose as another tracial free product inside  $\Crigid\setminus\Cvert$, i.e.
     $$M=*_{j\in J}N_j$$ 
     for another index set $J$ and for every $j\in J$ a factor $N_j\in \Crigid\setminus\Cvert$ that can not decompose as tracial free product of II$_1$-factors. Then $|I|=|J|$ and there is a bijection $\sigma$ between $J$ and $I$ such that $N_j$ unitarily conjugates to  $M_{\sigma(j)}$ in $M$.
\end{theoremx}
Let us remark that von Neumann algebras in the class $\Ccomplete\setminus \Cvert$ are examples of non-prime, non-amenable, semiexact, finite factors. Thus Ozawa's result in particular asserts a unique free product decomposition for free products of factors in $\Ccomplete\setminus \Cvert$.
The same result is also covered by 
\cref{intro:thm:free-product-decomposition} since any free product of factors in $\Ccomplete\setminus \Cvert$ lies in the class $\Crigid\setminus \Cvert$. We observe that, in contrast to Ozawa's result, in \cref{intro:thm:free-product-decomposition} it is possible for the factors $M_i, i \in I$ to be prime. More generally, we remark that the result of \cref{intro:thm:free-product-decomposition} is not covered by the result from  \cite{houdayerRigidityFreeProduct2016} (see \cref{remark:class-anti-free}). Thus our examples of unique free product decompositions are again new.\\

\subsection{Graph radius rigidity}
We are interested in the question whether from the graph product  $M_{\Gamma} = *_{v,\Gamma}(M_v,\tau_v)$ of II$_1$-factors $M_v$ we can retrieve the radius of the graph $\Gamma$. To study this question we introduce the notion of the radius of a von Neumann algebra $M$ (see \cref{definition:radius-von-Neumann-algebra}). As we show in the following theorem, we are in many cases able to estimate the radius of the von Neumann algebra $M_{\Gamma}$ with the radius of the graph $\Gamma$.
\begin{theoremx}[\cref{theorem:radius-graph-AO} and \cref{theorem:radius-graph-groups}]\label{theorem:graph-radius-rigidity}
    Let $\Gamma$ be a non-complete simple graph. For $v\in \Gamma$ let $M_v$ be a II$_1$-factor and let $M_{\Gamma} = *_{v,\Gamma}(M_v,\tau_v)$ be the tracial graph product. 
    Suppose one of the following holds true:
    \begin{enumerate}
        \item For all $v\in \Gamma$ the vertex algebra $M_v$ possesses strong (AO) and has separable predual.
        \item For all $v\in \Gamma$ we have $M_v = \calL(G_v)$ for some countable icc group $G_v$.
    \end{enumerate}
    Then,
    $$\Radius(\Gamma)-2 \leq \Radius(M_{\Gamma})\leq \max\{2,\Radius(\Gamma)\}.$$
\end{theoremx}
The above result allows us to distinguish certain von Neumann algebras coming from graph products. In particular, for graph products $R_{\Gamma_i}=*_{v,\Gamma_i}(R_v,\tau_v)$ of hyperfinite II$_1$-factors $R_v$, we are able to show that $R_{\Gamma_1}\not\simeq R_{\Gamma_2}$ whenever $2\leq \Radius(\Gamma_1)<\Radius(\Gamma_2)-2$ (see \cref{remark:distinguish-vNa-using-radius}).

We remark that when $\Lambda_i$ for $i=1,2$ are graphs of size $2\leq |\Lambda_1| =: n<  |\Lambda_2| =: m$ and with no edges, then  $R_{\Lambda_1} = \calL(\FF_{n})$ and $R_{\Lambda_2} = \calL(\FF_m)$ by \cite{dykemaInterpolatedFreeGroup1994}. 
In this case, it is very hard to distinguish $R_{\Lambda_1}$ from $R_{\Lambda_2}$ as this is precisely the free factor problem. Of course, \cref{theorem:graph-radius-rigidity} is of no use here since $\Radius(\Lambda_1) = \infty = \Radius(\Lambda_2)$.

\subsection{Strong solidity}
For a finite von Neumann algebra $M$ the notion of strong solidity was introduced by Ozawa and Popa in \cite{ozawaClassIIFactors2010c}. This property, which in particular implies solidity, asserts that for any diffuse amenable von Neumann subalgebra $A\subseteq M$, its normalizers $\Nor_{M}(A)$ generates a von Neumann algebra that is amenable. This property implies that for a non-amenable von Neumann algebra it does not have a Cartan subalgebra, and hence can not decompose as a crossed product in a natural way. {In \cite{ozawaClassIIFactors2010c}, it was shown in that the free group factors $\calL(\FF_t)$ are strong solidity. Nowadays, many examples of strongly solid von Neumann algebras are known, see e.g. \cite{CaspersJussieu,chifanStructuralTheoryRm2013,ding2023biexactvonneumannalgebras, IsonoTAMS, popaUniqueCartanDecomposition2012,     }.
  Moreover, we remark that using the resolution of the Peterson-Thom conjecture (see \cite{hayesRandomMatrixApproach2022a}, \cite{bordenaveNormMatrixvaluedPolynomials2024} ,\cite{belinschiStrongConvergenceTensor2024}), it has been shown in \cite{hayesConsequencesRandomMatrix2024} that the free group factors even satisfy a strengthened version of strong solidity.

In this paper we study strong solidity for graph products of von Neumann algebras. In \cite{borstClassificationRightangledCoxeter2023} strong solidity was characterized for group von Neumann algebras $\calL(\calW_{\Gamma})$ of right-angled Coxeter groups $\calW_{\Gamma}$. Using similar techniques, we characterize strong solidity for arbitrary graph products over finite graphs.

\begin{theoremx}[\cref{Thm=MainImplication}] \label{intro:Thm=MainImplication}
	Let $\Gamma$ be a finite simple graph, and for each $v\in \Gamma$ let $M_{v}$ ($\not=\CC$) be a von Neumann algebra with normal faithful trace $\tau_{v}$. Then $M_{\Gamma}$ is strongly solid if and only if the following conditions are  satisfied:
	\begin{enumerate}
		\item   For each vertex  $v\in \Gamma$ the von Neumann algebra $M_v$ is strongly solid;
    \item    For each subgraph $\Lambda\subseteq \Gamma$ with $M_{\Lambda}$ non-amenable, we have that $M_{\Link(\Lambda)}$  is not diffuse;
		\item   For each subgraph $\Lambda\subseteq \Gamma$ with $M_{\Lambda}$ non-amenable and diffuse, we have moreover that $M_{\Link(\Lambda)}$  is atomic. 
   \end{enumerate}
\end{theoremx}
We remark that for a large class of vertex von Neumann algebras $M_v$ it can be verified whether the conditions (1),(2) and (3) hold true for the graph products $M_{\Lambda}$ and $M_{\Link(\Lambda)}$.  In particular, for right-angled Hecke von Neumann algebras this characterizes strong solidity (using \cref{thm:Hecke-algebras:diffuseness-and-amenability} from \cite{caspersGraphProductKhintchine2021}, \cite{raumFactorialMultiparameterHecke2023}). Partial results in this direction had already been obtained in 
\cite{borstBimoduleCoefficientsRiesz2023}
and
\cite{caspersAbsenceCartanSubalgebras2020}.

\subsection{Other results}
The proofs of the stated theorems require several main results that are of independent interest, which we present here. Firstly, we give sufficient conditions for a graph product of unital C$^\ast$-algebras to be nuclear. This is a generalizion of Ozawa's result for free products \cite{OzawaNuclearFree} and is needed in the proof of \cref{intro:thm:rigid-graph-decomposition}.

\begin{theoremx}[\cref{thm:nuclearity-graph-productNEW}]
\label{intro:thm:nuclearity-graph-product} Let $\Gamma$ be a simple graph. 
Let $A_{\Gamma} = \astred_{v,\Gamma}(A_v,\varphi_v)$ be the reduced C$^\ast$-algebraic graph product of nuclear, unital C$^\ast$-algebras $A_v$ with GNS-faithful state $\varphi_v$. Let $\calH_v := L^2(A_v,\varphi_v)$ and let $\pi_v:A_v\to \bound(\calH_v)$ be the GNS-representation. If for any $v\in \Gamma$,   $\pi_v(A_v)$   contains the space of compact operators $\compact(\calH_v)$, then $A_\Gamma$  is nuclear. 
\end{theoremx}

The following result is the graph product analogue of \cite[Theorem 5.1]{houdayerUniquePrimeFactorization2017} and \cite[Theorem 3.3.]{OzawaKurosh}, and is crucial in the proof of  \cref{intro:thm:rigid-graph-decomposition} for establishing the graph isomorphism.
\begin{theoremx}[\cref{Thm=KeyAlternatives}]
\label{intro:Thm=KeyAlternatives}
    Let $\Gamma$ be a simple graph, $(M_\Gamma, \tau) = \ast_{v, \Gamma} (M_v, \tau_v)$ be the graph product of finite von Neumann algebras $M_v$ that satisfy condition strong (AO) and have separable preduals. Let $Q \subseteq M_\Gamma$ be a diffuse von Neumann subalgebra. At least one of the following holds:
    \begin{enumerate}
    \item \label{Item=MainOneIntro}The relative commutant $Q' \cap M_\Gamma$ is amenable;
    \item \label{Item=MainTwoIntro} There exists $\Gamma_0  \subseteq \Gamma$ such that $Q \prec_{M_\Gamma} M_{\Gamma_0}$ and $\Link(\Gamma_0) \not = \varnothing$. 
    \end{enumerate}
\end{theoremx}

The following result concerning relative amenability extends \cite[Theorem 3.7]{borstClassificationRightangledCoxeter2023} to the setting of graph product. This result is needed in the proof of the characterizations given in \cref{intro:thm:primeness-for-graph-products-II1-factors}, \cref{intro:theorem:free-indecompose} and \cref{intro:Thm=MainImplication}.

\begin{theoremx}[\cref{Thm=Square}]\label{intro:thm:Square}
	Let $\Gamma$ be a simple graph with subgraphs $\Gamma_1, \Gamma_2\subseteq \Gamma$. For each $v\in \Gamma$ let $(M_{v},\tau_{v})$ be a von Neumann algebra with a normal faithful trace.
	Let $P \subset M_\Gamma$ be a von Neumann subalgebra that is amenable relative to $M_{\Gamma_i}$ inside $M_{\Gamma}$ for  $i = 1,2$.   Then $P$ is amenable relative to $M_{\Gamma_1\cap \Gamma_2}$ inside $M_{\Gamma}$.
\end{theoremx}

\subsection{Paper overview}
In \cref{Sect=Prelim} we recall Popa's interwtining-by-bimodule technique and introduce our notation for simple graphs and for graph products.
In \cref{Sect=RigidGraph} we introduce the notion of rigid graphs and study some basic properties. Here, we also define graph products of graphs and precisely characterize when a graph product of graphs is rigid.
In \cref{Sect=NuclearGraphProduct} we prove \cref{intro:thm:nuclearity-graph-product} which establishes sufficient conditions for a graph product to be nuclear. In \cref{Sect=technical-results-graph-products} we prove some technical results concerning graph products.  In particular, using calculation for iterated conditional expectations in graph products we prove \cref{intro:thm:Square} regarding relative-amenability, and prove some embedding results for graph products.   In \cref{Sect=GraphProductRigidity} we prove \cref{intro:Thm=KeyAlternatives} which we then use to prove the major part of \cref{intro:thm:rigid-graph-decomposition}. In \cref{Sect=StrongSolidity} we prove \cref{intro:Thm=MainImplication} which characterizes strong solidity for graph products. In \cref{Sect=Primeness} we prove \cref{intro:thm:primeness-for-graph-products-II1-factors} which characterizes primeness in graph products. Moreover, we also complete the proof of \cref{intro:thm:rigid-graph-decomposition} and we prove \cref{intro:thm:unique-prime-factorization} which establishes UPF results for the class $\Crigid$. In \cref{Sect=freely-indecomposable} we prove \cref{intro:theorem:free-indecompose} which characterizes free-indecomposability for graph products and we prove \cref{intro:thm:free-product-decomposition} which establishes unique free product decomposition results for the class $\Crigid\setminus\Cvert$. Last, in \cref{section:graph-radius-rigidity} we define the radius of a von Neumann algebra and prove \cref{theorem:graph-radius-rigidity} which for graph products provides good estimates on the radius of the graph.

\subsection*{Comments} After our results have appeared on the arXiv new rigidity results for graph products have appeared in \cite{CaspersChen}, \cite{VaesDrimbe} and \cite{IoanaHorbez} for new classes of graph products. In particular \cite{CaspersChen}, \cite{VaesDrimbe} also consider infinite graphs. After these results appeared we revised this manuscript and updated earlier statements for finite graphs that hold for infinite graphs as well with the same proof.

\subsection*{Acknowledgements} The authors wish to thank David Jekel, Ben Hayes and the anonymous referees for communicating several suggestions of improvement for our paper.

\section{Preliminaries}\label{Sect=Prelim} 

\subsection{General notation} For a Hilbert space $\calH$ we denote $\bound(\calH)$ for the space of bounded operators on $\calH$ and denote $\compact(\calH)$ for the space of compact operators on $\calH$. For a von Neumann algebra $M$ we denote $\calZ(M) = M\cap M'$ for the center and denote $\calU(M)$ for the set of all unitaries. 
Furthermore, for $0<t<\infty$ and a II$_1$-factor $M$ we denote $M^t$ for the amplification of $M$ by $t$.
For $u \in \calU(\bound(\calH))$ we write $\Ad_u(x) = u x u^\ast, x \in \bound(\calH)$.
Inclusions of von Neumann algebras are always assumed to be unital inclusions. We write $1_M$ for the unit of a von Neumann algebra $M$. For a von Neumann subalgebra $A\subseteq 1_{A}M1_{A}$ we denote the set of normalizers and quasi-normalizers respectively by 
\begin{align*}
    \Nor_{M}(A) &:= \{u\in \calU(M): u^*Au = A\},\\
    \qNor_{M}(A) &:= \{x\in M: \exists x_1,\ldots, x_n,y_1,\ldots, y_m \in M \text{ s.t. } Ax \subseteq \sum_{i=1}^n x_iA \text{ and } xA \subseteq \sum_{i=1}^m A y_i\}.
\end{align*}
Tensor products of von Neumann algebras are defined as the strong closure of the their spacial tensor products.

\subsection{Jones extension}
Let $(M,\tau)$ be a tracial von Neumann algebra and $Q\subseteq M$ a  von Neumann subalgebra. We denote $\EE_{Q}:M\to Q$ for the trace-preserving conditional expectation, and denote $e_{Q}:L^2(M,\tau)\to L^2(Q,\tau)$ for its $L^2$-extension. We regard $M\subseteq \bound(L^2(M,\tau))$ as a subalgebra, and denote $\langle M,e_{Q}\rangle$ for the Jones extension which is the von Neumann algebra generated by $M\cup \{e_{Q}\}$. 

\subsection{Popa's intertwining-by-bimodule theory}\label{Sect=IntertwiningByBimodules} 

We recall the following definition from the fundamental work of \cite{popaStrongRigidityII12006, popaStrongRigidityII12006a}. In this section we let $M$ be a finite von Neumann algebra. 

\begin{definition}[Embedding $A \prec_{M}  B$]\label{Dfn=Intertwine} 
For von Neumann subalgebras $A \subseteq 1_A M 1_A,  B \subseteq 1_B M1_B$ we will say that \textit{$A$ embeds in $B$ inside $M$} (denoted by $A \prec_{M} B$) if one of the following equivalent statements holds:
 \begin{enumerate}
     \item\label{Item=Intertwine1}  There exist projections $p\in A$, $q\in B$, a  normal $\ast$-homomorphism $\theta: p A p \to q B q$ and a non-zero partial isometry $v\in q M p$ such that
	$\theta(x) v = vx$ for all $x\in p A p$;
 \item\label{Item=Intertwine2} There exists no net of unitaries $(u_i)_i$ in $A$ such that for any $x,y \in 1_A M 1_B$ we have that $\Vert \mathbb{E}_B(x^\ast u_i y) \Vert_2 \rightarrow 0$;
 \item  \label{Item=Intertwine3}  There exists a Hilbert $A$-$B$ bimodule $\calK \subseteq L^2(M, \tau)$ such that $\dim_{B} \calK < \infty$.  
 \end{enumerate}
 We say that \textit{$A$ embeds stably in $B$ inside $M$} (denoted by $A \prec_{M}^s B$) if  for any projection $r \in A' \cap M$ we have  $Ar \prec_{M} B$.
\end{definition}

\begin{lemma} [ Lemma 2.4 in \cite{drimbePrimeII1Factors2019}, see also \cite{vaesExplicitComputationsAll2008}]\label{Lem=StableEmbedding}
    Let $(M, \tau)$ be a tracial von Neumann algebra and let $P\subseteq 1_PM1_P$, $Q\subseteq 1_{Q}M1_{Q}$ and $R\subseteq 1_{R}M1_{R}$ be von Neumann subalgebras. Then the following hold:
    \begin{enumerate}
        \item \label{Item=StableEmbedding:transative}Assume that $P\prec_M Q$ and $Q\prec_M^s R$. Then $P\prec_M R$;
        \item \label{Item=StableEmbedding:condition} Assume that, for any non-zero projection $z\in \Nor_{1_{P}M1_{P}}(P)'\cap 1_{P}M1_{P}\subseteq \calZ(P'\cap 1_{P}M1_{P})$, we have $Pz\prec_M Q$. Then $P\prec_M^s Q$.
    \end{enumerate}
    In particular, we note that if $Q'\cap 1_{Q}M1_{Q}$ is a factor and $P\prec_{M}Q$ and $Q\prec_{M} R$ then $P\prec_{M} R$.
\end{lemma}

\subsection{Simple graphs}
\label{prelim:simple-graph}

A \textit{simple graph} $\Gamma$ is an undirected graph without double edges and without edges between a vertex and itself. We write $v \in \Gamma$ for saying that $v$ is a vertex of $\Gamma$. We write $\Lambda \subseteq \Gamma$ to say that $\Lambda$ is a subgraph of $\Gamma$, meaning that the vertices of $\Lambda$ form a subset of the vertices of $\Gamma$, and two vertices in $\Lambda$ share an edge iff they share an edge in $\Gamma$. For $v \in \Gamma$ we set 
\[
\begin{split}
\Link_{\Gamma}(v) = & \{ w \in \Gamma \mid  v \textrm{ and } w \textrm{ share an edge} \}, \\
\Star_{\Gamma}(v) =&   \{ v \} \cup \Link_\Gamma(v).  
\end{split}
\]
This entails in particular that $v \not \in \Link_{\Gamma}(v)$.   For $\Lambda \subseteq \Gamma$ we set $\Link_{\Gamma}(\Lambda) = \bigcap_{v \in \Lambda} \Link(v)$ with the convention $\Link_{\Gamma}(\varnothing) = \Gamma$. We remark that $v\in \Link_{\Gamma}(\Link_{\Gamma}(v))$. When the graph $\Gamma$ is fixed, we will leave out the subscript and simply write $\Link(v),\Star(v),\Link(\Lambda)$. We denote $|\Gamma|$ for the size of the graph, i.e. the number of vertices. We will call a graph $\Gamma$ \textit{irreducible} if we can not write $\Gamma = \Gamma_1\cup \Gamma_2$ for some non-empty subgraphs $\Gamma_1,\Gamma_2\subseteq \Gamma$ with $\Link(\Gamma_1)=\Gamma_2$. We call a graph $\Gamma$ connected if there exists a path between any two distinct vertices $v,w\in \Gamma$. An irreducible component of a graph $\Gamma$ is a non-empty subgraph $\Lambda\subseteq \Gamma$ that is irreducible and satisfies $\Link_{\Gamma}(\Lambda) = \Gamma\setminus \Lambda$. 
A connected component of a graph $\Gamma$ is a subgraph $\Lambda\subseteq \Gamma$ that is connected and satisfies for $v\in \Lambda$ that $\Link_{\Gamma}(v)\subseteq \Lambda$. We call a graph $\Gamma$ complete if any two vertices in $\Gamma$ share an edge. A complete subgraph  $\Lambda\subseteq \Gamma$ is called a clique of $\Gamma$.

\begin{lemma} \label{Lem=CompleteAdd}
Let $\Gamma$ be a connected simple graph such that for every $v \in \Gamma$ we have that $\Link(v)$ is a clique. Then $\Gamma$ is a complete graph. 
\end{lemma}
\begin{proof}
Let $v \in \Gamma$. As $\Link(v)$ is a clique $\Star(v)$ is complete. Now take $w \in \Link(v)$. Then any point in $\Link(w)$ is a clique and as $v \in \Link(w)$ every point in $\Link(w)$ is connected to $v$  and thus contained in $\Star(v)$. We conclude that $\Star(v)$ is a connected component of $\Gamma$. As $\Gamma$ is connected it equals $\Star(v)$ and thus is complete.   
\end{proof}

\subsection{Right-angled Coxeter groups}
Let $\Gamma$ be a simple graph. We let $\calW_{\Gamma}$ be the right-angled Coxeter group corresponding to the graph $\Gamma$, that is, $\calW_{\Gamma}$ is the group generated by the vertex set of $\Gamma$, subject to the relations that $vw=wv$ whenever $v,w$ share an edge in $\Gamma$ and $v^2 = e$ with $e$ the group unit (thus $\calW_{\Gamma}$ is equal to the graph product group $*_{v,\Gamma}(\ZZ/2\ZZ)$). For $\vv\in \calW_{\Gamma}$ we write $|\vv|$ for the length of $\vv$. For $\vv_1,\ldots, \vv_n\in \calW_{\Gamma}$ we say that the expression $\vv=\vv_1\cdots \vv_n$ is \textit{reduced} if $|\vv| = |\vv_1| + \ldots + |\vv_n|$.   We call a word $\ww\in \calW_{\Gamma}$ a \textit{clique word} if it has a reduced expression $\ww = w_1\cdots w_n$ with $w_i\in \Gamma$ such that $w_i$ commutes with $w_j$ for any $1\leq i,j\leq n$. For  a reduced word $\vv =v_1\cdots v_n\in \calW_{\Gamma}$ we write 
\begin{equation}\label{Eqn=LinkForRef}
\Link_{\Gamma}(\vv) = \bigcap_{1\leq i\leq n}\Link_{\Gamma}(v_i).
\end{equation}
For a subset $S\subseteq \calW_{\Gamma}$ we denote
\[
\begin{split}
    \calW(S) &:= \{\ww\in \calW_{\Gamma}: \uu\ww \text{ is reduced for any } \uu\in S\}; \\
     \calW'(S) &:= \{\ww\in \calW_{\Gamma}: \ww\uu \text{ is reduced for any } \uu\in S\}. 
\end{split}
\]
We apply this notation when $S\subseteq \Gamma\subseteq \calW_{\Gamma}$ is a subgraph of $\Gamma$ or when $S = \{\uu\}\subseteq \calW_{\Gamma}$ is a singleton. In the latter case we simply write $\calW(\uu)$ respectively $\calW'(\uu)$ for $\calW(\{\uu\})$ respectively $\calW'(\{\uu\})$.

\subsection{Graph products}\label{prelim:graph-products}
We introduce the notion of operator algebraic graph products as in \cite{caspersGraphProductsOperator2017a}, \cite{mlotkowskiLfreeProbability2004},  where we mainly follow the first reference.

Given a simple graph $\Gamma$ and for each $v\in \Gamma$ given a unital C$^\ast$-algebra $A_v$ with a GNS-faithful state $\varphi_v$. Let $\calH_{v}= L^2(A_v,\varphi_v)$ and let $\xi_v\in \calH_v$ be the cyclic vector s.t. $\varphi_v(x) = \langle x\xi_v,\xi_v\rangle$. We denote $\mathring{\calH}_v = (\mathbb{C}\xi_v)^{\perp}$ and $\mathring{A}_{v} = \ker \varphi_v$. For $\vv\in \calW_{\Gamma}\setminus \{e\}$ we fix a reduced representative $(v_1,\ldots, v_n)$ which we call the minimal representative. Then we define the Hilbert space $\mathring{\calH}_{\vv} := \mathring{\calH}_{v_1}\otimes \cdots \otimes \mathring{\calH}_{v_n}$ and furthermore we put $\mathring{\calH}_{e} := \CC\Omega$. 
\begin{remark}
We recall the following notational convention, that appeared first in \cite[top of page 8]{caspersGraphProductKhintchine2021}. 
In case $(v_1', \ldots, v_n')$ is another reduced representative for $\vv$ there is a unique permutation $\sigma$ of the $\{1, \ldots, n\}$ such that $v_{\sigma(i)} = v_i'$ and if $i < j$ are such that $v_i = v_j$ then $\sigma(i) < \sigma(j)$. Thus there exists a unique unitary map 
\[
U_\sigma: \mathring{\calH}_{\vv} := \mathring{\calH}_{v_1}\otimes \cdots \otimes \mathring{\calH}_{v_n} \rightarrow  \mathring{\calH}_{\vv}' := 
\mathring{\calH}_{v_1'}\otimes \cdots \otimes \mathring{\calH}_{v_n'},
\]
mapping $\xi_1 \otimes \ldots \otimes \xi_n$ to $\xi_{\sigma(1)} \otimes \ldots \otimes \xi_{\sigma(n)}$. We will from this point identify the spaces $\mathring{\calH}_{\vv}$ and $\mathring{\calH}_{\vv}'$ through $U_\sigma$ and omit $U_\sigma$ in the notation.  We say that vectors are identified this way through shuffle equivalence. 
\end{remark}
We denote 
$$\calH_{\Gamma} := \bigoplus_{\vv \in \calW_{\Gamma}}\mathring{\calH}_{\vv},$$
where we let $\Omega = 1$ in  $\mathring{\calH}_{e}$ and $\mathring{\calH}_{e} = \mathbb{C}$. $\calH_{\Gamma}$  
  is the graph product of the vertex Hilbert spaces $(\calH_{v},\xi_v)$. 
Furthermore, for $S\subseteq \calW_{\Gamma}$ we write
\begin{align*}
    \calH(S):= \bigoplus_{\vv\in \calW(S)} \mathring{\calH}_{\vv}& &and& &
    \calH'(S):= \bigoplus_{\vv\in \calW'(S)} \mathring{\calH}_{\vv}.
\end{align*}
Again, we apply this notation mainly when $S\subseteq \Gamma\subseteq \calW_{\Gamma}$ is a subgraph of $\Gamma$ or when $S = \{\uu\}\subseteq \calW_{\Gamma}$ is a singleton and in the latter case we simply write $\calH(\uu)$ respectively $\calH'(\uu)$ for $\calH(\{\uu\})$ respectively $\calH'(\{\uu\})$.
For $\vv_1,\ldots, \vv_n \in \calW_{\Gamma}$ with $\vv = \vv_1\cdots \vv_n$ reduced, we let $$\calQ_{(\vv_1,\ldots,\vv_n)}:\mathring{\calH}_{\vv_1}\otimes \cdots \otimes \mathring{\calH}_{\vv_n}\to \mathring{\calH}_{\vv}$$ 
be the natural unitary from \cite{borstCCAPGraphProducts2024}. Observe that simply $\calQ_{(\vv_1,\ldots, \vv_n)}(\eta_1\otimes \cdots \otimes \eta_n) \simeq \eta_1\otimes \cdots \otimes \eta_n$ up to shuffle equivalence when $\vv_i\not=e$ for all $1\leq i\leq n$.

For a subgraph $\Lambda\subseteq \Gamma$ we define two unitaries:
\begin{align*}
	U_{\Lambda}: \calH_{\Lambda} \otimes \calH(\Lambda)\to \calH_{\Gamma} & \quad \textrm{as} &
	U_{\Lambda}|_{\mathring{\calH}_{\uu}\otimes \mathring{\calH}_{\ww}} &= \calQ_{(\uu,\ww)} & \text{ for } \uu\in \calW_{\Lambda}, \ww\in \calW(\Lambda),\\
	U_{\Lambda}' : \calH'(\Lambda)\otimes \calH_{\Lambda}\to \calH_{\Gamma} & \quad \textrm{as} &
	U_{\Lambda}|_{\mathring{\calH}_{\uu}\otimes \mathring{\calH}_{\ww}} &= \calQ_{(\uu,\ww)} & \text{ for } \uu\in \calW'(\Lambda), \ww\in \calW_{\Lambda},
\end{align*}
and for $u\in \Gamma$ simply write $U_{u}$ respectively $U_{u}'$ instead of $U_{\{u\}}$ respectively $U_{\{u\}}'$.

As in \cite{caspersGraphProductsOperator2017a} we define the embeddings
\begin{align*}
    \lambda_v: \bound(\calH_{v}) \to  \bound(\calH_{\Gamma})  & \quad \textrm{ as } & \lambda_v(a) &= U_{v}(a\otimes 1)U_v^*,\\
     \rho_v: \bound(\calH_{v})\to \bound(\calH_{\Gamma}) & \quad  \textrm{ as } & \rho_v(a) &= U_{v}'(1\otimes a)(U_v')^*.
\end{align*}
As in \cite{borstCCAPGraphProducts2024}, for a word $\vv\in \calW_{\Gamma}\setminus \{e\}$ with minimal representative $(v_1,\ldots v_n)$ we denote  
$$\mathring{\boldA}_{\vv} := \mathring{A}_{v_1}\otimes \cdots \otimes \mathring{A}_{v_n}$$
for the algebraic tensor product and put $\mathring{\boldA}_{e} = \CC1$.  We denote $\boldA_{\Gamma} := \bigoplus_{\vv\in \calW_{\Gamma}}\mathring{\boldA}_{\vv}$ for the algebraic direct sum.
We let $\lambda:\boldA_{\Gamma}\to \bound(\calH_{\Gamma})$ be the linear embedding, from \cite[Equation (18)]{borstCCAPGraphProducts2024}, which is defined as 
\[
a_{1} \otimes \ldots \otimes a_n \mapsto \lambda_{v_1}(a_1) \ldots \lambda_{v_n}(a_n),
\]
where $a_i \in \mathring{A}_{v_i}$. We also put $\lambda(1) = \Id_{\calH_{\Gamma}}$. We call an operator of the form $a=\lambda_{v_1}(a_1)\cdots \lambda_{v_n}(a_n)$ with $\vv=v_1\cdots v_n$ reduced and $a_i\in \mathring{A}_{v_i}$ for $1\leq i\leq n$ a \textit{reduced operator}. Oftentimes, we leave out the embeddings $\lambda_{v_i}$ and simply write $a = a_1\cdots a_n$.\\

\noindent \textit{Graph products of unital $\Cstar$-algebras.} We denote $A_{\Gamma}:=\astred_{v,\Gamma}(A_v,\varphi_v)$ for the norm closure of $\lambda(\boldA_{\Gamma})$ which we call the reduced graph product of the C$^\ast$-algebras $A_v$ with respect to states $\varphi_v$.
Then $\varphi_\Gamma(x) = \langle x \Omega, \Omega \rangle$  is the graph product state which restricts to $\varphi_v \circ \lambda_v^{-1}$ on $\lambda_v(A_v)$. 
The vertex C*-algebras $A_v$ are included in $A_\Gamma$ through $\lambda_v$ and we simply identify  $A_v$ as subalgebras of $A_\Gamma$. By the universal property \cite[Proposition 3.12, Proposition 3.22]{caspersGraphProductsOperator2017a} these inclusions extend to an inclusion of $A_\Lambda \subseteq A_\Gamma$, for $\Lambda\subseteq\Gamma$.   This inclusion admits a unique $\varphi_\Gamma$-preserving conditional expectation   $\mathbb{E}_{A_\Lambda}: A_\Gamma \rightarrow A_\Lambda$ that is determined by the following formula, where $ a_1 \ldots a_n$ is a reduced operator with $a_i \in \mathring{A}_{v_i}$, 
\[
 \mathbb{E}_{A_\Lambda}( a_1 \ldots a_n) = \left\{
\begin{array}{cc}
a_1 \ldots a_n, & \forall i, v_i \in \Lambda; \\
0, & \textrm{ otherwise.}
\end{array}
 \right. 
\]  

\noindent \textit{Graph products of von Neumann algebras.}
In case  $A_v, v \in \Gamma$ is a von Neumannn algebra, usually denoted by $M_v$,  equipped with normal faithful state $\varphi_v$, the graph product von Neumann algebra $M_\Gamma = *_{v,\Gamma}(M_v,\varphi_v)$ is constructed as the strong operator topology closure of the reduced C$^\ast$-algebraic graph product constructed above.  Moreover, we define $\mathring{M}_{\vv}$ as the closure of $\lambda(\mathring{\boldM}_{\vv})$ in the strong operator topology. We also define the graph product state $\varphi_\Gamma(x) = \langle x \Omega, \Omega \rangle$  which restricts to $\varphi_v \circ \lambda_v^{-1}$ on $\lambda_v(M_v)$. The conditional expectation $\mathbb{E}_{A_\Lambda}$ extends to a normal conditional expectation $\mathbb{E}_{M_\Lambda}: M_\Gamma \rightarrow M_\Lambda$.   \\

\noindent \textit{Amalgamated free product decomposition.}
Let $v \in \Gamma$ and set $\Gamma_1 = \Star(v), \Lambda = \Link(v), \Gamma_2 = \Gamma \backslash \{v\}$ and assume $\Gamma_2$ is not empty. Recall that we have natural inclusions $A_{\Lambda} \subseteq A_{\Gamma_1}$ and  $A_{\Lambda} \subseteq A_{\Gamma_2}$ with conditional expectations.  Also for the von Neumann algebras $M_{\Lambda} \subseteq M_{\Gamma_1}$ and  $M_{\Lambda} \subseteq M_{\Gamma_2}$ with conditional expectations. 
Then we have a reduced amalgamated free product decomposition  \cite{caspersGraphProductsOperator2017a}
\begin{equation}\label{Eqn=Amalgam}
A_\Gamma = A_{\Gamma_1} \ast_{A_{\Lambda}} A_{\Gamma_2}, \qquad 
M_\Gamma = M_{\Gamma_1} \ast_{M_{\Lambda}} M_{\Gamma_2}, 
\end{equation}
for both the reduced C$^\ast$-algebraic and von Neumann algebraic graph products.

\subsection{Creation and annihilation operators}
Let $\Gamma$ be a simple graph as before and for each $v \in \Gamma$ let $M_v$ be a von Neumann algebra with faithful normal state $\varphi_v$. 
Let $M_\Gamma$ be the graph product von Neumann algebra.  
As in \cite[Equation (27)]{borstCCAPGraphProducts2024}, for $\ww\in \calW_{\Gamma}$ we denote
\begin{align}
    \calS_{\ww} := \biggl\{(\ww_1,\ww_2,\ww_3)\in \calW_{\Gamma}^3: \begin{array}{l}\ww = \ww_1\ww_2\ww_3 \text{ reduced}, \\ \ww_2 \text{ is a clique word.}\end{array}\biggr\}
\end{align}
and $\calS_{\Gamma} = \bigcup_{\ww\in \calW_{\Gamma}}\calS_{\ww}$. For $v\in \Gamma$ let $P_v\in \bound(\calH_{\Gamma})$ be the projection onto $\calH(v)^\perp$. For $(\ww_1,\ww_2,\ww_3)\in \calS_{\Gamma}$ we let $\lambda_{(\ww_1,\ww_2,\ww_3)}: \boldM_{\Gamma}\to \bound(\calH_{\Gamma})$ be the linear map defined in \cite[Equation (26)]{borstCCAPGraphProducts2024}. This map satisfies $\lambda_{(\ww_1,\ww_2,\ww_3)}(a)=0$ for $a\in \mathring{\boldM}_{\vv}$ with $\vv\not=\ww$. Furthermore, for $a\in \mathring{\boldM}_{\ww}$ it is given as follows. Write $\ww_i = w_{i,1}\cdots w_{i,n_i}$ for $i=1,2,3$, and let $a_i := a_{i,1}\otimes \cdots \otimes a_{i,n_i}\in \mathring{\boldM}_{\ww_i}$ for $i=1,2,3$ be such that $\lambda(a)=\lambda(a_1)\lambda(a_2)\lambda(a_3)$. Then 
\[
\begin{split}
    \lambda_{(\ww_1,\ww_2,\ww_3)}(a) = &(P_{w_{1,1}}a_{1,1}P_{w_{1,1}}^{\perp})\cdots (P_{w_{1,n_1}}a_{1,n_1}P_{w_{1,n_1}}^{\perp}) \cdot\\ 
    &(P_{w_{2,1}}a_{2,1}P_{w_{2,1}})\cdots (P_{w_{2,n_1}}a_{2,n_2}P_{w_{2,n_2}})\cdot \\
&(P_{w_{3,1}}^{\perp}a_{3,1}P_{w_{3,1}})\cdots (P_{w_{3,n_1}}^{\perp}a_{3,n_3}P_{w_{3,n_3}}).
\end{split}
\]
Intuitively, for $a\in \boldM_{\Gamma}$ the operator $\lambda_{(\ww_1,\ww_2,\ww_3)}(a)$ is the part of $\lambda(a)$ that acts as annihilation on a word of type $\ww_3^{-1}$, acts diagonally on a word of type $\ww_2$ and acts as creation on a word of type $\ww_1$. We will use the following two result. \cref{prelim:lemma-actions-graph-products} was originally proven in \cite{caspersGraphProductKhintchine2021}.

\begin{lemma}[Lemma 2.2. in \cite{borstCCAPGraphProducts2024}]\label{prelim:lemma-actions-graph-products}
Let $\ww\in \calW_{\Gamma}$ and $a\in \mathring{M}_{\ww}$. We have
    \begin{align*}
        \lambda = \sum_{(\ww_1,\ww_2,\ww_3) \in \calS_{\Gamma}} \lambda_{(\ww_1,\ww_2,\ww_3)} &\quad\text{ and }\quad 
        \lambda(a) = \sum_{(\ww_1,\ww_2,\ww_3) \in \calS_{\ww}} \lambda_{(\ww_1,\ww_2,\ww_3)}(a).
    \end{align*}
\end{lemma}
\begin{lemma}\label{lemma:word_action_non_zero}
    Let $\ww,\vv\in \calW_{\Gamma}$, let $(\ww_1,\ww_2,\ww_3)\in \calS_{\ww}$  and let $a\in \mathring{M}_{\ww}$ and $\eta\in \mathring{\calH}_{\vv}$. Then  
    $\lambda_{(\ww_1,\ww_2,\ww_3)}(a)\eta \in \mathring{\calH}_{\uu}$
    where $\uu = \ww_1\ww_3\vv$. Moreover, if $\lambda_{(\ww_1,\ww_2,\ww_3)}(a)\eta$ is non-zero, then $\vv$ and $\uu$ start with $\ww_3^{-1}\ww_2$ and $\ww_1\ww_2$ respectively.
    \begin{proof}
        Let $\ww,\vv, (\ww_1,\ww_2,\ww_3), a$ and $\eta$ be a stated.  Suppose $\lambda_{(\ww_1,\ww_2,\ww_3)}(a)\eta$ is non-zero. We may assume that $a$ is of the form $a = a_1a_2a_3$  with $a_i \in \mathring{M}_{\ww_i}$.
        We use the comments stated after \cite[Equation (25)]{borstCCAPGraphProducts2024}. These imply that 
        $\lambda_{(e,e,\ww_3)}(a_3)\eta\in \mathring{\calH}_{\ww_3\vv}$ and that $\vv$ starts with $\ww_3^{-1}$.
        Moreover the comments then imply that $\lambda_{(e,\ww_2,\ww_3)}(a_2a_3)\eta = \lambda_{(e,\ww_2,e)}(a_2)\lambda_{(e,e,\ww_3)}(a_3)\eta\in \mathring{\calH}_{\ww_3\vv}$ and that $\ww_3\vv$ starts with $\ww_2$. This already shows that $\vv$ starts with $\ww_3^{-1}\ww_2$ (observe that $\ww_2 = \ww_2^{-1}$). Last, the comments imply that $\lambda_{(\ww_1,\ww_2,\ww_3)}(a) = \lambda_{(\ww_1,e,e)}(a_1)\lambda_{(e,\ww_2,\ww_3)}(a_2a_3)\eta\in \mathring{\calH}_{\ww_1\ww_3\vv}$ and that $\ww_1\ww_3\vv$ starts with $\ww_1$. Hence $\uu = \ww_1\ww_3\vv$ starts with $\ww_1\ww_2$.
    \end{proof}
\end{lemma}

\section{Rigid graphs}\label{Sect=RigidGraph} 
In this section we introduce the notion of rigid graphs. 

\begin{definition}[Rigid graphs] \label{def:rigid-graphs}
We say that a simple graph $\Gamma$ is \textit{rigid} if for every $v \in \Gamma$ we have $\Link_{\Gamma}(\Link_{\Gamma}(v)) = \{v\}$. When $|\Gamma|\geq 2$ this means in particular for each $v \in \Gamma$ that $\Link_{\Gamma}(v)$ is not empty. 
\end{definition}

\begin{example}\label{example:rigid-graphs}
We give some examples of rigid graphs which are easy to check:
    \begin{enumerate}
 \item By the convention $\Link_{\Gamma}(\varnothing) = \Gamma$ it follows that if $\vert \Gamma \vert = 1$ then $\Gamma$ is rigid. 

        \item Any complete graph is rigid.
        \item\label{Item=rigid-graphsZn} For $n\geq 2$ let $\ZZ_n = \{1,\ldots, n\}$ be the cyclic graph of length $n$, i.e. $i,j$ share an edge if and only if $|i-j| =1$ or $\{i,j\} = \{1,n\}$. Then for $n\geq 5$ the graph $\ZZ_n$ is rigid. Note also that $\ZZ_2$ and $\ZZ_3$ are rigid, but $\ZZ_4$ is not.
        \item Consider $\ZZ$ as the infinite cyclic graph, i.e. $i,j$ share an edge in $\ZZ$ if and only if $|i-j|=1$. Then $\ZZ$ is rigid.
    \end{enumerate}
\end{example}

We will now define the notion of graph products of graphs, and construct a large variety of rigid graphs.
\begin{definition}\label{def:graph-product-graph}
    Let $\Gamma$ be a simple graph and for each $v\in \Gamma$ let $\Lambda_{v}$ be a simple graph. We denote $\Lambda_{\Gamma} := *_{v,\Gamma}\Lambda_v$ for the graph product of the graphs $\{\Lambda_{v}\}_{v\in \Gamma}$. This is defined as the graph with vertices set 
    \begin{align}
    \{(v,s): v\in \Gamma, s\in \Lambda_{v}\},
    \end{align}
    where vertices $(v,s)$ and $(w,t)$ share an edge in $\Lambda_{\Gamma}$ if either $v=w$ and $t,s$ share an edge in $\Lambda_{v}$ or $v\not=w$ and $v,w$ share an edge in $\Gamma$.
\end{definition}
We observe that $\Lambda_{\Gamma}$ contains the graphs $\Lambda_v$ for $v\in \Gamma$ as (mutually disjoint) subgraphs. Furthermore, we observe that if we take $|\Lambda_v|=1$ for each $v\in \Gamma$ then $\Lambda_{\Gamma} = \Gamma$.

\begin{remark}\label{remark:graph-product-of-graphs-consistency}
    For a simple graph $\Gamma$ and graphs $\{\Lambda_{v}\}_{v\in \Gamma}$, and groups $G_{w}$ and von Neumann algebras $(N_w,\varphi_w)$ with normal GNS-faithful state, with  $w\in \Lambda_v$, we have
\begin{align}
    *_{w,\Lambda_{\Gamma}}G_w &= *_{v,\Gamma}(*_{w,\Lambda_v} G_{w})\\
    *_{w,\Lambda_{\Gamma}}(N_w,\varphi_w),
 &= *_{v,\Gamma}(*_{w,\Lambda_v} (N_w,\varphi_w)).
\end{align}
Indeed, this follows by the defining universal property of graph products of groups as well as its analogue for operator algebras that can be found in \cite[Proposition 3.22]{caspersGraphProductsOperator2017a}.
\end{remark}

\begin{lemma}\label{lemma:graph-product-of-rigid-graphs}
    Let $\Gamma$ be a simple graph and for each $v\in \Gamma$ let $\Lambda_{v}$ be a non-empty graph. 
    Then the graph product graph $\Lambda_{\Gamma}$ is rigid if and only if for each vertex $v\in \Gamma$ the graph $\Lambda_v$ is rigid and the vertex $v$ satisfies at least one of the following conditions:
    \begin{enumerate}
        \item $\Link_{\Gamma}(\Link_{\Gamma}(v)) = \{v\};$ \label{item:rigid-graph-products:rigidity-at-vertex}
        \item $|\Lambda_{v}|\geq 2$.\label{item:rigid-graph-products:size-vertex-graph}
    \end{enumerate} 
    \begin{proof} We may assume $\Gamma$ is non-empty.
     First, suppose the conditions in the lemma are satisfied. We show $\Lambda_{\Gamma}$ is rigid.
        Let $(v,j) \in \Lambda_{\Gamma}$ for some $v\in \Gamma$, $j\in \Lambda_v$. Let $(z,k)\in \Link_{\Lambda_{\Gamma}}(\Link_{\Lambda_{\Gamma}}(v,j))$. We need 
 to show that $(z,k) = (v,j)$. 
 
 Suppose first that $|\Lambda_v|\geq 2$. Then, as $\Lambda_v$ is rigid, we have that $\Link_{\Lambda_v}(j)$ is non-empty. Let $l\in \Link_{\Lambda_v}(j)$. Then $(v,l)\in \Link_{\Lambda_{\Gamma}}(v,j)$ and similarly $(z,k)\in \Link_{\Lambda_{\Gamma}}(v,l)$. 
        If $z\not=v$ then by the definition of the graph product graph this implies $z\in \Link_{\Gamma}(v)$. But then, again by the definition of the graph product graph, we obtain $(z,k)\in \Link_{\Lambda_{\Gamma}}(v,j)$. However, as $(z,k)\not\in \Link_{\Lambda_{\Gamma}}(z,k)$, this contradicts that 
        $(z,k)\in \Link_{\Lambda_{\Gamma}}(\Link_{\Lambda_{\Gamma}}(v,j))$. We conclude that $z=v$. Hence, since $(z,k)\in \Link_{\Lambda_{\Gamma}}(v,l)$ we obtain that $k\in \Link_{\Lambda_v}(l)$. Since this holds true for all $l\in  \Link_{\Lambda_v}(j)$, we obtain that $k\in \Link_{\Lambda_v}(\Link_{\Lambda_v}(j))$, so that $k=j$ by rigidity of $\Lambda_v$. Thus $(z,k) = (v,j)$.

        Now suppose that $|\Lambda_v|<2$, i.e. $\Lambda_v = \{j\}$, and just assume that $\Link_{\Gamma}(\Link_{\Gamma}(v)) = \{v\}$. If $|\Gamma|=1$ then $\Lambda_{\Gamma} = \Lambda_v$ is rigid. Thus we can assume $|\Gamma|\geq 2$. Then $\Link_{\Gamma}(v)$ must be non-empty since $\Link(\varnothing)=\Gamma\not=\{v\}$.
        Take $w\in \Link_{\Gamma}(v)$. Then, as by assumption $\Lambda_w$ is non-empty, we can pick $i\in \Lambda_w$. Now $(w,i)\in \Link_{\Lambda_{\Gamma}}(v,j)$, by the definition of the graph $\Lambda_{\Gamma}$. Thus $(z,k)\in \Link_{\Lambda_{\Gamma}}(w,i)$. If $w=z$ then $z\in \Link_{\Gamma}(v)$ and so also $(z,k)\in \Link_{\Lambda_{\Gamma}}(v,j)$. But as $(z,k)\not\in \Link_{\Lambda_{\Gamma}}(z,k)$, this contradicts that $(z,k)\in \Link_{\Lambda_{\Gamma}}(\Link_{\Lambda_{\Gamma}}(v,j))$. Thus $w\not=z$, and therefore, as  $(z,k) \in \Link_{\Lambda_{\Gamma}}(w,i)$, we obtain that $z\in \Link_{\Gamma}(w)$. Therefore, since this holds for any $w\in \Link_{\Gamma}(v)$, we obtain that $z\in \Link_{\Gamma}(\Link_{\Gamma}(v)) = \{v\}$ and thus $z=v$. Thus as $k\in \Lambda_{z} = \Lambda_v = \{j\}$, we obtain $(z,k) = (v,j)$. 

        We now prove the reverse direction. First, suppose there is a vertex $v\in \Gamma$ such that $\Lambda_{v}$ is not rigid. Take $j\in \Lambda_{v}$ such that $\Link_{\Lambda_{v}}(\Link_{\Lambda_v}(j))\not=\{j\}$ so that we can choose $k\in \Link_{\Lambda_{v}}(\Link_{\Lambda_v}(j))$ with $k\not=j$. Now, one can check that $(v,k)\in \Link_{\Lambda_{\Gamma}}(\Link_{\Lambda_{\Gamma}}(v,j))$, hence $\Lambda_{\Gamma}$ is not rigid.

        Now suppose there is vertex $v\in \Gamma$ such that $\Link_{\Gamma}(\Link_{\Gamma}(v))\not= \{v\}$ and $|\Lambda_v|=1$, i.e. $\Lambda_v=\{j\}$ for some $j$.  Then $\Link_{\Lambda_{\Gamma}}(v,j) = \bigcup_{w\in \Link_{\Gamma}(v)}\{(w,i): i \in \Lambda_{w}\}$.
        We can choose a $z\in \Link_{\Gamma}(\Link_{\Gamma}(v))$ with $z\not=v$ and let $k\in \Lambda_z$. Then we see that $(z,k)\in \Link_{\Lambda_{\Gamma}}(\Link_{\Lambda_{\Gamma}}(v,j))$, which shows $\Lambda_{\Gamma}$ is not rigid.
    \end{proof}
\end{lemma}
By the result of \cref{lemma:graph-product-of-rigid-graphs}, it is possible to construct many different rigid graphs using the rigid graphs from \cref{example:rigid-graphs}. 

\begin{remark}\label{remark:rigid-components}
    Let $\Gamma$ be a rigid graph. Then any connected component of $\Gamma$ is rigid and any irreducible component of $\Gamma$ is rigid. Indeed, if $\Lambda_1,\ldots, \Lambda_n$ are the irreducible components of $\Gamma$ and we let $\Pi = \{1,\ldots, n\}$ be a complete graph, then $\Gamma = *_{v,\Pi}\Lambda_v = \Lambda_{\Pi}$. Hence, by \cref{lemma:graph-product-of-rigid-graphs} and rigidity of $\Gamma$ we obtain that the graphs $\Lambda_1,\ldots, \Lambda_n$ are rigid. 
    Similarly, if we let $\Lambda_1',\ldots, \Lambda_m'$ be connected components of $\Gamma$ and we let $\Pi' =\{1,\ldots,m\}$ be a graph with no edges, then $\Gamma = *_{v,\Pi'}\Lambda_v' = \Lambda_{\Pi'}'$ so that by \cref{lemma:graph-product-of-rigid-graphs}  and rigidity of $\Gamma$ we obtain that $\Lambda_1',\ldots, \Lambda_m'$ are rigid.
\end{remark}

We now define the core of a graph.
 \begin{definition}[Core of a graph]\label{Dfn=DoubleGraph}
  Let $\Gamma$ be a simple graph. We say that two vertices $v,w \in \Gamma$ are {\it core equivalent}, with notation $v \sim w$,  if $\Star(v)   = \Star(w)$.  Let $\overline{v}$ be the core equivalence class of $v \in \Gamma$.  We define the {\it core} of $\Gamma$, with notation $\mathcal{C}\Gamma$, as the graph whose vertices set is the set of all core equivalence classes of $\Gamma$. The edges set of $\mathcal{C}\Gamma$ is defined by declaring that  $\overline{v}, \overline{w} \in \mathcal{C}\Gamma$ with $\overline{v} \not = \overline{w}$ share an edge in $\mathcal{C}\Gamma$ if and only if $v,w$ share an edge in $\Gamma$.
\end{definition}
We remark that $\calC\calC\Gamma = \calC\Gamma$, that is, the core of the core of a graph is equal to the core of the graph. In the following lemma we show that any graph can be written as a graph product over its core.
\begin{lemma}\label{Lem=CoreReduction}
    Let $\Gamma$ be a simple graph. For $\overline{v}\in \calC\Gamma$ let $\Lambda_{\overline{v}}$ be the complete graph of size $|\Lambda_{\overline{v}}| = |\overline{v}|$. Then $\Gamma \simeq \Lambda_{\calC \Gamma}$. Furthermore, if $\calC\Gamma$ is rigid, then so is $\Gamma$.
    \begin{proof}
        Indeed, as for $\overline{v}\in \calC\Gamma$ we have $|\overline{v}| = |\Lambda_{\overline{v}}|$, we can build a bijection $\iota_{\overline{v}}: \overline{v}\to \Lambda_{\overline{v}}$. We then define the bijection
        $\iota: \Gamma\to \Lambda_{\calC\Gamma}$ as $\iota(v) = (\overline{v},\iota_{\overline{v}}(v))$. We show this is a graph isomorphism. Let $v\not=w\in \Gamma$. If $v,w$ do not share an edge in $\Gamma$ then $\overline{v}\not= \overline{w}$ and  $\overline{v}, \overline{w}$ do not share an edge in $\calC\Gamma$. Hence $(\overline{v},\iota_{\overline{v}}(v))$ and $(\overline{w},\iota_{\overline{w}}(w))$ do not share an edge in $\Lambda_{\calC\Gamma}$. 
        Now suppose $v,w$ do share an edge in $\Gamma$. If $\overline{v} = \overline{w}$ then since $\Lambda_{\overline{v}} = \Lambda_{\overline{w}}$ is complete we obtain that $(\overline{v},\iota_{\overline{v}}(v))$ and $(\overline{w},\iota_{\overline{w}}(w))$ share an edge in $\Lambda_{\calC\Gamma}$ . On the other hand, if $\overline{v}\not= \overline{w}$, then $\overline{v}, \overline{w}$ share an edge in $\calC\Gamma$ so that also $(\overline{v},\iota_{\overline{v}}(v))$ and $(\overline{w},\iota_{\overline{w}}(w))$ share an edge in $\Lambda_{\calC\Gamma}$. This shows that $\iota$ is an isomorphism and hence $\Gamma\simeq \Lambda_{\calC\Gamma}$.
        
        We prove the last statement. Suppose $\calC\Gamma$ is rigid.  Since for each $\overline{v}\in \calC\Gamma$ the graph $\Lambda_{\overline{v}}$ is rigid (since it is complete) and since by rigidity of $\calC\Gamma$ we have $\Link_{\calC\Gamma}(\Link_{\calC\Gamma}(\overline{v})) = \overline{v}$, we obtain by \cref{lemma:graph-product-of-rigid-graphs} that $\Lambda_{\calC\Gamma}$ is rigid. Thus $\Gamma \simeq \Lambda_{\calC\Gamma}$ is rigid.
    \end{proof}
\end{lemma}
We make two remarks on \cref{Lem=CoreReduction}
\begin{remark}
We remark that if a simple graph $\Gamma$ is rigid, then its core is, in general, not rigid. Indeed, let $\Pi = \{v,w\}$ denote the simple graph of size $2$ with no edges and let $\Lambda_v, \Lambda_w$ denote complete graphs of size $|\Lambda_v|,|\Lambda_w|\geq 2$. Then the graph $\Gamma:= \Lambda_{\Pi}$ is rigid by \cref{lemma:graph-product-of-rigid-graphs} but $\calC\Gamma = \Pi$ is not rigid.    
\end{remark}

\begin{remark}\label{remark:class-CC1}
If a graph $\Gamma$ is in the class CC$_1$ as described in \cite{chifanRigidityNeumannAlgebras2022} then $\Gamma$ is rigid. Indeed if $\Gamma$ is  CC$_1$ then its core $\mathcal{C}\Gamma$, which is in fact also CC$_1$,  is given by the graph of \cite[Eqn. (1.1)]{chifanRigidityNeumannAlgebras2022}. This graph is rigid as can be checked directly from the very definition of rigidity.  We can then apply  \cref{Lem=CoreReduction} to obtain that $\Gamma$ is rigid. It thus follows that the graphs considered in the current paper form a much  richer class   than \cite{chifanRigidityNeumannAlgebras2022}.
\end{remark}

\section{Graph products of nuclear C$^\ast$-algebras}\label{Sect=NuclearGraphProduct}
The aim of this section is to give a sufficient condition for when the reduced graph product of nuclear C$^\ast$-algebras is nuclear again. Such a result cannot hold in full generality as it is clear from the fact that the free product of amenable discrete groups is non-amenable as soon as one group has at least 2 elements and the other group has at least 3 elements.  Hence the stability result in this section   requires particular conditions on the states with respect to which we take the graph product. Such a result was obtained by Ozawa in \cite{OzawaNuclearFree} for amalgamated free products and we use the amalgamated free product decomposition of graph products \eqref{Eqn=Amalgam}  to show that the same holds for graph products. 

\vspace{0.3cm}

Let $\Gamma$ be a simple graph. 
Let $(A_v, \varphi_v)$ with $v \in \Gamma$ be unital C$^\ast$-algebras $A_v$,  GNS-faithful states $\varphi_v$ and GNS-representation $\pi_v$ of $A_v$ on the Hilbert space $\calH_v = L^2(A_v, \varphi_v)$.

For Hilbert C$^\ast$-modules we refer to \cite{lanceHilbertModulesToolkit1995}.  
Consider the reduced graph product C$^\ast$-algebras $(A_\Lambda, \varphi_\Lambda)$ for any $\Lambda \subseteq \Gamma$ which is a subalgebra of $(A_\Gamma, \varphi_\Gamma)$ with conditional expectation $\mathbb{E}_\Lambda$.

\begin{definition}
We construct a Hilbert C$^\ast$-module  $\calH_{\mathbb{E}_\Lambda}$ as the completion of $A_\Gamma$ with respect to the $A_\Lambda$-valued inner product
\[
\langle a, b \rangle_{\mathbb{E}_\Lambda}  =  \mathbb{E}_\Lambda(b^\ast a)
\]
and the corresponding Hilbert $A_\Lambda$-module norm $\Vert a \Vert = \Vert \langle a, a \rangle^{\frac{1}{2}} \Vert$.    Let $\pi_{\mathbb{E}_\Lambda}: A_\Gamma \rightarrow \bound(\calH_{\mathbb{E}_\Lambda})$ be the GNS-representation of $A_\Gamma$ on the Hilbert C$^\ast$-module  $\calH_{\mathbb{E}_\Lambda}$ by adjointable operators. Then $\pi_{\mathbb{E}_\Lambda}$  is given by extending left multiplication 
\[
\pi_{\mathbb{E}_\Lambda}(x) a = xa, x \in A_\Gamma, a \in A_\Gamma \subseteq \calH_{\mathbb{E}_\Lambda}
\]
and we shall omit $\pi_{\mathbb{E}_\Lambda}$  in the notation if the module action is clear.
\end{definition}

\begin{definition}
An operator on the Hilbert $A_\Lambda$-module $\calH_{\mathbb{E}_\Lambda}$ is called {\it finite rank} if it is  in the linear span of operators of the form 
\[
\theta_{\eta_2, \eta_1}: \xi \mapsto  \eta_2 \langle \xi, \eta_1 \rangle_{\mathbb{E}_\Lambda}, \qquad \eta_i \in \calH_{\mathbb{E}_\Lambda}. 
\]
The closure of the space of all finite rank operators are defined as the space of {\it compact operators} $\compact(\mathcal{H}_{\mathbb{E}_\Lambda})$. 
\end{definition}

\begin{lemma}\label{Lem=ContainsCompacts}
Suppose there exists $v \in \Gamma$ such that $\Gamma = \Star(v)$.  
If $\pi_v(A_v)$ contains $\Kompact(\calH_v)$ then $\pi_{\mathbb{E}_{\Link(v)}}(A_{\Star(v)})$ contains   $\Kompact(\calH_{\mathbb{E}_{\Link(v)}})$. 
\end{lemma}
\begin{proof}
We have that $A_{\Star(v)} = A_v \otimes A_{\Link(v)}$ where the tensor product is the minimal tensor product and under this correspondence we have 
\[
\langle a \otimes b, c \otimes d \rangle_{\mathbb{E}_{\Link(v)}} = 
\varphi_v(c^\ast a) d^\ast b, \qquad a,c \in A_v, b,d \in A_{\Link(v)}.
\]
We thus may identify $\calH_{\mathbb{E}_{\Link(v)}}$ as the closure of the algebraic tensor product $\calH_{v} \otimes A_{\Link(v)}$ with respect to the inner product $\langle \xi \otimes b, \eta \otimes d \rangle = \langle \xi, \eta \rangle d^\ast b$. 
Further, under this correspondence  $\pi_{\mathbb{E}_{\Link(v)}}  = \pi_v \otimes \pi_l$ 
where $\pi_l(x)a = xa, x,a \in A_{\Link(v)}$ is the left multiplication.   Let $p_v$ be the projection of $\calH_v$ onto $\mathbb{C} \xi_v$. Then $p_v \otimes 1$ equals the extension of  $\mathbb{E}_{\Link(v)}$ as a bounded map on $\calH_{\mathbb{E}_{\Link(v)}}$ identified with the closure of  $\calH_{v} \otimes A_{\Link(v)}$.  As by assumption $p_v$ lies in  $\pi_v(A_v)$ it thus follows that $p_v \otimes 1$ lies in   $\pi_{\mathbb{E}_{\Link(v)}}(A_{\Star(v)})$. It thus follows that for  $a,c,x \in A_v, b,d,y \in A_{\Link(v)}$ we have 
\[
\theta_{a \otimes b, c \otimes d}(x \otimes y) = 
\varphi_v(c^\ast x)  a \otimes    b d^\ast y
 = \pi_{\mathbb{E}_{\Link(v)}}(a \otimes b) (p_v \otimes 1) \pi_{\mathbb{E}_{\Link(v)}} (c^\ast \otimes d^\ast) (x \otimes y). 
\]
The right hand side is contained in  $  \pi_{\mathbb{E}_{\Link(v)}}(A_{\Star(v)})$. Hence  $  \pi_{\mathbb{E}_{\Link(v)}}(A_{\Star(v)})$  contains a dense set of finite rank operators and hence must contain all compact operators. 
\end{proof}

\begin{theorem}\label{thm:nuclearity-graph-productNEW}
Let $\Gamma$ be a  simple graph. If for each $v\in\Gamma$, $A_v$ is nuclear and   $\pi_v(A_v)$   contains the    compact operators $\compact(\calH_v)$, then $A_\Gamma$  is nuclear. 
\end{theorem}
\begin{proof}
It suffices to prove the theorem for $\Gamma$ a finite graph as inductive limits of inclusions of nuclear C$^\ast$-algebras are nuclear.

Our proof proceeds by induction to the number of vertices in $\Gamma$. So we assume that for any $\Lambda \subsetneq \Gamma$ we have proved that $A_\Lambda$ is nuclear. We shall prove that $A_\Gamma$ is nuclear.

If $\Gamma$ is complete then $A_\Gamma$ is the minimal tensor product of $A_v, v \in \Gamma$ which is nuclear as each $A_v$ is nuclear. 

Assume $\Gamma$ is not complete. Then we may take $v \in \Gamma$ such that $\Star(v) \not =  \Gamma$. In this case 
\[
A_\Gamma = A_{\Star(v) } \ast_{A_{\Link(v)}} A_{\Gamma \backslash \{v\}},
\]
where all graph products and amalgamated free products are reduced. By induction  $A_{\Star(v) }$ and $A_{\Gamma \backslash \{v\}}$ are nuclear. Further the GNS-representation of $A_{\Star(v) }$ with respect to its conditional expectation onto $A_{\Link(v)}$ contains all compact operators by \cref{Lem=ContainsCompacts}. Hence \cite[Theorem 1.1]{OzawaNuclearFree} concludes that $A_\Gamma$  is nuclear.  
\end{proof}

\section{Relative amenability, quasi-normalizers and embeddings in graph products}\label{Sect=technical-results-graph-products}
In this section we establish the required machinery we need  throughout the paper.
First in \cref{subsection:calculating-conditional-expectation} we discuss how to calculate conditional expectations in graph products. This will be used in \cref{subsection:RelativeAmenability} to prove a result concerning relative amenability in graph products.
The calculations from \cref{subsection:calculating-conditional-expectation} will furthermore be used in \cref{sub:embedding-quasi-normalizers} to keep control of certain quasi-normalizers in graph products. Last, in \cref{subsection:unitary-conjugacy-in-graph-products} we apply results from \cref{sub:embedding-quasi-normalizers} to establish a unitary embedding of certain subalgebras in graph products. 

\subsection{Calculating conditional expectations in graph products}\label{subsection:calculating-conditional-expectation}

For a simple graph $\Gamma$, a graph product $(M_{\Gamma}, \varphi_\Gamma) = *_{v,\Gamma}(M_v,\varphi_v)$ and subgraphs $\Gamma_1,\Gamma_2\subseteq \Gamma$ we discuss how to calculate iterated conditional expectations of the form $\EE_{M_{\Gamma_2}}(a\EE_{M_{\Gamma_1}}(x)b)$ for $a,b,x\in M_{\Gamma}$. Such calculations have been done in \cite{borstClassificationRightangledCoxeter2023} in the setting of right-angled Coxeter groups (i.e. the setting $M_v = \calL(\ZZ/2\ZZ)$ for all $v\in \Gamma$) and in \cite{charlesworthStructureGraphProduct2024} for general graph products. We state \cref{lem:Expectation-for-n=2} which largely follows from \cite[Lemma 3.17]{charlesworthStructureGraphProduct2024}.

In this section we use the following notation. For $\Lambda \subseteq \Gamma$ any subgraph there exists a unique normal $\varphi_{\Gamma}$-perserving conditional expecation of $M_\Gamma$ onto $M_\Lambda$ that we denote by  $\EE_{M_{\Lambda}}$. Then $\EE_{M_{\Lambda}}(x) = 0$ for any reduced operator $x \in M_\Gamma \backslash M_\Lambda$, c.f. \cite[Remark 2.14]{caspersGraphProductsOperator2017a}. 
Recall that for an element $\uu$ in the right angled Coxeter group $\calW_{\Gamma}$ we have defined $\Link(\uu)$ as a subgraph of $\Gamma$ in \eqref{Eqn=LinkForRef}.

\begin{proposition}\label{lem:Expectation-for-n=2}
	Let $\Gamma$ be a simple graph and let $\Gamma_1,\Gamma_2$ be its subgraphs. Let $\uu,\vv\in \calW_{\Gamma}$ and write $\uu=\uu_l\uu_c\uu_r$ and $\vv = \vv_l\vv_c\vv_r$ (both reduced) with $\uu_l,\vv_l\in \calW_{\Gamma_1}$, $\uu_r,\vv_r\in \calW_{\Gamma_2}$ and such that $\uu_c,\vv_c$ do not start with letters from $\Gamma_1$ and do not end with letters from $\Gamma_2$.
	
	For $v\in \Gamma$ let $(M_{v},\varphi_{v})$ be a von Neumann algebra with a normal faithful state.
	Let $a = a_la_ca_r$ and $b = b_lb_cb_r$ where $a_l\in \mathring{M}_{\uu_l}$, $a_c\in \mathring{M}_{\uu_c}$, $a_r\in \mathring{M}_{\uu_r}$ and  $b_l\in \mathring{M}_{\vv_l}$, $b_c\in \mathring{M}_{\vv_c}$, $b_r\in \mathring{M}_{\vv_r}$.
	Then for $x\in M_{\Gamma}$ we have
	\[
		\EE_{M_{\Gamma_2}}(a^*\EE_{M_{\Gamma_1}}(x)b) = \varphi(a_c^*b_c)a_r^*\EE_{M_{\Gamma_1\cap \Gamma_2\cap \Link(\uu_c)}}(a_l^*xb_l)b_r.
	\]
	\begin{proof}
		As $a_l^*,b_l\in M_{\Gamma_1}$ and $a_r^*,b_r\in M_{\Gamma_2}$ we have
		\begin{align}\label{eq:calculate-expectation-eq1}
			\EE_{M_{\Gamma_2}}(a^*\EE_{M_{\Gamma_1}}(x)b) &=a_r^*	\EE_{M_{\Gamma_2}}(a_c^*\EE_{M_{\Gamma_1}}(a_l^*xb_l)b_c)b_r.
		\end{align}
  We will now apply \cite[Lemma 3.17]{charlesworthStructureGraphProduct2024}. In that lemma we take   $V_1 = \Gamma_1, V_2 = \Gamma_2$ and $w = \uu_c$, so that $U = \Gamma_1 \cap \Gamma_2 \cap \Link(\uu_c)$. Further, we take $y = \EE_{M_{\Gamma_1}}(a_l^*xb_l)\in M_{\Gamma_1}$ and further note that $a_{c}\in \mathring{M}_{\uu_c}$ and $b_c\in \mathring{M}_{\vv_c}$.  Applying   \cite[Lemma 3.17]{charlesworthStructureGraphProduct2024} yields
		\begin{equation}
  \begin{split}
			\EE_{M_{\Gamma_2}}(a_c^*\EE_{M_{\Gamma_1}}(a_l^*xb_l)b_c)
			&= \varphi(a_c^*b_c)\EE_{M_{\Gamma_1\cap \Gamma_2\cap \Link(\uu_c)}}(\EE_{M_{\Gamma_1}}(a_l^*xb_l))\\
			&= \varphi(a_c^*b_c)\EE_{M_{\Gamma_1\cap \Gamma_2\cap \Link(\uu_c)}}(a_l^*xb_l).
			\label{eq:calculate-expectation-eq3}
		\end{split}
  \end{equation}
		This proves the statement by combining \eqref{eq:calculate-expectation-eq1} and \eqref{eq:calculate-expectation-eq3}.
		
	\end{proof}	
\end{proposition}

\subsection{Relative amenability in graph products}
\label{subsection:RelativeAmenability}
We state the definition of relative amenability for which we refer to \cite[Proposition 2.4]{popaUniqueCartanDecomposition2014}. \begin{definition}\label{Dfn=RelativeAmenableFirst} Let $(M,\tau)$ be a tracial von Neumann algebra and let $P\subseteq 1_{P}M1_{P}, Q\subseteq M$ be von Neumann subalgebras. We say that $P$ is amenable relative to $Q$ inside $M$ if  there exists a $P$-central positive functional on 
	$1_P \langle M, e_Q \rangle        1_P$ that restricts to the trace       $\tau$ on $1_P M 1_P$.
    \end{definition}
    
    Using the calculations of conditional expectations we will prove \cref{Thm=Square} which asserts that when a von Neumann algebra $P\subseteq M_{\Gamma}$ is amenable relative to $M_{\Gamma_i}$ inside $M_{\Gamma}$ for some subgraphs $\Gamma_i\subseteq \Gamma$ for $i=1,2$, then $P$ is also amenable relative to $M_{\Gamma_1\cap \Gamma_2}$ inside $M_{\Gamma}$. \cref{Thm=Square}
    generalizes \cite[Theorem 3.7]{borstClassificationRightangledCoxeter2023} where the statement was proven in the setting of right-angled Coxeter groups.  The proof of \cref{Thm=Square} is more or less identical to \cite[Theorem 3.7] {borstClassificationRightangledCoxeter2023}.
    In the proof of \cref{Thm=Square} we will therefore refer to the proof and notation from \cite[Section 3, Theorem 3.7]{borstClassificationRightangledCoxeter2023}. We remark that bimodule computations we do in the proof of \cref{Thm=Square} are also related to those done in \cite[Section 5]{charlesworthStructureGraphProduct2024}.  Further note that the proof uses the characterisations of relative amenability in terms of bimodules that were given in \cite[Section 5.2]{popaUniqueCartanDecomposition2014}.  

\begin{theorem}\label{Thm=Square}
	Let $\Gamma$ be a simple graph and let $\Gamma_1, \Gamma_2\subseteq \Gamma$ be subgraphs. For $v\in \Gamma$ let $(M_{v},\tau_{v})$ be a von Neumann algebra with a normal faithful trace.
	Let $P \subset 1_{P}M_\Gamma 1_{P}$ be a von Neumann subalgebra that is amenable relative to $M_{\Gamma_i}$ inside $M_{\Gamma}$ for  $i = 1,2$.   Then $P$ is amenable relative to $M_{\Gamma_1\cap \Gamma_2}$ inside $M_{\Gamma}$.
    \begin{proof}
        By \cite[Remark 2.3]{borstClassificationRightangledCoxeter2023} we may assume that the inclusion $P\subseteq M_{\Gamma}$ is unital.
        As in  \cite[Theorem 3.7]{borstClassificationRightangledCoxeter2023}  let $Q_i := M_{\Gamma_i}$ for $i=1,2$, put $Q_0 = Q_1\cap Q_2$ and put the Connes relative tensor product bimodule
        $$\calH = L^2( \langle M_{\Gamma},e_{Q_1} \rangle )\otimes_{M_{\Gamma}}L^2( \langle M_{\Gamma},e_{Q_2} \rangle).$$ Similar to the proof of  \cite[Theorem 3.7]{borstClassificationRightangledCoxeter2023}  we obtain that the bimodule $_{M_{\Gamma}} L^2(M_{\Gamma})_{P}$ is weakly contained in $_{M_{\Gamma}}\calH_P$. 
        Denote 
        $$\calV = \{\vv\in \calW_{\Gamma}: \vv \text{ does not start with letters from } \Gamma_1 \text{ and does not end with letters from } \Gamma_2\},$$
        and  define the subspace 
        $$\calH_{0} = \Span\{xe_{Q_1}y\otimes_{M_{\Gamma}} e_{Q_2}z : \vv\in \calV, x,z\in M_{\Gamma}, y\in \mathring{M_{\vv}}\},$$ which is dense in $\calH$ (see  \cite[Theorem 3.7]{borstClassificationRightangledCoxeter2023}). We also define a bimodule $\calK := L^2(M_{\Gamma})\otimes_{Q_0} L^2(M_{\Gamma})$.   Now let $x,x',z,z'\in M_{\Gamma}$, $\vv\in \calV$ and $y,y'\in \mathring{M}_{\vv}$. By \cref{lem:Expectation-for-n=2} we have
        \begin{align} 
        \EE_{Q_0}(y^*\EE_{Q_0}(x^*x')y') = \tau(y^*y')\EE_{Q_0\cap \Link(\vv)}(x^*x') = \EE_{Q_1}(y^*\EE_{Q_2}(x^*x')y').
        \end{align}
        Similarly to   \cite[Theorem 3.7]{borstClassificationRightangledCoxeter2023}  we can use this to show that \begin{align}
            \langle x' \otimes_{Q_0} y' \otimes_{Q_0} z', x \otimes_{Q_0} y \otimes_{Q_0} z\rangle = \langle x'e_{Q_1}y' \otimes_{M_{\Gamma}} e_{Q_2}z', x e_{Q_1}y \otimes_{M_{\Gamma}} e_{Q_2}z\rangle.
        \end{align}
     Therefore there exists a well-defined unitary map  $U:\calH_{0}\to L^2(M_{\Gamma})\otimes_{Q_0} \calK$ that is given by 
        $$xe_{Q_1}y\otimes_{M_{\Gamma}}e_{Q_2}z \mapsto x\otimes_{Q_0} y \otimes_{Q_0} z \quad x,z\in M_{\Gamma}, 
        \vv\in \calV, y\in \mathring{M}_{\vv}.$$   
        As in  \cite[Theorem 3.7]{borstClassificationRightangledCoxeter2023} we conclude $P$ is amenable relative to $Q_0= Q_1\cap Q_2$ inside $M_{\Gamma}$.
    \end{proof}
\end{theorem}

\subsection{Embeddings of quasi-normalizers in graph products}\label{sub:embedding-quasi-normalizers}
We prove \cref{prop:embed} and \cref{Lem=Selfembedding} concerning embeddings in graph products. To prove \cref{prop:embed} we need some auxiliary lemmas. We state Lemma 2.1 from \cite{borstClassificationRightangledCoxeter2023}, which was essentially proven in \cite[Remark 3.8]{vaesExplicitComputationsAll2008}. The result is surely known but for completeness we give the proof. 
\begin{lemma}[Lemma 2.1 in \cite{borstClassificationRightangledCoxeter2023}]\label{Lem=TransitiveEmbedding}
	Let $A, \{B_k\}_{k\in I}, Q \subseteq M$ be von Neumann subalgebras with $B_k \subseteq Q$ for every $k\in I$ where $I$ is some index set.  Assume that $A \prec_M Q$ but $A \not \prec_M B_k$ for any $k \in I$. Then there exist projections $p \in A, q \in Q$, a non-zero partial isometry $v \in q M p$ and a normal $\ast$-homomorphism $\theta: pAp \rightarrow q Q q$ such that $\theta(x) v = vx, x \in pAp$ and such that $\theta(pAp) \not \prec_Q B_k$ for any $k\in I$. 
   Moreover, it may be assumed that $p$ is majorized by the support of $\EE_{A}(v^*v)$.
\end{lemma}
\begin{proof}
Let $p \in A, q \in Q$ and $\theta: pAp \rightarrow q Qq$ be a normal $\ast$-homomorphism such that there is a partial isometry $v \in q M p$ such that $\theta(x) v = vx$ for all $x \in p Ap$. We first prove that without loss of generality we can assume that $p$ is majorized by the support of $\EE_{A}(v^*v)$. 

Let $z$ be the support of $\EE_{A}(v^*v)$. As $p \EE_{A}(v^*v) p =  \EE_{A}( p v^*v p ) =  \EE_{A}(  v^*v )  $ it follows that $z \in p Ap$. Further for $x \in p Ap$ we have  $x \EE_{A}(v^*v) = \EE_{A}(xv^*v) = \EE_{A}(v^* \theta(x) v ) =  \EE_{A}(v^*v x) =  \EE_{A}(v^*v ) x$ so that $z \in (pA p)'$. We conclude $z \in (pAp)'  \cap pAp$.  Now let $p' := p z \in A$,  let $\theta': p'  A p' \rightarrow q Q q$ be the restriction of $\theta$ to  $p'  A p'$ and let $v' := v z \in q M p' $. Then for $x \in p'  A p'$ we have $\theta' (x) v' = \theta(x) v z = v x z = vz x$. We claim further that $v'$ is non-zero. Indeed, $v'  = vz = 0$ iff $z v^\ast v z = 0$ iff $0 = \mathbb{E}_A(z v^\ast v z) =  z \mathbb{E}_A( v^\ast v) z$. But as $v$ is non-zero $\mathbb{E}_A(v^\ast v)$ is non-zero and hence  $z \mathbb{E}_A( v^\ast v) z \not = 0$ by construction of $z$. We conclude that $v'  \not = 0$. In all the tuple $(\theta, p' , q, v' )$ witnesses that $A \prec_M Q$ and the support of $\mathbb{E}_A((v')^\ast v')$ majorizes $p'$.  

For the remainder of the proof one just follows \cite[Lemma 2.1]{borstClassificationRightangledCoxeter2023} which does not affect the assumption that $p$ is majorized by the support of $\EE_{A}(v^*v)$. Althought the statement of \cite[Lemma 2.1]{borstClassificationRightangledCoxeter2023} is for $I$ with finite cardinality, the same proof can also be applied to $I$ with infinite cardinality.
\end{proof}

The following lemma is similar to \cite[Remark 2.3]{drimbePrimeII1Factors2019}. 
 
\begin{lemma}\label{lemma:non-embedding-single-netNEW} 
    Let $(M,\tau)$ be a tracial von Neumann algebra and let $A,  \{ B_k\}_{k\in I}$ be (possibly non-unital) von Neumann subalgebras of $M$. Here $I$ is an index set. Assume $A\not\prec_{M} B_k$ for every $k\in I$.  
    Then there is a single net $(u_i)_i$ of unitaries in $A$ such that for $k\in I$ and $a,b\in 1_{A}M1_{B_k}$ we have $\|\EE_{B_k}(a^*u_ib)\|_{2}\to 0$ as $i\to \infty$
    \begin{proof}
        Put 
        \begin{align}
            \widetilde{B} &= \bigoplus_{k\in I} B_k, &\quad 
            \widetilde{M} &= \bigoplus_{k\in I} M.
        \end{align} 
        Let $\pi:M\to \widetilde{M}$ be the (normal) diagonal embedding $\pi(x) = \bigoplus_{k\in I} x$.
        Suppose $\pi(A)\prec_{\widetilde{M}} \widetilde{B}$. Then there are projections $p\in \pi(A)$, $q\in \widetilde{B}$, a normal $*$-homomorphism $\theta:p\pi(A)p\to q\widetilde{B}q$ and a non-zero partial isometry $v\in q\widetilde{M}p$ s.t. $\theta(x)v = vx$ for $x\in p\pi(A)p$. For $k\in I$ let $\pi_k:\widetilde{M}\to M$ be the coordinate projections. Denote $p_k := \pi_k(p) \in A$, $q_k := \pi_k(q)\in B_k$ and $v_k := \pi_k(v) \in \pi_k(q\widetilde{M}p) = q_kMp_k$.
        Define a normal $*$-homomorphism $\theta_k: p_kAp_k\to q_kBq_k$ as $\theta_k(x) = \pi_k(\theta(\pi(x)))$. Then
        $\theta_k(x)v_k = \pi_k(\theta(\pi(x))v) = \pi_k(v\pi(x)) = v_k\pi_k(\pi(x)) = v_kx$. Since  $0\not=v = \bigoplus_{k\in I} v_k$ there is $ k_0\in I$ s.t. $v_{k_0}\not=0$. This then shows that $A\prec_{M} B_{k_0}$ which is a contradiction. We conclude that $\pi(A)\not\prec_{\widetilde{M}}\widetilde{B}$.

        Thus, there is a net of unitaries $(\widetilde{u}_i)_i$ in $\pi(A)$ s.t. for $a',b'\in 1_{\pi(A)}\widetilde{M}1_{\widetilde{B}}$ we have $\|\EE_{\widetilde{B}}(a'^*\widetilde{u}_ib')\|_2\to 0$ as $i\to \infty$. Let $u_i\in A$ be the unitary s.t. $\pi(u_i) = \widetilde{u_i}$. Fix $1\leq k\leq n$ and let $a,b\in 1_{A}M1_{B_k}$. We can choose $\widetilde{a},\widetilde{b}\in 1_{\pi(A)}\widetilde{M}1_{\widetilde{B}}$ s.t. $\pi_k(\widetilde{a}) = a$ and $\pi_k(\widetilde{b}) = b$. We have 
        \begin{align}
            \|\EE_{B_k}(a^*u_i b)\|_{2} = \|\EE_{\pi_k(\widetilde{B})}(\pi_k(\widetilde{a}^*\widetilde{u}_i\widetilde{b}))\|_{2} \leq \|\EE_{\widetilde{B}}(\widetilde{a}^*\widetilde{u}_i\widetilde{b})\|_2\to 0 &\quad \text{ as } i\to \infty.
        \end{align} This shows the net $(u_i)_i$ satisfies the stated property. 
    \end{proof}
\end{lemma}

In order to have control over quasi-normalizers we need the following lemma. The lemma is stated in \cite[Lemma D.3]{vaesRigidityResultsBernoulli2006a} for sequences, but holds equally well for nets.

\begin{remark}\label{Rmk=ExpecationCorners} Consider an inclusion $B \subseteq 1_B M 1_B$ of finite von Neumann algebras with conditional expectation $\EE_B: 1_B M 1_B \rightarrow B$. We extend it to $\EE_B: M \rightarrow B$ by setting $\EE_B(x) = \EE_B(1_B x 1_B)$. Fix a normal faithful tracial state $\tau$ on $M$.  If $p \in B$ is a non-zero projection then $p$ is in the multiplicative domain of $\EE_B$ and so $\EE_B: p M p \rightarrow p B p$ is a conditional expectation. If $\EE_B$ preserves $\tau$ then it also preserves the normal faithful tracial state $\tau(p)^{-1} \tau$ on $pMp$. 
  
\end{remark}

\begin{lemma}[Lemma D.3 in \cite{vaesRigidityResultsBernoulli2006a}]\label{lemma:sufficient-condition-quasi-normalizer}
	Let $(M,\tau)$ be a finite von Neumann algebra with normal faithful trace $\tau$ and let $B\subseteq 1_{B}M1_{B}$ and $A\subseteq 1_{A}B1_{A}$ be von Neumann subalgebras. Suppose there is a net of unitaries $(u_i)_i$ in $A$ such that for all $a,b\in M$ with $\EE_{B}(a) = \EE_{B}(b)=0$ we have
	\begin{align}
		\|\EE_{B}(au_ib)\|_{2}\to 0 \text{ as } i\to \infty.
	\end{align}
        Then if $n\geq 1$, $x_0,x_1,\ldots, x_n\in M$ satisfy $Ax_0\subseteq \sum_{k=1}^n x_kB$ then we have that $1_{A}x_01_{A}\in B$. 
        \begin{proof}                   
            We put $B_0 = 1_{A}B1_{A}$ and $M_0 = 1_{A}M1_{A}$ so that $A\subseteq B_0\subseteq M_0$ are unital inclusions. We observe $B_0 = B\cap M_0$.
            Now let $a,b\in M_0$ be such that $\EE_{B_0}(a)=\EE_{B_0}(b) =0$.             
             Then by \cref{Rmk=ExpecationCorners} with $p = 1_A$ we find $\EE_{B}(a) =   \EE_{B_0}(a) =0$ and similarly $\EE_{B}(b)=0$.  Thus by assumption $\|\EE_{B}(au_ib)\|_{2}\to 0$ as $i\to \infty$. Hence,  since $B_0 = 1_A B 1_A$  we obtain $\|\EE_{B_0}(au_ib)\|_{2}\to 0$ as $i\to \infty$.

                 Choose a central projection $z \in B \cap B'$ such that there exists $m\geq 1$ and partial isometries $v_i\in B$ for $1\leq i\leq m$ with $v_iv_i^*\leq 1_{A}$ and $\sum_{i=1}^m v_i^*v_i = z$. 
            Now let $n\geq 1$, $x_0,x_1,\ldots, x_n\in M$ be such that $Ax_0\subseteq \sum_{k=1}^n x_kB$. Then 
            \[
            A(1_{A} x_0 z 1_{A}) = ( A x_0 z ) 1_{A} \subseteq \sum_{k=1}^n x_k B z 1_{A}  = \sum_{k=1}^n\sum_{i=1}^m x_k (v_i^*1_{A}v_i)  B 1_{A}
            \subseteq \sum_{k=1}^n\sum_{i=1}^m x_k v_i^*B_0.
            \]          
            Multiplying both sides from the left with $1_A$ gives $ A(1_{A} x_0 z 1_{A}) \subseteq \sum_{k=1}^n\sum_{i=1}^m 1_A x_k v_i^*B_0$ where $ 1_A x_k v_i^* \in 1_A M 1_A$. 
            By the existence of the net $(u_i)_i$ this implies, by applying \cite[Lemma D.3]{vaesRigidityResultsBernoulli2006a} to the inclusions $A \subseteq B_0 \subseteq M_0$, that $1_{A}x_0 z 1_{A}\in B$. As we may let $z$ approximate $1_B$ in the strong topology we find that $1_{A}x_0 1_{A}\in B$. 
        \end{proof}
\end{lemma}

We are now able to show the following result which generalizes \cite[Proposition 2.3]{borstClassificationRightangledCoxeter2023} to general graph products. The second statement in the proposition should be compared to \cite[Lemma 9.4]{ioanaCartanSubalgebrasAmalgamated2015}. While the inclusion $M_{\Lambda}\subseteq M_{\Gamma}$ is generally not mixing, we still have enough control over the (quasi)-normalizer of subalgebras. 
The proof of \cref{prop:embed}(1) uses \cref{lemma:non-embedding-single-netNEW}, \cref{lemma:sufficient-condition-quasi-normalizer}  and the results from \cref{subsection:calculating-conditional-expectation} for calculating conditional expectations in graph products. The proof of \cref{prop:embed}(2) uses (1) and \cref{Lem=TransitiveEmbedding} which is analogues to \cite{borstClassificationRightangledCoxeter2023}, but we included it for convenience.

\begin{proposition}\label{prop:embed}
    Let $\Gamma$ be a simple graph and  for $v\in \Gamma$ let $(M_{v},\tau_v)$ be a finite von Neumann algebra with normal faithful trace $\tau_v$. Let $\Lambda\subseteq \Gamma$ be a subgraph, and $\{\Lambda_j\}_{j\in \calJ}$ be a non-empty, collection of subgraphs of $\Gamma$. Define  
    \begin{align}\label{eq:embedding-graph}
        \Lambda_{\emb}:= \Lambda\cup \bigcap_{j\in \calJ}\bigcup_{v\in \Lambda\setminus \Lambda_j}\Link_{\Gamma}(v). 
    \end{align}
     Let $A\subseteq 1_A  M_{\Gamma} 1_A$ be a von Neumann subalgebra. 
    \begin{enumerate}
        \item \label{Item=Embed1} If $A\subseteq 1_A M_{\Lambda} 1_A$ and 
        $A\not\prec_{M_{\Gamma}}M_{\Lambda_j}$ for all $j\in \calJ$ then the following properties hold true:
         \begin{enumerate}
                \item There is a net $(u_i)_i$ of unitaries in $A$ such that for all $a,b\in 1_{A}M_{\Gamma}1_{A}$ with $\EE_{M_{\Lambda_{\emb}}}(a)=\EE_{M_{\Lambda_{\emb}}}(b)=0$ we have    
                $       \|\EE_{M_{\Lambda_{\emb}}}(au_ib)\|_{2} \rightarrow  0$;
                \label{Item=Embed1a}
                \item $1_{A}\qNor_{M_{\Gamma}}(A)''1_{A}\subseteq M_{\Lambda_{\emb}}$;\label{Item=Embed1b}
                \item For any unitary $u\in M_{\Gamma}$ satisfying $u^*Au\subseteq M_{\Lambda_{\emb}}$ we have $1_{A}u1_{A}\in M_{\Lambda_{\emb}}$. \label{Item=Embed1c}
            \end{enumerate}
        \item \label{Item=Embed2} Denote $P=\Nor_{M_{\Gamma}}(A)''$ and let $r\in P\cap P'$ be a projection. If $rA\prec_{M_{\Gamma}}M_{\Lambda}$ and $rA\not\prec_{M_{\Gamma}}M_{\Lambda_j}$ for $j\in \calJ$ then $rP\prec_{M_{\Gamma}}M_{\Lambda_{\emb}}$.
    \end{enumerate}
    We remark that if $\{\Lambda_j\}_{j\in \calJ}$ enumerates all strict subgraphs of $\Lambda$ then $\Lambda_{\emb} = \Lambda\cup \Link_{\Gamma}(\Lambda)$.     
\end{proposition}
    \begin{proof}
	\eqref{Item=Embed1a}
  By \cref{lemma:non-embedding-single-netNEW} we can build a net of unitaries $(u_i)_i$ in $A$ such that for any $a,b\in M_{\Gamma}$ and any $j\in \calJ$ we have $\|\EE_{M_{\Lambda_j}}(a u_i b)\|_2\to 0$  when $i\to \infty$.  
		We show the net $(u_i)_i$ satisfies the properties of \eqref{Item=Embed1a}.			
	Let $b\in \mathring{M}_{\vv}$ and $c\in \mathring{M}_{\ww}$ for some $\vv,\ww \in \calW_{\Gamma}\setminus \calW_{\Lambda_{\emb}}$. Write $\vv = \vv_l\vv_c\vv_r$ and $\ww = \ww_l\ww_c\ww_r$ with $\vv_l,\ww_l\in \calW_{\Lambda_{\emb}}$, $\vv_r,\ww_r\in \calW_{\Lambda}$ and
		such that $\vv_c$ and $\ww_c$ do not start with letter from $\Lambda_{\emb}$ nor do they end with letters from $\Lambda$.
		Now write $b = b_l b_c b_r$ and 
		$c = c_l c_c c_r$ with $b_l \in \mathring{M}_{\vv_l}, c_l \in \mathring{M}_{\ww_l}$, $b_c \in \mathring{M}_{\vv_c},  c_c \in \mathring{M}_{\ww_c}$ 
		and $b_r\in \mathring{M}_{\vv_r}, c_r \in \mathring{M}_{\ww_r}$.
            Then as $\vv\not\in \calW_{\Lambda_{\emb}}$ and $\vv_l\in \calW_{\Lambda_{\emb}}$ and $\vv_r\in \calW_{\Lambda}\subseteq \calW_{\Lambda_{\emb}}$, we have $\vv_{c}\not\in \calW_{\Lambda_{\emb}}$ and hence there is a letter $v$ of $\vv_c$ such that $v\not\in \Lambda_{\emb}$. Thus, there is an index $j\in \calJ$ such that $v\not\in \bigcup_{w\in \Lambda\setminus \Lambda_j}\Link_{\Gamma}(w)$.
            Hence $\Link(v)\subseteq \Gamma\setminus (\Lambda\setminus \Lambda_j) = \Lambda_j \cup (\Gamma\setminus \Lambda)$ and thus $\Lambda \cap \Link(\vv_c)\subseteq \Lambda_j$.
             Using \cref{lem:Expectation-for-n=2} we get,  
		\[
  \begin{split}
			\|\EE_{M_{\Lambda_{\emb}}}(b^*u_ic)\|_2 =& \|\EE_{M_{\Lambda_{\emb}}}(b^*\EE_{M_{\Lambda}}(u_i)c)\|_{2} \\
		=& \|\tau(b_c^*c_c)b_r^*\EE_{M_{\Lambda\cap \Link(\vv_c)}}(b_l^*u_ic_{l})c_{r}\|_{2}\\
	  = & \|\tau(b_c^*c_c)b_r^*\EE_{M_{\Lambda\cap \Link(\vv_c)}}(   \EE_{M_{\Lambda_j}}(   b_l^* u_i c_{l} ) )c_{r}\|_{2} \\
   \leq & \Vert b_c \Vert_2 \Vert c_c\Vert_2  \Vert b_r \Vert \Vert c_r \Vert \Vert \EE_{M_{\Lambda_j}}(   b_l^* u_i c_{l} )  \|_{2}.
	\end{split}
 \]  
    We see that this expression converges to 0 when $i\to\infty$. 
	Thus, more generally, for $b,c\in M_{\Gamma}$ with $\EE_{M_{\Lambda_{\emb}}}(b)=\EE_{M_{\Lambda_{\emb}}}(c) =0$, we obtain $\|\EE_{M_{\Lambda_{\emb}}}(b^*u_ic)\|_2\to 0$ when $i\to \infty$, which shows \eqref{Item=Embed1a}.

    \eqref{Item=Embed1b}  Observe that if $x\in \qNor_{M_{\Gamma}}(A)$ then for some $n\geq 1$ and 
 $x_1,\ldots,x_n\in M_{\Gamma}$ we have $Ax\subseteq \sum_{k=1}^n x_k A \subseteq \sum_{k=1}^n x_k M_{\Lambda_{\emb}}$. Therefore by the existence of the net $(u_i)_i$ shown by \eqref{Item=Embed1a} and by \cref{lemma:sufficient-condition-quasi-normalizer}, we have 
 that $1_{A}x1_{A}\in M_{\Lambda_{\emb}}$.  This shows $1_{A}\qNor_{M_{\Gamma}}(A)1_{A}\subseteq M_{\Lambda_{\emb}}$ and thus proves \eqref{Item=Embed1b}. 

 \eqref{Item=Embed1c} Let $u\in M_{\Gamma}$ be a unitary for which $u^*Au\subseteq M_{\Lambda_{\emb}}$. Then $Au\subseteq uM_{\Lambda_{\emb}}$ so again by the existence of the net $(u_i)_i$ shown by \eqref{Item=Embed1a} and by \cref{lemma:sufficient-condition-quasi-normalizer}, we obtain that $1_{A}u1_{A}\in M_{\Lambda_{\emb}}$.

\eqref{Item=Embed2} By replacing $\{\Lambda_j\}_{j\in \calJ}$  with $\{\Lambda_j \cap \Lambda\}_{j\in \calJ}$ we may assume that $\Lambda_j\subseteq \Lambda$ for $j\in \calJ$. We observe that $r$ is central in $A$, which we will use a number of times in the proof. By Lemma \ref{Lem=TransitiveEmbedding} the assumptions imply that there exist projections $p \in r A, q \in M_{\Lambda}$ a non-zero partial isometry $v \in q  M_{\Gamma} p$ and a normal $\ast$-homomorphism $\theta: p  A p \rightarrow q M_{\Lambda} q$ such that $\theta(x) v = v x$ for all $x \in p  A p$ and such that moreover  $\theta(p  A p) \not \prec_{M_{\Lambda}} M_{\Lambda_j}$ for $j\in \calJ$. From  \eqref{Item=Embed1} we see that $\theta(p)\qNor_{M_{\Gamma}}(\theta(p A p)) \theta(p) \subseteq M_{\Lambda_{\emb}}$.

	Now take $u \in \Nor_{M_{\Gamma}}(A)$. We follow the proof of \cite[Lemma 3.5]{popaStrongRigidityII12006} or \cite[Lemma 9.4]{ioanaCartanSubalgebrasAmalgamated2015}. Take $z \in A$ a central projection such that $z = \sum_{j=1}^n v_j v_j^\ast$ with $v_j \in A$ partial isometries such that $v_j^\ast v_j \leq p$. Then 
	\[
	pz u pz (p Ap) \subseteq pzuA = pz A u = p A zu \subseteq \sum_{j=1}^n (p A v_j) v_j^\ast u \subseteq \sum_{j=1}^n (p A p) v_j^\ast u,
	\]
	and similarly $(pAp) pz u pz \subseteq \sum_{j=1}^n u v_j (p A p)$. We conclude that $pz u pz \in \qNor_{pMp}(pAp) $.
	
	Now if $x \in \qNor_{pM_{\Gamma}p}(p A p)$ then by direct verification we see that we have that  $v x v^\ast \in \theta(p)\qNor_{q M_{\Gamma} q}(\theta(p A p)) \theta(p) $. It follows that $v p z u pz v^\ast$, with $u \in \Nor_{M_{\Gamma}}(A)$ as before, is contained in  $\theta(p)\qNor_{q M_{\Gamma} q}(\theta(pAp)) \theta(p)$ which was contained in  $M_{\Lambda_{\emb}}$. We may take the projections $z$ to approximate the central support of $p$ and therefore  $v u v^\ast = v p   u p  v^\ast \in M_{\Lambda_{\emb}}$.  Hence $v  \Nor_{M_{\Gamma}}( A)'' v^\ast \subseteq M_{\Lambda_{\emb}}$. 
	Set $p_1 = v^\ast v \in pA'p$.  Note that $p_1 \leq p \leq r$. As both $A$ and $A'$ are contained in  $\Nor_{M_{\Gamma}}( A)''$ we find that $p_1 \in \Nor_{M_{\Gamma}}( A)''$ (as $p \in A$). 
	So we have the
	$\ast$-homomorphism $\rho: p_1  \Nor_{M_{\Gamma}}( A)'' p_1  =  p_1  r \Nor_{M_{\Gamma}}( A)'' p_1   \rightarrow M_{\Lambda_{\emb}}: x \mapsto v x v^\ast$ with  
	$v \in q M_{\Gamma} p_1$ and clearly $\rho(x) v = v x$. We conclude that    $r \Nor_{M_{\Gamma}}( A)'' \prec_{M_{\Gamma}}  M_{\Lambda_{\emb}}$.
	\end{proof}
 
We prove the following result concerning embeddings in graph products.

\begin{proposition}\label{Lem=Selfembedding}
Let $\Gamma$ be a simple graph and for $v\in \Gamma$, and let $(M_{v},\tau_v)$  be a tracial von Neumann algebra. Fix $v\in \Gamma$ and let $N\subseteq M_v$ be diffuse. If $N \prec_{M_{\Gamma}} M_\Lambda$ for some subgraph $\Lambda \subseteq \Gamma$, then $v \in \Lambda$. In particular if $\Lambda = \{w\}$, a singleton set, then $v = w$.
\begin{proof}  
Let $\Lambda\subseteq \Gamma$ be a subgraph with $v\not\in \Lambda$. We show that $N\not\prec_{M_{\Gamma}}M_{\Lambda}$. Since $N$ is diffuse, we can choose a net $(u_k)_{k}$ of unitaries in $N$ such that $\tau(u_k)=0$ and $u_k \rightarrow  0$ $\sigma$-weakly. Since $\lambda(\boldM_{\Gamma})$ is a dense subspace of $M_{\Gamma}$, it is sufficient to show for any reduced operators $x=x_1x_2\dots x_m,y=y_1y_2\dots y_n$, s.t. $x_i\in \mathring{M}_{v_i}$, $y_i\in \mathring{M}_{w_i}$, we have $\|\EE_{M_{\Lambda}}(xu_ky)\|_2\to0$. Indeed, writing $x=x' a$, $y=by'$, where $a,b\in M_{v}$ and where $x'$ respectively $y'$ is a reduced operator without letter $v$ at the end respectively start. Then $$xu_ky=x' au_k by'=x'\tau(au_kb)y'+x'(au_kb-\tau(au_kb))y'.$$ On the one hand, $\mathbb{E}_{M_{\Lambda}}(x'\tau(au_kb)y')=\tau(au_kb)\mathbb{E}_{M_{\Lambda}}(x'y')=\langle u_kb,a^*\rangle\mathbb{E}_{M_{\Lambda}}(x'y')\to0$. On the other hand, we write $x'=x''d$, $y'=ey''$, where $d,e\in M_{\Link(v)}$ and where $x''$ respectively $y''$ has no letter from $\Star(v)$ at the end respectively at the start. Then we have 
\[
\begin{split}
    x'(au_kb-\tau(au_kb))y'&=x''d(au_kb-\tau(au_kb))ey''\\
    &=x''de(au_kb-\tau(au_kb))y''\\
    &=\sum_i x''f_i(au_kb-\tau(au_kb))y'',
\end{split}
\]
where we write $de=\sum_if_i$ and $f_i\in M_{Link(v)}$ reduced. Since $x''f_i(au_kb-\tau(au_kb))y''$ is reduced and $v\not\in \Lambda$ we obtain that $\mathbb{E}_{M_{\Lambda}}(x''f_i(au_kb-\tau(au_kb))y'')=0$   by \cite[Remark 2.4, final remark]{caspersGraphProductsOperator2017a}. Thus $\|\EE_{M_{\Lambda}}(xu_ky)\|_{2}\to 0$, which completes the proof.
\end{proof}
\end{proposition}
 
\begin{remark}
    We remark in particular for any graph $\Gamma$, II$_1$-factors $\{M_{v}\}_{v\in \Gamma}$ and a finite subgraph $\Lambda\subseteq \Gamma$ that $\qNor_{M_{\Gamma}}(M_{\Lambda})'' = \Nor_{M_{\Gamma}}(M_{\Lambda})'' = M_{\Lambda\cup \Link(\Lambda)}$. Indeed, clearly $M_{\Lambda},M_{\Link(\Lambda)}\subseteq \Nor_{M_{\Gamma}}(M_{\Lambda})''$ (as $M_{\Link(\Lambda)} = M_{\Lambda}'\cap M_{\Gamma}$) so that $M_{\Lambda\cup \Link(\Lambda)}\subseteq \Nor_{M_{\Gamma}}(M_{\Lambda})''\subseteq \qNor_{M_{\Gamma}}(M_{\Lambda})''$.
    Furthermore, by \cref{Lem=Selfembedding} we have $M_{\Lambda}\not\prec_{M_{\Gamma}} M_{\widetilde{\Lambda}}$
    for any strict subgraph $\widetilde{\Lambda}\subsetneq \Lambda$ so that by \cref{prop:embed} we obtain $\qNor_{M_{\Gamma}}(M_{\Lambda})''\subseteq M_{\Lambda\cup \Link(\Lambda)}$.
\end{remark}

\subsection{Unitary conjugacy in graph products}
\label{subsection:unitary-conjugacy-in-graph-products}

We prove \cref{lemma:unitary-inclusion-in-star-for-vNa} which gives sufficient conditions for a subalgebra $Q\subseteq M_{\Gamma}$ to unitarily embed in a subalgebra $M_{\Lambda_{\emb}}$. This can be seen as a generalization of \cite[Theorem 3.3]{OzawaKurosh} where a unitary embedding is proven for free products.
The proof of \cref{lemma:unitary-inclusion-in-star-for-vNa} combines (the second half of) the proof of \cite[Theorem 3.3]{OzawaKurosh} with results of \cref{sub:embedding-quasi-normalizers} concerning embeddings in graph products. 
 
\begin{theorem}
\label{lemma:unitary-inclusion-in-star-for-vNa}
Let $\Gamma$ be a simple graph and for $v\in \Gamma$ let $(M_v,\tau_v)$ be a II$_1$-factor with normal faithful trace $\tau_v$. Let $Q\subseteq M_{\Gamma}$ be a subfactor whose relative commutant $Q' \cap M_{\Gamma}$ is also a factor. Let $\Lambda\subseteq \Gamma$ be a subgraph and let $\{\Lambda_j\}_{j\in \calJ}$ be a non-empty collection of subgraphs of $\Lambda$. Suppose $Q\prec_{M_{\Gamma}} M_{\Lambda}$ and $Q\not\prec_{M_{\Gamma}} M_{\Lambda_j}$ for $j\in \calJ$. Then there is a unitary $u\in M_{\Gamma}$ such that $u^*Qu\subseteq M_{\Lambda_{\emb}}$,  where $\Lambda_{\emb}$ is defined as  in \eqref{eq:embedding-graph}. 
\begin{proof}
Since $Q\prec_{M_{\Gamma}} M_{\Lambda}$ and $Q \not\prec_{M_{\Gamma}} M_{\Lambda_j}$ for $j\in \calJ$ we have by \cref{Lem=TransitiveEmbedding} that there are projections $q\in Q$, $e\in M_{\Lambda}$, a normal $*$-homomorphism $\theta:qQq\to eM_{\Lambda}e$ and a non-zero partial isometry $v\in eM_{\Gamma}q$ such that $\theta(x)v=vx$ for $x\in qQq$ and such that moreover $\theta(qQq)\not\prec_{M_{\Gamma}} M_{\Lambda_j}$ for $j\in \calJ$. 
We may moreover assume that $q$ is majorized by the  support of $\EE_Q(v^\ast v)$.
Let $q_0 \in Q$ be a  non-zero projection with $q_0 \leq q$ and trace $\tau(q_0)=\frac{1}{m}$ for some $m\geq 1$. Put $v_0 := vq_0$. Note that  $v^\ast v \in (q Q q)' \cap qM_\Gamma q$.  Then $\EE_{Q}(v_0^\ast v_0) = \EE_Q(q_0 v^\ast v q_0) =  \EE_Q( q_0 v^\ast v)  = q_0\EE_Q(v^\ast v)$ and the latter expression is non-zero by the assumption that the support of   $\EE_Q(v^\ast v)$  majorizes  $q$. As $\EE_Q$ is faithful $v_0 \not = 0$. 
Define $\theta_0: q_0Qq_0\to eM_{\Lambda}e$ as $\theta_0:=\theta|_{q_0Qq_0}$. Then for $x\in q_0 Q q_0$ we have $\theta_0(x)v_0 = \theta(x)vq_0 = vxq_0 = v_0x$. Automatically this implies $v_0^\ast v_0\in (q_0Qq_0)' \cap q_0M_{\Gamma}q_0$. Furthermore for $j\in \calJ$, the corner $\theta_0(q_0Qq_0) = \theta_0(q_0)\theta(qQq)\theta(q_0)$ does not embed in $M_{\Lambda_j}$ inside $M_{\Gamma}$ since $\theta(qQq)$ does not embed in $M_{\Lambda_j}$ inside $M_{\Gamma}$. Hence, by \cref{prop:embed}\eqref{Item=Embed1b} we obtain $\theta(q_0)\Nor_{M_{\Gamma}}(\theta_0(q_0Qq_0))''\theta(q_0)\subseteq M_{\Lambda_{\emb}}$.

Since $Q$ is a factor and $\tau(q_0)=\frac{1}{m}$ we can for $j=1,\ldots, m$ choose a partial isometry $u_j$ in $Q$ such that $u_j^*u_j = q_0$ and $\sum_{j=1}^m u_ju_j^* =1_{M_{\Gamma}}$. We may moreover assume that $u_1 = q_0$. 
We define a projection $q' := \sum_{j=1}^m u_jv_{0}^*v_{0}u_j^*\in M_{\Gamma}$. We show that $q'\in Q'\cap M_{\Gamma}$. 
Indeed, let $y\in Q$. Then using that $v_0^\ast v_0\in (q_0Qq_0)'$ and $u_j^*yu_j\in q_0Qq_0$ for $j=1,\ldots, n$ we get 
       \[
       \begin{split}
            q'y &= \sum_{j=1}^m u_j(v_0^\ast v_0)u_j^*y 
            = \sum_{j=1}^m\sum_{i=1}^m u_j(v_0^\ast v_0)(u_j^*yu_i)u_i^* \\
            &= \sum_{j=1}^m\sum_{i=1}^m u_j(u_j^*yu_i)(v_0^\ast v_0)u_i^* 
            = \sum_{i=1}^m yu_i(v_0^\ast v_0)u_i^* = yq'.
       \end{split}
       \]
       and thus $q'\in Q'\cap M_{\Gamma}$. We observe that $v_0^\ast v_0 = q_0q'q_0 = q_0q'$  which shows in particular that $q'$ is non-zero (since $v_0\not=0$). Since $Q'\cap M$ is a (finite) factor and $q'$ is a non-zero projection, we can choose a projection $q_0'\in Q'\cap M$ with $q_0'\leq q'$ and $\tau(q_0') = \frac{1}{n}$ for some $n\geq 1$.  Since $Q'\cap M_\Gamma$ is a factor and since $\tau(q_0')=\frac{1}{n}$ we can find partial isometries $u_1',\ldots, u_n'\in Q'\cap M_{\Gamma}$ with  $(u_k')^*u_k' = q_0'$ for $k=1,\ldots, n$ and such that $\sum_{k=1}^n u_k'(u_k')^* = 1_{M_{\Gamma}}$.
       
        We then put $v_{00} := v_0q_0' = vq_0q_0'\in eM_{\Gamma}q_0$. 
        Observe that $v_{00}^\ast v_{00} = q_0'v_0^\ast v_0 q_0' = q_0'q_0$ has trace $\tau(v_{00}^*v_{00}) = \tau(q_0')\tau(q_0) =\frac{1}{nm}$ so in particular $v_{00}$ is non-zero.
        Then for $x\in q_0Qq_0$ we have $\theta_0(x)v_{00} = \theta_0(x)v_0q_0' = v_0xq_0' = v_{00}x$. 
        Therefore, we obtain $v_{00}v_{00}^\ast \in \theta_0(q_0Qq_0)'\cap M_{\Gamma}\subseteq \Nor_{M_{\Gamma}}(\theta_0(q_0Qq_0))''$.
        As $v_{00}v_{00}^\ast \leq \theta(q_0)$ we obtain   $v_{00}v_{00}^\ast\in \theta(q_0)\Nor_{M_{\Gamma}}(\theta_0(q_0Qq_0))''\theta(q_0)\subseteq M_{\Lambda_{\emb}}$ using the first paragraph.

      Since $M_{\Lambda_{\emb}}$ is a factor (as it is a graph product of II$_1$-factors), and since $\tau(v_{00}v_{00}^\ast) = \frac{1}{nm}$ there exist for $j=1,\ldots, m$, $k=1,\ldots, n$ partial isometries $w_{j,k}\in M_{\Lambda_{\emb}}$ with $w_{j,k}w_{j,k}^* = v_{00}v_{00}^\ast$ and $\sum_{j=1}^m\sum_{k=1}^n w_{j,k}^*w_{j,k} = 1_{M_{\Gamma}}$.        
         Finally, we define the unitary $u := \sum_{j=1}^m\sum_{k=1}^n u_ju_k'v_{00}^\ast w_{j,k}\in M_{\Gamma}$. Then for $x\in Q$ we have 
         \begin{align*}
        u^*xu &= \sum_{j_1=1}^m\sum_{k_1=1}^n\sum_{j_2=1}^m\sum_{k_2=1}^n w_{j_1,k_1}^*v_{00}(u_{k_1}')^*(u_{j_1}^*xu_{j_2})u_{k_2}'v_{00}^*w_{j_1,k_2}\\    
        &=\sum_{j_1=1}^m\sum_{k=1}^n\sum_{j_2=1}^m w_{j_1,k}^*v_{00}(u_{j_1}^*xu_{j_2})q_0'v_{00}^\ast w_{j_1,k}\\  
        &=\sum_{j_1=1}^m\sum_{k=1}^n\sum_{j_2=1}^m w_{j_1,k}^*\theta_0(u_{j_1}^*xu_{j_2})v_{00}v_{00}^\ast w_{j_1,k}\\ 
        &=\sum_{j_1=1}^m\sum_{k=1}^n\sum_{j_2=1}^m w_{j_1,k}^*\theta_0(u_{j_1}^*xu_{j_2})w_{j_1,k} \in M_{\Lambda_{\emb}}.
         \end{align*}
        Hence $u^*Qu\subseteq M_{\Lambda_{\emb}}$.
    \end{proof}
\end{theorem}

\section{Graph product rigidity}\label{Sect=GraphProductRigidity}
The aim of this section is to prove \cref{thm:rigid-graph-decomposition}. This provides a rather general class of graphs and von Neumann algebras such that the graph product completely remembers the graph and the vertex von Neumann algebra up to stable isomorphism. Note that we cannot expect to cover all graphs as this would imply the free factor problem and which is beyond reach of our methods. The class of rigid graphs as presented in Section \ref{Sect=RigidGraph} is therefore natural.

\subsection{Vertex von Neumann algebras}
We define classes of von Neumann algebras for which we first recall a version of the Akemann-Ostrand property \cite{houdayerUniquePrimeFactorization2017}. 

\begin{definition}
    Let  $M$ be a von Neumann algebra with standard form $(M, L^2(M), J, L^2(M)^+)$. We say that $M$  has  {\it strong property (AO) }  if there exist   unital C$^\ast$-subalgebras $A \subseteq M$ and $C \subseteq \bound(L^2(M))$ such that:
    \begin{itemize}
\item $A$ is $\sigma$-weakly dense in $M$, 
\item $C$ is nuclear and contains $A$, 
\item The commutators $[C, JAJ] = \{  [c, JaJ] \mid c \in C, a \in A \}$ are contained in the space of compact operators $\Kompact(L^2(M))$. 
\end{itemize}
\end{definition}

We recall that a wide class of examples of von Neumann algebras with property strong (AO) comes from hyperbolic groups. 
 
\begin{theorem}[See Lemma 3.1.4 of  \cite{isonoBiexactnessDiscreteQuantum2015} and remarks before] \label{Thm=NuclearInt}
Let $G$ be a discrete hyperbolic group.    Consider the anti-linear isometry $J$ determined by  
\[
J: \ell^ 2(G) \rightarrow \ell^2(G): \delta_s \mapsto \delta_{s^{-1}}, \qquad s \in G.
\] 
Then there is a  nuclear  C$^\ast$-algebra $C$ such that: 
\begin{enumerate}
    \item $C_r^\ast(G) \subseteq C \subseteq \bound(\ell^2(G))$.
    \item $C$ contains all compact operators. 
    \item The commutator $[C, J C_r^\ast(G)  J]$ is contained in the space of compact operators. 
\end{enumerate}
\end{theorem}

\begin{remark}
In view of Section \ref{Sect=NuclearGraphProduct} it is worth to note that we may always assume without loss of generality that $C$ contains the space of compact operators by replacing $C$ by $C + \Kompact(L^2(M))$ if necessary, see \cite[Remark 2.7]{houdayerUniquePrimeFactorization2017}. This fact also underlies \cref{Thm=NuclearInt}. 
\end{remark}

\begin{definition}\label{def:class-vertex-class-rigid}
We define the following classes of von Neumann algebras:
\begin{itemize}
    \item  Let $\Cvert$ be the class of II$_1$-factors $M$ with separable predual $M_\ast$ that satisfy condition strong (AO)  and which are non-amenable;
\item  Let $\Ccomplete$ be the class of all von Neumann algebraic graph products $(M_{\Gamma},\tau) = *_{v,\Gamma}(M_v,\tau_v)$ 
of tracial von Neumann algebras $(M_v,\tau_v)$ in $\Cvert$ taken over non-empty, finite,  complete graphs $\Gamma$;
\item  Let $\Crigid$ be the class of all von Neumann algebraic graph products $(M_{\Gamma},\tau) = *_{v,\Gamma}(M_v,\tau_v)$ 
of tracial von Neumann algebras $(M_v,\tau_v)$ in $\Cvert$ taken over non-empty, rigid graphs $\Gamma$.
\item  Let $\Crigid^f$ be defined in the same way as $\Crigid$ with the additional assumption that $\Gamma$ is finite. 
\end{itemize}
\end{definition}

\begin{remark}\label{remark:classes-graph-products}
    We remark that $\Cvert\subseteq \Ccomplete\subseteq \Crigid^f \subseteq \Crigid$. Furthermore,
    \begin{enumerate}
        \item \label{Item=class-Cvert} The class $\Cvert$ is closed under  taking (finitely many) free products (see \cite[Example 2.8(5)]{houdayerUniquePrimeFactorization2017}). Furthermore, all von Neumann algebras $M\in \Cvert$ are solid and prime, see \cite{ozawaSolidNeumannAlgebras2004};
        \item \label{Item=class-Ccomplete} The class $\Ccomplete$ is closed under taking tensor products. Moreover, we observe that $\Ccomplete$ coincides with the class of tensor products of factors from $\Cvert$;
        \item \label{Item=class-Crigid} The class $\Crigid$ is closed under taking  graph products over non-empty, \textit{rigid} graphs by \cref{remark:graph-product-of-graphs-consistency} and \cref{lemma:graph-product-of-rigid-graphs}. In particular, it is closed under tensor products;
        \item \label{Item=class-Crigid-minus-Cvert} The class $\Crigid\setminus \Cvert$ is closed under taking graph products over \textit{arbitrary} non-empty, graphs by \cref{remark:graph-product-of-graphs-consistency} and \cref{lemma:graph-product-of-rigid-graphs}. In particular, it is closed under tensor products and under free products.
    \end{enumerate}
\end{remark}

\begin{remark}\label{remark:need-for-rigid-graphs}
 We show  that it may happen that a graph product over a rigid graph is isomorphic to a graph product over a non-rigid graph; even if all vertex von Neumann algebras come from the class $\Cvert$. Consider the graph $\mathbb{Z}_4$ defined in \cref{example:rigid-graphs}\eqref{Item=rigid-graphsZn}. The graph $\ZZ_4$ is not rigid. For $v \in \mathbb{Z}_4$ let $G_v$ be a countable discrete group. Let $H_v = G_v \ast G_{v+2}$. We have for the graph products of groups that
\[
\ast_{v, \mathbb{Z}_4} G_v = (G_0 \ast G_2) \times (G_1 \ast G_3) = \ast_{v, \mathbb{Z}_2} H_v.
\]
We now set $G_v = \mathbb{F}_2$ and $H_v = \mathbb{F}_4$ to be free groups with 2 and 4 generators respectively. 
Set $M_v = \mathcal{L}(\mathbb{F}_2), v \in \mathbb{Z}_4$ and $N_v = \mathcal{L}(\mathbb{F}_4), v \in \mathbb{Z}_2$ equipped with their tracial Plancherel states $\tau_v$. Then $M_v$ and $N_v$ are in class $\Cvert$ and  $\ast_{v, \mathbb{Z}_4} (M_v, \tau_v) = \ast_{v, \mathbb{Z}_2} (N_v, \tau_v)$. We have thus given an example of a rigid and non-rigid graph that give isomorphic graph products.
 \end{remark}

\begin{remark} 
 We show  that it may happen that a graph product over a rigid graph with vertex algebras in $\Cvert$ is isomorphic to a graph product over a different rigid graph with vertex algebras that are not in $\Cvert$.   Let $\Gamma$ be a rigid graph and for $v \in \Gamma$ let $\Lambda_v$ be a rigid graph; assume all these graphs are non-empty. The graph product of graphs $\Lambda_\Gamma$ is rigid by Lemma \ref{lemma:graph-product-of-rigid-graphs}.   
 Then for any $v \in \Gamma, w \in \Lambda_v$ let $G_w = \mathbb{F}_2$. Then 
 \[
 \ast_{v, \Gamma} (\ast_{w, \Lambda_v}  \calL(G_w)) = \ast_{w, \Lambda_\Gamma} \calL(G_w).
 \]
 The right hand side is a graph product over $\Lambda_\Gamma$ of von Neumann algebras in  $\Cvert$ and hence is contained in $\Crigid$. The left hand side is a graph product over $\Gamma$ of von Neumann algebras $\ast_{w, \Lambda_v}  \calL(G_w)$. The latter von Neumann algebras are not in $\Cvert$ for the fact that otherwise  they would be solid \cite{ozawaSolidNeumannAlgebras2004}. However,  $\Lambda_v$ being rigid implies that it contains at least two points that share an edge and hence $\ast_{w, \Lambda_v}  \calL(G_w)$ contains $\calL(\mathbb{F}_2 \times \mathbb{F}_2)$ which is an obstruction to solidity.   
 \end{remark}

\subsection{Key result for embeddings of diffuse subalgebras in graph products}  In this section we fix the following notation.  
Let $\Gamma$ be a simple graph. For $v\in \Gamma$ let $(M_{v},\tau_v)$ be a tracial von Neumann algebra ($M_v\not=\CC$) that satisfies strong (AO) and has a separable predual. Let $(M_{\Gamma},\tau_\Gamma) = \ast_{v,\Gamma}(M_{v},\tau_v)$ be the von Neumann algebraic graph product. For $v\in \Gamma$ let $\calH_{v} = L^2(M_v,\tau_v)$ and let $\calH_{\Gamma}$ be the graph product of these Hilbert spaces, which is the standard Hilbert space of $M_{\Gamma}$ \cite{caspersGraphProductsOperator2017a}. We denote by $J:\calH_{\Gamma}\to \calH_{\Gamma}$ the modular conjugation.   Let $B_v = \bound(\calH_v)$. Let $\Omega_v = 1_{M_v}$ as a vector in $\calH_v$ and let $\omega_v(x) = \langle x \Omega_v, \Omega_v \rangle, x \in B_v$. Then $\omega_v$ is a GNS-faithful  state on $B_v$ and the GNS-space of $\omega_v$ can canonically be identified with $\calH_v$. The reduced C$^\ast$-algebraic  graph product $(B_\Gamma, \omega_\Gamma) = \astred_{v, \Gamma} (B_v,\omega_v)$ gives then by construction a C$^\ast$-subalgebra $B$ of $\bound(\calH_{\Gamma})$. We let $\lambda_v: B_v \rightarrow B$ be the canonical embedding. Furthermore we let $\rho_{v}: B_{v}^{\op}\to B^{\op}$ be the map $\rho_{v}(x^{\op}) = J\lambda_{v}(x)^*J$. As for $v\in \Gamma$ the von Neumann algebra $M_{v}$ has strong property (AO) by assumption, there are unital $\Cstar$-subalgebras $C_v\subseteq B_v$ and $A_v\subseteq M_v\cap C_v$ such that
\begin{enumerate}
    \item The C$^\ast$-algebra $A_v$  are $\sigma$-weakly dense in $M_v$,
    \item The C$^\ast$-algebra $C_v$ are nuclear,
    \item The commutators $[C_v,J_vA_vJ_v]$ are contained in $\KK(\calH_v)$.
\end{enumerate}
As in \cite[Remarks 2.7 (1)]{houdayerUniquePrimeFactorization2017} we may and will moreover assume that $\KK(\calH_v)\subseteq C_v$.
We let $(C_{\Gamma}, \omega_\Gamma) = \astred_{v,\Gamma} (C_v, \omega_v)$ and $(A_{\Gamma}, \omega_\Gamma) = \astred_{v,\Gamma} (A_v, \omega_v)$ be the reduced  graph products of the C*-algebras. 
Observe that we now have
\begin{align*}
    A_{\Gamma} \subseteq M_{\Gamma} \subseteq B_{\Gamma} \text{ and } A_{\Gamma}\subseteq C_{\Gamma}\subseteq B_{\Gamma},
\end{align*}
and the states $\omega_\Gamma$ defined through the different graph products coincide.

\begin{lemma}\label{Lem=CGammaNuclear}
$C_\Gamma$ is nuclear.
\end{lemma}
\begin{proof}
The vector $\Omega_v$ is cyclic for $M_v$. Furthermore,  $A_v$ is $\sigma$-weakly dense in $M_v$ by assumption and so $\Omega_v$ is also cyclic for $A_v$. It follows that the GNS-representation $\pi_v$ of $C_v$ with respect to $\omega_v$ is unitarily equivalent with the canonical representation given by the inclusion $C_v \subseteq \bound(\calH_v)$, see \cite[Theorem VIII.5.14 (b)]{conwayCourseFunctionalAnalysis1997}.  We assumed  that  $\Kompact(\calH_v) \subseteq C_v$ and that $C_v$ is nuclear and so we may apply   \cref{thm:nuclearity-graph-productNEW} to conclude that $C_\Gamma$ is nuclear. 
\end{proof}

  We refer to Section \ref{prelim:graph-products} for the definition of $U_\Lambda'$ that is used in the following definition. 
 
\begin{definition}
For $\Lambda \subseteq \Gamma$ we define the C$^\ast$-algebra 
\[ 
    D_{\Lambda} = U_{\Lambda}'(\KK(\calH'(\Lambda)) \otimes \bound(\calH_{\Lambda}))(U'_{\Lambda})^*.  
\]
 The tensor product in the definition of $D_\Lambda$ is understood as the spatial (minimal) tensor product, which is the norm closure of the algebraic tensors acting on the tensor product Hilbert space. In particular $D_\varnothing = \Kompact(\calH_\Gamma)$.
 \end{definition}
 
\begin{lemma}\label{Lem=DIdeal}
Let $v \in \Gamma$. We have $B_\Gamma D_{\Link(v)} B_\Gamma \subseteq D_{\Link(v)}$.
\end{lemma}
\begin{proof}
We note that the proof we give here in particular also works if $\Link(v)$ is empty; though in that case the statement trivially follows from the fact that $D_\varnothing = \Kompact(\calH_\Gamma)$ is an ideal in $\bound(\calH_\Gamma)$.  Take $x \in \bound(\calH_w)$. Then if $w \not \in \Link(v)$ we have that $\calH'(\Link(v))$ is an invariant subspace of $x$ and 
\begin{equation}\label{Eqn=Ideal1}
x = U_{\Link(v)}' ( x  \otimes 1 )  U_{\Link(v)}'^\ast.
\end{equation}
Now suppose that $w  \in \Link(v)$.  
Let $P$ be the orthogonal projection of $\calH'(\Link(v))$ onto $\calH'(\Link(v)) \cap \calH_{ \Link(w)}$.   Then 
\begin{equation}\label{Eqn=Ideal2}
x = U_{\Link(v)}' ( x P^\perp \otimes 1 )  U_{\Link(v)}'^\ast + U_{\Link(v)}' (   P \otimes x )  U_{\Link(v)}'^\ast.
\end{equation}
From the decompositions \eqref{Eqn=Ideal1}, \eqref{Eqn=Ideal2} we see that $x D_{\Link(v)}, D_{\Link(v) } x \subseteq D_{\Link(v)}$. As $B_\Gamma$ is the closed linear span of products of elements in $\bound(\calH_w), w \in \Gamma$ the proof follows. 
\end{proof}

 Denote $P_{\Omega}$ for the orthogonal projection onto $\CC\Omega$.
 
\begin{lemma}\label{lemma:commutator}
    Let $v,w\in \Gamma$. Let $a\in B_{v}$, $b\in B_w$.
    Then 
    \begin{align}
        [a,JbJ] = \begin{cases}
            U_{\Star(v)}'(P_{\Omega}\otimes [a,JbJ])(U_{\Star(v)}')^*, & v=w;\\
             0, & v\not=w. 
        \end{cases}
    \end{align}
    \begin{proof}
        If $v\not=w$ then the result follows from \cite[Proposition 3.3]{caspersGraphProductsOperator2017a}. Suppose $v=w$. Let $\vv_1\in \calW_{\Gamma}'(\Star(v))$ and $\vv_2\in \calW_{\Star(v)}$, and put $\vv=\vv_1\vv_2$. Let $\eta_1\in \mathring{\calH}_{\vv_1}$, $\eta_2\in \mathring{\calH}_{\vv_2}$ be pure tensors and denote $\eta := U_{\Star(v)}'(\eta_1\otimes \eta_{2})\in \mathring{\calH}_{\vv}$. We claim
        \begin{align}\label{eq:action-on-word-star}
            a\eta = \begin{cases}
                U_{\Star(v)}'(\eta_1\otimes (a\eta_2)), & \text{ if } \vv_1 = e;\\
                U_{\Star(v)}'((a\eta_1)\otimes \eta_2), & \text{ if } \vv_1 \not= e.
            \end{cases}
        \end{align}
        Indeed, if $\vv_1=e$ then $\eta_1 =\Omega$ and $\eta = \eta_2$ up to scalar multiplication, so that $U_{\Star(v)}'(\eta_1 \otimes (a\eta_2)) = a\eta_2 = a\eta$. Thus suppose $\vv_1\not=e$. Then it follows that $v\vv_1\in \calW'(\Star(v))$ since $\vv_1\in \calW'(\Star(v))$.
        
        Suppose $v\vv$ is reduced. Then also $v\vv_1$ is reduced, and we have $a\eta = \lambda_{(v,e,e)}(a)\eta$ and $a\eta_1 = \lambda_{(v,e,e)}(a)\eta_1$.  It follows from \cite[Lemma 2.3(iii)]{borstCCAPGraphProducts2024} that
        \begin{align*}
            a\eta &= \lambda_{(v,e,e)}(a)\eta \\
            &= \calQ_{(v\vv_1,\vv_2)}((\lambda_{(v,e,e)}(a)\eta_1)\otimes \eta_2)\\
            &=\calQ_{(v\vv_1,\vv_2)}((a\eta_1)\otimes \eta_2),
        \end{align*}
        and therefore, as $v\vv_1\in \calW'(\Star(v))$ and $\vv_2\in \calW_{\Star(v)}$, we obtain $a\eta = U_{\Star(v)}'((a\eta_1)\otimes \eta_2)$.
        
        Now, suppose $v\vv$ is not reduced. Then also $v\vv_1$ is not reduced as $\vv_1\in \calW'(\Star(v))$ and $\vv_1\not=e$. We have $a\eta = \lambda_{(e,v,e)}(a)\eta + \lambda_{(e,e,v)}(a)\eta$ and $a\eta_1 = \lambda_{(e,v,e)}(a)\eta_1 + \lambda_{(e,e,v)}(a)\eta_1$. Again, using
        \cite[Lemma 2.3(iii)]{borstCCAPGraphProducts2024} we obtain 
        \begin{align*}
            a\eta &= \lambda_{(e,v,e)}(a)\eta + \lambda_{(e,e,v)}(a)\eta \\
            &= \calQ_{(\vv_1,\vv_2)}((\lambda_{(e,v,e)}(a)\eta_1)\otimes \eta_2) + \calQ_{(v\vv_1,\vv_2)}((\lambda_{(e,e,v)}(a)\eta_1)\otimes \eta_2).
        \end{align*} And thus 
        \[
        a\eta = U_{\Star(v)}'(\lambda_{(e,v,e)}(a)\eta_1)\otimes \eta_2) + U_{\Star(v)}'(\lambda_{(e,e,v)}(a)\eta_1)\otimes \eta_2) = U_{\Star(v)}'((a\eta_1)\otimes \eta_2)).
        \]
        This shows \eqref{eq:action-on-word-star}.
        
        We now claim that 
        \begin{align}\label{eq:action-on-word-star-right}
            JbJ\eta = U_{\Star(v)}'(\eta_1\otimes JbJ\eta_2).
        \end{align}
        First, by \cite[Proposition 2.20]{caspersGraphProductKhintchine2021}we observe that $J\eta_1 \in \mathring{\calH}_{\vv_1^{-1}}$, $J\eta_2\in \mathring{\calH}_{\vv_2^{-1}}$  and $J\eta = J\calQ_{(\vv_1,\vv_2)}(\eta_1\otimes \eta_2) = \calQ_{(\vv_2^{-1},\vv_1^{-1})}(J\eta_2\otimes J\eta_1)\in \mathring{\calH}_{\vv^{-1}}$. Furthermore, note that $v\vv_2^{-1} = \vv_2^{-1}v$ and $v\vv_2\in W_{\Star(v)}$.
        
        Suppose that $v\vv^{-1}$ is reduced. Then $v\vv_2^{-1}$ is also reduced. Hence, similar as before we obtain $bJ\eta = \calQ_{(v\vv_2^{-1},\vv_1^{-1})}((bJ\eta_2)\otimes J\eta_1)$. Hence \[
        JbJ\eta = \calQ_{(\vv_1,v\vv_2)}(\eta_1\otimes (JbJ\eta_2)) = U_{\Star(v)}'(\eta_1\otimes (JbJ\eta_2)).
        \]
        
        Now, suppose that $v\vv^{-1}$ is not reduced. Then $v\vv_2^{-1}$ is not reduced. Similar as before we obtain
        \begin{align*}
            bJ\eta = \calQ_{(\vv_2^{-1},\vv_1^{-1})}((\lambda_{(e,v,e)}(b)J\eta_2)\otimes J\eta_1) + \calQ_{(v\vv_2^{-1},\vv_1^{-1})}((\lambda_{(e,e,v)}(b)J\eta_2)\otimes J\eta_1).
        \end{align*}
        Hence 
         \begin{align*}
            JbJ\eta &= \calQ_{(\vv_1,\vv_2)}(\eta_1\otimes (\lambda_{(e,v,e)}(b)J\eta_2)) + \calQ_{(\vv_1,v\vv_2)}(\eta_1\otimes (\lambda_{(e,e,v)}(b)J\eta_2))\\
            &= U_{\Star(v)}'(\eta_1\otimes (J\lambda_{(e,v,e)}(b)J\eta_2))
            + U_{\Star(v)}'(\eta_1\otimes (J\lambda_{(e,e,v)}(b)J\eta_2))\\
            &= U_{\Star(v)}'(\eta_1\otimes (JbJ\eta_2)),
        \end{align*}
        which shows \eqref{eq:action-on-word-star-right}. Now, combining 
        \eqref{eq:action-on-word-star} with \eqref{eq:action-on-word-star-right} we obtain
        \begin{align*}
            [a,JbJ]\eta = \begin{cases}
                U_{\Star(v)}'(\eta_1\otimes ([a,JbJ]\eta_2)),
                & \text{ if } \vv_1 = e;\\
                0, & \text{ if } \vv_1 \not=e
            \end{cases}
        \end{align*}
        and the statement follows.
    \end{proof}
\end{lemma}

\begin{lemma}\label{Lem=CommKompact}
For  $v,w \in \Gamma, c \in C_v, a \in A_w$ we have  $[c, J a J ] \in D_{\Link(v)}$.
\end{lemma}
\begin{proof} 
If $v \not = w$ it actually holds since  by \cite[Proposition 2.3]{caspersGraphProductsOperator2017a} $[c, JaJ] = 0$. So assume $v = w$.   \cref{lemma:commutator} gives that 
\begin{equation}\label{Eqn=Unravel1}
[c, J a J] = U_{\Star(v)}'(P_{\Omega}\otimes [c, J a J])(U_{\Star(v)}')^*.
\end{equation} 
In what follows we will use the decomposition of Section \ref{prelim:graph-products} applied to $\Link(v)$ as a subgraph of $\Star(v)$, opposed to $\Link(v)$ as a subgraph of $\Gamma$, and correspondingly define the Hilbert space $\calH'(\Link(v))$ with respect to this inclusion.  
We thus have a natural unitary 
  \[
  U_{\Link(v)}'': \calH'(\Link(v)) \otimes \calH_{\Link(v)} \rightarrow \calH_{\Star(v)}.
  \]
  Further as $v$ commutes with all vertices in $\Link(v)$ it follows that with respect to this decomposition we have $\calH'(\Link(v)) = \calH_v$. So  
\[ 
U_{\Link(v)}'':  \calH_v \otimes \calH_{\Link(v)} \rightarrow \calH_{\Star(v)}.
\]
For $x \in \bound(\calH_v)$  we get that
\begin{equation}\label{Eqn=Unravel2}
x = U_{\Link(v)}''( x \otimes {\rm 1}  )(U_{\Link(v)}'')^*.  
\end{equation}
 Set the unitary
\[
U_{v}'' := U_{\Star(v)}' (1  \otimes U_{\Link(v)}''): \calH'(\Star(v)) \otimes \calH_v \otimes \calH_{\Link(v)} \rightarrow \calH_\Gamma. 
\]
Combining \eqref{Eqn=Unravel1} and \eqref{Eqn=Unravel2} we have
\[
[c,JaJ] = 
U_{v}''  (P_\Omega \otimes  [c,JaJ]  \otimes 1)  U_{v}''^\ast,
\]
where $[c,JaJ]$ on the left hand side acts on $\calH_\Gamma$ and on the right hand side on $\calH_v$. As we assumed $[c, JaJ ] \in \Kompact(\calH_v)$ it follows that $[c,JaJ] $ is contained in 
\[
  U_{v}''  (\Kompact(\calH'(\Star(v))) \otimes  \Kompact(\calH_v)  \otimes 1)  U_{v}''^\ast = U_{\Link(v)}'  (\Kompact(\calH'(\Link(v))) \otimes 1 )  U_{\Link(v)}'^\ast \subseteq D_{\Link(v)},
\]
and thus the lemma is proved.  \end{proof}

Let $Q \subseteq M_\Gamma$ be an amenable von Neumann subalgebra. As explained in \cite[p. 228]{ozawaPrimeFactorizationResults2004} there exists a conditional expectation $\Psi_Q: \bound( \calH_\Gamma ) \rightarrow Q'$ that is {\it  proper} in the sense that for any $a \in \bound( \calH_\Gamma )$ we have that $\Psi_Q(a)$ is in the $\sigma$-weak closure of
\[
\Conv \left\{   u a u^\ast \mid u \in \mathcal{U}(Q) \right\},
\]
where $\Conv$ denotes the convex hull.

\begin{lemma}\label{lemma:kernel-proper-CE}
   Let $Q \subseteq M_\Gamma$ be an amenable von Neumann subalgebra.  If there is $\Lambda\subseteq \Gamma$ such that $Q\not\prec_{M_{\Gamma}}M_{\Lambda}$, then   $D_\Lambda$ is contained in $\ker \Psi_{Q}$.
\end{lemma}
    \begin{proof}
Let $p \in  \KK(\calH'(\Lambda))$ be a finite rank projection. We first claim that 
\[
U_{\Lambda}'( p \otimes 1)U'^*_{\Lambda} \in  \ker \Psi_{Q}.
\]
We prove this claim by contradiction so suppose that $d := \Psi_Q( U_{\Lambda}'( p \otimes 1)U'^*_{\Lambda} ) \not = 0$. First observe that for $a \in M_\Lambda$ we have
\[
  J a J   =  U_{\Lambda}'  (1 \otimes  J_\Lambda a J_\Lambda)  U_{\Lambda}'^\ast,     
\]
where $J_\Lambda$ is the modular conjugation operator of $M_\Lambda$ acting on $\mathcal{H}_\Lambda$. It follows in particular that 
\[
 (J M_\Lambda J)' =     U_{\Lambda}'(\bound(\calH'(\Lambda)) \bar{\otimes} M_\Lambda)U_{\Lambda}'^*. 
\]
Any $u \in \mathcal{U}(Q)$ commutes with $M_\Gamma' = J M_\Gamma J$ and so certainly it commutes with $J M_\Lambda J$. As $\Psi_Q$ is proper we find that $d$ as defined above thus commutes with $J M_\Lambda J$. Thus $d \in  U_{\Lambda}'(\bound(\calH'(\Lambda)) \bar{\otimes} M_\Lambda)U'^*_{\Lambda}$.
Let ${\rm Tr}$ the trace on  $\bound(\calH'(\Lambda))$ and let $\Phi_\Lambda$ be the center valued trace of $M_\Lambda$ onto $\mathcal{Z}(M_\Lambda) = M_\Lambda \cap M_\Lambda'$.
Using again that $\Psi_Q$ is proper we find by lower semi-continuity \cite[Theorem VII.11.1]{takesakiTheoryOperatorAlgebras2002} that for any normal (necessarily tracial) state $\tau$ on the center $\mathcal{Z}(M_\Lambda)$ we have
\[
({\rm Tr} \otimes (\tau \circ \Phi_{\Lambda}))  (  U_{\Lambda}'^\ast d  U_{\Lambda}') \leq ({\rm Tr} \otimes (\tau \circ \Phi_{\Lambda}))(p\otimes 1) < \infty. 
\]
Let $e$ be a spectral projection of $d$ corresponding to the interval $[\Vert d \Vert/2, \Vert d \Vert ]$. Then 
\[
({\rm Tr} \otimes (\tau \circ  \Phi_{\Lambda}))  (  U_{\Lambda}'^\ast e  U_{\Lambda}')  \leq 2 ({\rm Tr} \otimes (\tau \circ \Phi_{\Lambda}))  (  U_{\Lambda}'^\ast d  U_{\Lambda}') < \infty. 
\]
Thus it follows that $({\rm Tr} \otimes \Phi_{\Lambda})  (  U_{\Lambda}'^\ast e  U_{\Lambda}') < \infty$. 
Then $\calK := e \calH_\Gamma$ is a $Q$-$M_\Lambda$ sub-bimodule of $\calH_\Gamma$ with $\dim_{M_\Lambda}(\calK) < \infty$ and $\calH_\Gamma$ is the standard representation Hilbert space of $M_\Gamma$. It thus follows from \cref{Dfn=Intertwine} \eqref{Item=Intertwine3} that $Q \prec_{M_\Gamma} M_\Lambda$. This contradicts the assumptions and the claim is proved. 

Taking linear spans and closures it thus follows from the previous paragraph that 
\[
U_{\Lambda}'( \KK(\calH'(\Lambda)) \otimes 1)U'^*_{\Lambda} \subseteq  \ker \Psi_{Q}.
\]
Using the multiplicative domain of $\Psi_Q$ it follows then that 
 \[
U_{\Lambda}'( \KK(\calH'(\Lambda)) \otimes \bound(\calH_\Lambda))U'^*_{\Lambda} \subseteq \ker \Psi_{Q}.
\]
    This concludes the proof.

\end{proof}

\begin{lemma}\label{Lem=ComKernel}
   Let $Q \subseteq M_\Gamma$ be an amenable von Neumann subalgebra. Assume that for every $v \in \Gamma$ we have $Q\not\prec_{M_{\Gamma}}M_{\Link(v)}$. Then we have  $[C_\Gamma, JA_\Gamma J] \subseteq  \ker \Psi_Q$. 
\end{lemma}
\begin{proof} 
 The commutator  $[C_\Gamma, J A_v J]$  is contained in the closed linear span of the sets 
 \[
 C_\Gamma [C_w, J A_v J] C_\Gamma, \qquad v,w  \in \Gamma.
 \]
 We have, as $C_\Gamma \subseteq B_\Gamma$,  by \cref{Lem=CommKompact} and \cref{Lem=DIdeal} that 
\[
C_\Gamma [C_w, J A_v J] C_\Gamma \subseteq B_\Gamma  D_{\Link(v)} B_\Gamma  \subseteq D_{\Link(v)}.
\]
By \cref{lemma:kernel-proper-CE} we see that $D_{\Link(v)}, v \in \Gamma$ is contained in the kernel of $\Psi_Q$.   We thus conclude that $[C_\Gamma, J A_v J]$  is contained in  $\ker\Psi_Q$.

Now  $[C_\Gamma, J A_\Gamma J]$ is contained in the closed linear span of the sets 
\[
J A_\Gamma J [C_\Gamma, J A_v J] J A_\Gamma J, \qquad v \in \Gamma.
\]
Note $J A_\Gamma J$ is contained in   $M_\Gamma'$ so certainly in $Q'$. As $\Psi_Q$ is a $Q'$-bimodule map it follows that $J A_\Gamma J [C_\Gamma, J A_v J]  J A_\Gamma J$ is contained in $\ker\Psi_Q$. This finishes the proof.  
\end{proof}

\begin{lemma}\label{Lem=MinBoundAmenable}
   Let $Q \subseteq M_\Gamma$ be an amenable von Neumann subalgebra. Assume that for every $v \in \Gamma$ we have $Q\not\prec_{M_{\Gamma}}M_{\Link(v)}$. 
    The map 
    \begin{equation}\label{Eqn=ThetaMinBound} 
    \begin{split}
        \Theta: A_{\Gamma} \otimes J A_{\Gamma} J  &\to \bound(\calH_\Gamma) \\
         a\otimes J  b J  &\mapsto \Psi_Q(a J bJ).
\end{split}
    \end{equation}
    is  continuous with respect to the minimal tensor norm. 
    \begin{proof} 
     Observe that $\Psi_Q$ is a $Q'$-bimodule map and we have $J A_\Gamma J \subseteq M_\Gamma'  \subseteq  Q'$. It thus follows from  \cref{Lem=ComKernel} that for $x \in C_\Gamma$ and $y \in JA_\Gamma J$ we have 
    \[
      \Psi_Q(x) y = \Psi_Q(xy) = \Psi_Q(yx + [x, y]) = \Psi_Q(yx) = y \Psi_Q(x).   
    \]
     So $\Psi_Q(   C_{\Gamma}   ) \subseteq (JA_\Gamma J)' =  M_\Gamma$.  Now consider the composition of maps, see \cite[Theorem 3.3.7 and 3.5.3]{brownAlgebrasFiniteDimensionalApproximations2008},  
\[
\widetilde{\Theta}:  C_{\Gamma} \otimes_{{\rm max}} J A_{\Gamma} J \rightarrow^{\Psi_Q \otimes {\rm Id}}  M_\Gamma  \otimes_{{\rm max}} J A_{\Gamma} J \rightarrow^m \bound(\calH), 
\]
where $m$ is the multiplication map. Note that $C_\Gamma$ is nuclear by \cref{Lem=CGammaNuclear}. Thus $ C_{\Gamma} \otimes_{{\rm max}} J A_{\Gamma} J  =  C_{\Gamma} \otimes_{{\rm min}} J A_{\Gamma} J $. Then the restriction of $\widetilde{\Theta}$ to $A_{\Gamma} \otimes_{{\rm min}} J A_{\Gamma} J $  gives the map $\Theta$.

    \end{proof}
\end{lemma}

The following is one of the core theorems of this paper. The result has been proved in the tensor product case in \cite[Theorem 5.1]{houdayerUniquePrimeFactorization2017}.

\begin{theorem}\label{Thm=KeyAlternatives}
 Let $\Gamma$ be a simple graph.  Let $(M_\Gamma, \tau) = \ast_{v, \Gamma} (M_v, \tau_v)$ be a graph product of finite von Neumann algebras $M_v$ ($\not=\CC$) that satisfy condition strong (AO) and have separable preduals.  Let $Q \subseteq M_\Gamma$ be a diffuse von Neumann subalgebra. At least one of the following holds:
    \begin{enumerate}
    \item \label{Item=MainOne}The relative commutant $Q' \cap M_\Gamma$ is amenable;
    \item \label{Item=MainTwo} There exists a non-empty $\Gamma_0  \subseteq \Gamma$ such that $\Link(\Gamma_0) \not = \varnothing$  and  $Q \prec_{M_\Gamma} M_{\Gamma_0}$.
    \end{enumerate}
    \begin{proof}
    We first show we can reduce it to the case that $Q$ is amenable.
    Indeed, suppose we have proven that every amenable diffuse subalgebra $Q_0\subseteq M_{\Gamma}$ satisfies (1) or (2). Let $Q\subseteq M_{\Gamma}$ be an arbitrary diffuse subalgebra. Then by \cite[Corollary 4.7]{houdayerUniquePrimeFactorization2017} there is an amenable diffuse von Neumann subalgebra $Q_0\subseteq Q$ such that for subgraphs $\Lambda\subseteq \Gamma$ we have $Q_0\not\prec_{M_{\Gamma}}M_{\Lambda}$ whenever $Q\not\prec_{M_{\Gamma}}M_{\Lambda}$. If $Q$ does not satisy (2), then neither does $Q_0$. Hence $Q_0$ satisfies (1), so $Q_0'\cap M_{\Gamma}$ is amenable. Hence also the subalgebra $Q' \cap M_{\Gamma}\subseteq Q_0'\cap M_{\Gamma}$ is amenable, i.e. $Q$ satisfies (1), which shows the reduction.

    We now prove the statement with the notation introduced in this section. Assume \eqref{Item=MainTwo} does not hold and we shall prove  \eqref{Item=MainOne}.       
    By assumption    for $\Lambda\subseteq \Gamma$ with $\Link(\Lambda)\not=\varnothing$ we have $Q\not\prec_{M_{\Gamma}}M_{\Lambda}$.
    In particular we have for all $v \in \Gamma$ with $\Link(v)$ non-empty that $v$ is contained in $\Link(\Link(v))$ and so $Q\not\prec_{M_{\Gamma}} M_{\Link(v)}$. If $\Link(v)$ is empty then  $M_{\Link(v)} = \mathbb{C}$ and so $Q\not\prec_{M_{\Gamma}} M_{\Link(v)}$ as $Q$ is diffuse. 
    It follows now from \cref{Lem=MinBoundAmenable} that $\Theta$ defined in \eqref{Eqn=ThetaMinBound} is bounded for the minimal tensor norm. 

     Each $A_v$ is exact being included in the nuclear C$^\ast$-algebra $C_v$. Therefore the C$^\ast$-algebra $A_{\Gamma}$ is exact by \cite[ Corollary 3.17]{caspersGraphProductsOperator2017a}. Furthermore, the inclusions $A_{\Gamma}\subseteq M_{\Gamma}$  and $JA_{\Gamma}J\subseteq M_{\Gamma}'$ are $\sigma$-weakly dense.   

     The conclusions of the previous two paragraphs show that the assumptions of  \cite[Lemma 2.1]{OzawaKurosh} are satisfied and this lemma concludes that  $Q'\cap M_{\Gamma}$ is amenable.
  
    \end{proof}
\end{theorem}

We recall the following lemma about relative commutants which we shall use without further reference. 

\begin{lemma}[Lemma 3.5 of \cite{vaesExplicitComputationsAll2008}]\label{Lem=RelativeCom}
If $A  \subseteq  1_A M1_A , B \subseteq 1_B M 1_B$ are von Neumann subalgebras and $A \prec_M B$, then $B' \cap 1_B M 1_B \prec_M A' \cap 1_A M 1_A$.
\end{lemma}

\subsection{Unique rigid graph product decomposition}
We will prove our main result \cref{thm:rigid-graph-decomposition} which asserts for a graph product $M_{\Gamma} = *_{v,\Gamma}(M_{v},\tau_v)\in\Crigid$ with $M_v\in \Cvert$ that we can retrieve the rigid graph $\Gamma$ and retrieve the vertex von Neumann algebras $M_v$ up to stable isomorphism. To prove the result we need the following lemma.
\begin{lemma}\label{lemma:vNa-isomorphism-graph-products}
Let $\Gamma$ be a simple graph. For $v\in \Gamma$, let $M_v,N_v$ be II$_1$-factors and put $M_{\Gamma} = *_{v,\Gamma}(M_v,\tau_v)$ and $N_{\Gamma} = *_{v,\Gamma}(N_v,\widetilde{\tau}_v)$. 
Suppose $\iota:N_{\Gamma}\to M_{\Gamma}$ is a $*$-isomorphism and for $v\in \Gamma$ we have
\begin{align*}
    \iota(N_{v})\prec_{M_{\Gamma}}M_{v}& &and&
    &M_{v}\prec_{M_{\Gamma}}\iota(N_{v}).
\end{align*} 
Then the following holds true:
\begin{enumerate}
    \item \label{Item=vNa-isomorphism-graph-products:Star} For $v\in \Gamma$ there is a unitary $u_v\in M_{\Gamma}$ such that $u_v^*\iota(N_{\Star(v)})u_v = M_{\Star(v)}$.
    \item \label{Item=vNa-isomorphism-graph-products:union}  Let $\Lambda_0\subseteq \Lambda\subseteq \Gamma$ be subgraphs such that $\iota(N_{\Lambda}) = M_{\Lambda}$. Then $\iota(N_{\Lambda\cup \Link_{\Gamma}(\Lambda_0)}) = M_{\Lambda\cup \Link_{\Gamma}(\Lambda_0)}$.
    \item \label{Item=vNa-isomorphism-graph-products:path} Let $P = (v_1,\ldots, v_n)$ be a path in $\Gamma$ and denote $\Gamma_0:= \bigcup_{i=1}^n \Star(v_i)$. If there exist $1\leq j\leq n$ and a subgraph $\Lambda\subseteq \Gamma_0$ such that $v_j\in \Lambda$ and $\iota(N_{\Lambda}) = M_{\Lambda}$, then $\iota(N_{\Gamma_0}) = M_{\Gamma_0}$.
    \item \label{Item=vNa-isomorphism-graph-products:connected-component} Let $\Gamma_0$ be a connected component of $\Gamma$. If there is a non-empty subgraph $\Lambda\subseteq \Gamma_0$ with $\iota(N_{\Lambda})=M_{\Lambda}$ then $\iota(N_{\Gamma_0}) = M_{\Gamma_0}$.
\end{enumerate}
\begin{proof}  
\eqref{Item=vNa-isomorphism-graph-products:Star} As $\iota(N_{v})\prec_{M_{\Gamma}} M_{v}$ and $\iota(N_{v})\not\prec_{M_{\Gamma}} M_{\varnothing}$ (since $N_{v}$ diffuse), and since $\iota(N_{v})$ and $\iota(N_{v})' \cap M_{\Gamma}$ ($=\iota(N_{\Link(v)})$) are factors, we obtain by \cref{lemma:unitary-inclusion-in-star-for-vNa} a unitary $u_v\in M_{\Gamma}$ such that $u_v^*\iota(N_{v})u_v\subseteq M_{\Star(v)}$.
    By assumption $M_{v}\prec_{M_{\Gamma}} \iota(N_{v})$ so that $M_{v}\prec_{M_{\Gamma}} u_v^*\iota(N_{v})u_v$. If $u_v^*\iota(N_{v})u_v\prec_{M_{\Gamma}} M_{\Link(v)}$ then  $u_v^*\iota(N_{v})u_v\prec_{M_{\Gamma}}^s M_{\Link(v)}$ by \cref{Lem=StableEmbedding} \eqref{Item=StableEmbedding:condition}, since $M_v'\cap M_{\Gamma}$ is a factor. Consequently, by \cref{Lem=StableEmbedding} \eqref{Item=StableEmbedding:transative} we obtain $M_v\prec_{M_{\Gamma}} M_{\Link(v)}$, which gives a contradiction by \cref{Lem=Selfembedding}.  We conclude that $u_v^*\iota(N_{v})u_v\not\prec_{M_{\Gamma}} M_{\Link(v)}$.     Now, we are in the situation that  $u_v^*\iota(N_{v})u_v\subseteq M_{\Star(v)}$ and $u_v^*\iota(N_{v})u_v\not\prec_{M_{\Gamma}} M_{\Link(v)}$. We apply \cref{prop:embed}\eqref{Item=Embed1b}  to $\Lambda = \Star(v)$ and $\{ \Lambda_j \}_{j \in \mathcal{J}} = \{ \Link(v) \}$ so there is only one index in $\mathcal{J}$. In this case $\Lambda_{{\rm emb}} = \Star(v)$.  So  \cref{prop:embed}\eqref{Item=Embed1b} yields that $\Nor_{M_{\Gamma}}(u_v^*\iota(N_{v})u_v)\subseteq M_{\Star(v)}$, hence $u_v^*\iota(N_{\Star(v)})u_v \subseteq M_{\Star(v)}$.

By symmetry there is also a unitary $\widetilde{u}_v\in M_{\Gamma}$ such that $\widetilde{u}_v^*M_{\Star(v)}\widetilde{u}_v\subseteq \iota(N_{\Star(v)})$.
Hence 
\begin{align}\label{eq:unitary-inclusions-star}
u_v^*\widetilde{u}_v^*M_{\Star(v)}\widetilde{u}_vu_v\subseteq u_v^*\iota(N_{\Star(v)})u_v \subseteq M_{\Star(v)}.
\end{align} 
Hence, since $M_{\Star(v)}\not\prec_{M_{\Gamma}} M_{\widetilde{\Lambda}}$ for any strict subgraph $\widetilde{\Lambda}\subsetneq \Star(v)$ we obtain by \cref{prop:embed}\eqref{Item=Embed1c},   and the final remark of \cref{prop:embed} applied to $\Lambda = \Star(v)$ and  $\Lambda_{{\rm emb}} = \Star(v)$,   that $\widetilde{u}_vu_v\in M_{\Star(v)}$. From this we conclude that the inclusions in \eqref{eq:unitary-inclusions-star} are in fact equalities so $u_v^*\iota(N_{\Star(v)})u_v = M_{\Star(v)}$.
    
\eqref{Item=vNa-isomorphism-graph-products:union} Let $\Lambda_0\subseteq \Lambda$ be a subgraph. Then $\iota(N_{\Lambda_0})\subseteq \iota(N_{\Lambda})=M_{\Lambda}$ and by the assumptions $\iota(N_{\Lambda_0})\not\prec_{M_{\Gamma}}M_{\widetilde{\Lambda}}$ for any strict subgraph $\widetilde{\Lambda}\subsetneq \Lambda_0$.
    Hence, by \cref{prop:embed}\eqref{Item=Embed1b},   and the final remark of \cref{prop:embed},  we obtain that $\iota(N_{\Link(\Lambda_0)})\subseteq \Nor_{M_{\Gamma}}(\iota(N_{\Lambda_0}))'' \subseteq M_{\Lambda\cup \Link(\Lambda_0)}$.
    Thus $\iota(N_{\Lambda\cup \Link(\Lambda_0)})\subseteq M_{\Lambda\cup \Link(\Lambda_0)}$. By symmetry we also obtain that $M_{\Lambda\cup \Link(\Lambda_0)}\subseteq \iota(N_{\Lambda\cup \Link(\Lambda_0)})$ so we get the equality.

\eqref{Item=vNa-isomorphism-graph-products:path} As $v_{j}\in \Lambda$ and $\iota(N_{\Lambda}) = M_{\Lambda}$, using \eqref{Item=vNa-isomorphism-graph-products:union}  we obtain that $\iota(N_{\Lambda\cup \Star(v_j)}) = \iota(N_{\Lambda\cup \Link(v_j)}) = M_{\Lambda\cup \Link(v_j)} = M_{\Lambda\cup \Star(v_j)}$. Now for $1\leq i\leq n$ with $|i-j|=1$ we have $v_i \in \Lambda\cup \Star(v_j)$. Hence, applying \eqref{Item=vNa-isomorphism-graph-products:union} again we obtain 
    $\iota(N_{\Lambda\cup \Star(v_j)\cup \Star(v_i)}) =  M_{\Lambda\cup \Star(v_j)\cup \Star(v_i)}$. Repeating the same argument at most $n$ times we obtain $\iota(N_{\Gamma_0}) = M_{\Gamma_0}$.
        
\eqref{Item=vNa-isomorphism-graph-products:connected-component} Let $P = (v_1,\ldots,v_n)$ be a path in $\Gamma$ traversing all vertices in $\Gamma_0$.  Then $\Gamma_0$ is equal to $\bigcup_{i=1}^n \Star(v_i)$. Now since $\Lambda\subseteq \Gamma_0$ is non-empty, we can choose $1\leq j\leq n$ s.t. $v_j\in \Lambda$. Now by \eqref{Item=vNa-isomorphism-graph-products:path} we obtain $\iota(N_{\Gamma_0}) = M_{\Gamma_0}$.
\end{proof}
\end{lemma}

\begin{theorem}\label{thm:rigid-graph-decomposition}
 Let $\Gamma$ be a rigid graph. For $v \in \Gamma$, let $M_v$ be  von Neumann algebras in class $\Cvert$.
Let $M_\Gamma = \ast_{v, \Gamma} (M_{v},\tau_v)$. Suppose $M_{\Gamma}= \ast_{w, \Lambda} (N_{w},\tau_w)$ for another rigid graph $\Lambda$ and other von Neumann algebras $N_w\in\Cvert$ for $w\in \Lambda$. 
Then there is a graph isomorphism $\alpha: \Gamma \rightarrow \Lambda$, and for each $v\in \Gamma$ there is a unitary $u_v\in M_{\Gamma}$ and a real number $0<t_v<\infty$ such that
\begin{align}
    M_{\Star(v)}= u_v^*N_{\Star(\alpha(v))}u_v& &and&   &M_v \simeq N_{\alpha(v)}^{t_v}.
\end{align}
Furthermore, for each vertex $v\in \Gamma$ its connected component $\Gamma_v\subseteq \Gamma$ satisfies $M_{\Gamma_v} = u_v^*N_{\alpha(\Gamma_v)}u_v$.
\begin{proof} 
First we construct the graph isomorphism $\alpha$. Take $v \in \Gamma$. As the vertex von Neumann algebras are factors we have  by \cite[Corollary 2.28]{caspersGraphProductsOperator2017a},
\[
M_{\Link(v)} ' \cap M = M_{\Link(\Link(v) )} = M_{v}.
\]
In particular $M_{\Link(v)} ' \cap M$ is non-amenable. Therefore \cref{Thm=KeyAlternatives} implies that there exists $\Lambda_0  \subseteq \Lambda$ such that $M_{\Link(v)} \prec_{N_{\Gamma}} N_{\Lambda_0}$ and $\Link(\Lambda_0) \not = \varnothing$. Thus taking relative commutants  (\cref{Lem=RelativeCom}) we find that  $N_{\Link(\Lambda_0)} \prec_{M_{\Gamma}} M_{v}$.  

So we have shown that for every $v \in \Gamma$ there exists a subgraph $\alpha(v) \subseteq \Lambda$ that occurs as the link of a set such that $N_{\alpha(v)} \prec_{M_{\Gamma}} M_{v}$.  Conversely, by symmetry, for every $w \in \Lambda$ there exists $\beta(w) \subseteq \Gamma$ that occurs as the link of a set such that $M_{\beta(w)} \prec_{M_{\Gamma}} N_{w}$.

Let again  $v \in \Gamma$. Then for any $w \in \alpha(v)$ we have $N_w \prec_{M_{\Gamma}}  M_{v}$ and consequently as $N_w'\cap M_{\Gamma}$ is a factor  $N_w \prec_{M_{\Gamma}}^s  M_{v}$, see \cref{Lem=StableEmbedding}\eqref{Item=StableEmbedding:condition}. 
Therefore, by transitivity of stable embeddings, i.e. \cref{Lem=StableEmbedding}\eqref{Item=StableEmbedding:transative}, we find $M_{\beta(w)}  \prec_{M_{\Gamma}}  M_{v}$. Hence for any $v' \in \beta(w)$ we have $M_{v'}  \prec_{M_{\Gamma}}  M_{v}$.  But then by \cref{Lem=Selfembedding} we see that $v' = v$. Hence $\beta(w) = v$ for any $w \in \alpha(v)$ and in particular is a singleton set. So we have proved that for $v \in \Gamma$ we have $\beta(\alpha(v) ) :=  \bigcup_{w \in \alpha(v)} \beta(w)     = v$ and by symmetry for $w \in \Lambda$ we have $\alpha(\beta(w)) = w$. But this can only happen if  the values of $\alpha$   and $\beta$ are singletons and $\alpha$ and $\beta$ are inverses of each other. 

If $v \in \Gamma$ then we know that $N_{\alpha(v)}\prec_{M_{\Gamma}} M_v$ and  $M_v \prec_{M_{\Gamma}} N_{\alpha(v)}$. Taking relative commutants, using again factoriality of the vertex von Neumann algebras,  we find
\begin{align*}
  M_{\Link(v)}   \prec_{M_{\Gamma}}   N_{\Link(\alpha(v))}, \qquad   N_{\Link(\alpha(v))}  \prec_{M_{\Gamma}}    M_{\Link(v)}.
\end{align*}
Now take $v' \in \Link(v)$ so that the first of these embeddings gives  $M_{v'}  \prec_{M_{\Gamma}}   N_{\Link(\alpha(v))}$, hence $M_{v'}  \prec_{M_{\Gamma}}^s  N_{\Link(\alpha(v))}$ by \cref{Lem=StableEmbedding}\eqref{Item=StableEmbedding:condition}. Then again by \cref{Lem=StableEmbedding}\eqref{Item=StableEmbedding:transative} we obtain
$N_{\alpha(v')}  \prec_{N_{\Gamma}}   N_{\Link(\alpha(v))}$.  This then implies by   \cref{Lem=Selfembedding} that
$\alpha(v')  \in    \Link(\alpha(v))$. So we conclude that $\alpha$ preserves edges. Similarly $\beta$ preserves edges, and it follows that $\alpha:\Gamma\to \Lambda$ is a graph isomorphism.\\

Since $\Gamma\simeq \Lambda$ we obtain by \cref{lemma:vNa-isomorphism-graph-products}\eqref{Item=vNa-isomorphism-graph-products:Star} that for each $v\in \Gamma$ there is a unitary $u_v\in M_{\Gamma}$ such that $u_{v}^*N_{\Star(\alpha(v))}u_v= M_{\Star(v)}$.
Consider the $*$-isomorphism $\iota_{v} := \Ad_{u_v^*}: N_{\Gamma}\to M_{\Gamma}$ which satisfies $\iota_v(N_{\Star(\alpha(v))}) = M_{\Star(v)}$. Then by \cref{lemma:vNa-isomorphism-graph-products}\eqref{Item=vNa-isomorphism-graph-products:connected-component} we obtain for the connected component $\Gamma_v\subseteq \Gamma$ of $v$ that $u_{v}^*N_{\Gamma_v}u_v = \iota_{v}(N_{\Gamma_v}) = M_{\Gamma_v}$.

We show the isomorphism of vertex von Neumann algebras up to amplification. Let $w\in \Gamma$.
Since $\iota_w(N_{\Star(\alpha(w))}) = M_{\Star(w)}$ and since $\iota_w(N_{\Link(\alpha(w))})'\cap M_{\Star(w)} = \iota_w(N_{\alpha(w)})$ is non-amenable, we obtain by \cref{Thm=KeyAlternatives} that $\iota_w(N_{\Link(\alpha(w))})\prec_{M_{\Star(w)}} M_{\Gamma_1}$ for some subgraph $\Gamma_1\subseteq \Star(w)$ with $\Link_{\Star(w)}(\Gamma_1)\not=\varnothing$. Thus, by \cref{Lem=RelativeCom} we obtain $M_{\Link(\Gamma_1)}\prec_{M_{\Star(w)}} \iota_w(N_{\alpha(w)})$.
Let $v \in \Link(\Gamma_1)$ (which is non-empty). Then $M_{v}\prec_{M_{\Star(w)}} \iota_w(N_{\alpha(w)})$ and, as before, $\iota_w(N_{\alpha(w)})\prec_{M_{\Gamma}}^s M_w$. Hence $M_{v}\prec_{M_{\Gamma}} M_{w}$ and so $v=w$ by \cref{Lem=Selfembedding}. Therefore $M_{w}\prec_{M_{\Star(w)}} \iota_w(N_{\alpha(w)})$. Analogously, we obtain $\iota_w(N_{\alpha(w)})\prec_{M_{\Star(w)}} M_{w}$.

Since we are dealing with II$_1$-factors these embeddings are also with expectation, i.e. $\iota_w(N_{\alpha(w)})\\ \preceq_{M_{\Star(w)}} M_{w}$ as in \cite[Definition 4.1]{houdayerUniquePrimeFactorization2017}.
Thus, since $M_{\Star(w)} = M_{w}\overline{\otimes} M_{\Link(w)}$ we obtain by \cite[Lemma 4.13]{houdayerUniquePrimeFactorization2017} 
non-zero projections $p_w,q_w\in M_{\Star(w)}$ and a partial isometry $v_w\in M_{\Star(v)}$ with $v_w^*v_w =p_w$ and $v_wv_w^* =q_w$ and a subfactor $P_w\subseteq q_w\iota_w(N_{\alpha(w)})q_w$ so that
\begin{align*}
    q_w\iota_w(N_{\alpha(w)})q_w = v_wM_wv_w^* \overline{\otimes} P_w, &\quad& v_wM_{\Link(w)}v_w^* = P_w\overline{\otimes} q_w\iota_w(N_{\Link(\alpha(w))})q_w.
\end{align*}
Since $N_{w}$ is prime, so is $q_w\iota_w(N_{\alpha(w)})q_w$. Hence, as $v_wM_{w}v_w^*$ is a II$_1$-factor, we obtain that $P_w$ is a factor of type I$_n$ for some $n\in\NN$. We conclude that $N_{\alpha(w)}$ is isomorphic to some amplification of $M_{w}$. \\
\end{proof}
\end{theorem}

We state two corollaries that follow from \cref{thm:rigid-graph-decomposition}. The following result tells us when a rigid graph product $M_{\Gamma}$ can decompose as graph product over another rigid graph $\Lambda$. 

\begin{corollary}
    Let $\Gamma, \Lambda$ be rigid graphs. Let $M_{\Gamma} = *_{v,\Gamma}(M_v,\tau_v)$ be the graph product of factors $M_v\in \Cvert$. The following are equivalent:
    \begin{enumerate}
        \item We can write $\Gamma = *_{w,\Lambda}\Gamma_{w}$ for some non-empty graphs $\Gamma_{w}, w\in \Lambda$;
        \item We can write $M_{\Gamma} = *_{w,\Lambda}(M_{w},\tau_w)$ for some factors $M_{w}\in \Crigid, w\in \Lambda$.
    \end{enumerate}
    \begin{proof}
        Suppose we can write $\Gamma = *_{w,\Lambda}\Gamma_w$ for non-empty graphs $\Gamma_w$ for $w\in \Lambda$. Note that $\Gamma_w$ is rigid by \cref{lemma:graph-product-of-rigid-graphs}.
        Now by \cref{remark:graph-product-of-graphs-consistency} we have
        $M_{\Gamma} = *_{w,\Lambda}(M_{w},\tau_w)$ where $M_w:= M_{\Gamma_w}\in \Crigid$.
        
        For the other direction, suppose $M_{\Gamma} = *_{w,\Lambda}(M_w,\tau_w)$ for some $M_w\in \Crigid$ for $w\in \Lambda$. Then, by definition of $\Crigid$,  there are non-empty, rigid graphs $\Gamma_w$ and factors $N_v\in \Cvert$ for $v\in \Gamma_w$ such that $M_w = *_{v,\Gamma_w}(N_v,\tau_v)$ for $w\in \Lambda$.
        Hence, by \cref{remark:graph-product-of-graphs-consistency} we obtain $M_{\Gamma} = *_{v,\Gamma_{\Lambda}}(N_v,\tau_v)$. Since $\Gamma_{\Lambda}$ is rigid by \cref{lemma:graph-product-of-rigid-graphs}, we obtain by \cref{thm:rigid-graph-decomposition} that $\Gamma \simeq \Gamma_{\Lambda} = *_{w,\Lambda}\Gamma_{w}$.
    \end{proof}
\end{corollary}

The following corollary states a unique prime factorization for the class $\Ccomplete$. This result recovers the result of \cite{houdayerUniquePrimeFactorization2017} for a slightly smaller class. 
\begin{corollary}\label{corollary:rigidy-for-complete-graphs}
    Any von Neumann algebra $M\in \Ccomplete$ can decompose as tensor product 
    \begin{align}\label{eq:UPF-class-Ccomplete}
        M = M_1\overline{\otimes} \cdots \overline{\otimes} M_m
    \end{align}
    for some $1\leq m < \infty$ and prime factors $M_1,\ldots, M_m\in \Cvert$. 
    
    Furthermore, suppose $M\simeq N$ for 
    \begin{align}
        N = N_1\overline{\otimes} \cdots \overline{\otimes} N_n,
    \end{align}
    where $1\leq n < \infty$, and $N_1,\ldots, N_n\in \Cvert$ are other prime factors.
     Then $n=m$ and there is a permutation $\alpha$ of $\{1,\ldots, m\}$ such that $M_i$ is isomorphic to an amplification of $N_{\alpha(i)}$.
    \begin{proof}
    Since $M\in \Ccomplete$, there is a non-empty complete graph $\Gamma$ and factors $M_v\in \Cvert$ for $v\in \Gamma$ such that $M = *_{v,\Gamma}(M_{v},\tau_v)$. Hence $M = \overline{\bigotimes}_{v\in \Gamma}M_v$ since $\Gamma$ is complete. Moreover, for each $v\in \Gamma$ the factor $M_v$ is prime (see \cref{remark:classes-graph-products} \eqref{Item=class-Cvert}). This shows \eqref{eq:UPF-class-Ccomplete} with $m=|\Gamma|$.  Let $\Lambda$ be a complete graph with $n$ vertices. Then $N = \overline{\bigotimes}_{1\leq i\leq n}N_i = *_{v,\Lambda}(N_i,\tau_i)$. Since $\Gamma$ and $\Lambda$ are rigid we obtain by \cref{thm:rigid-graph-decomposition} a graph isomorphism $\alpha:\Gamma\to \Lambda$ such that $M_i$ is isomorphic to an amplification of $N_{\alpha(i)}$. In particular, $n=|\Lambda| = |\Gamma|=m$.
    \end{proof}
\end{corollary}

\section{Classification of strong solidity for graph products}\label{Sect=StrongSolidity}
We state the definition of strong solidity. We recall the assumption that inclusions of von Neumann algebras are understood as unital inclusions. 

\begin{definition}\label{Dfn=StronglySolid}
    A von Neumann algebra $M$ is called \textit{strongly solid} if for any diffuse, amenable, von Neumann subalgebra $A\subseteq M$, $\Nor_{M}(A)''$ is also amenable.
\end{definition}

\begin{remark}
Note that a tracial von Neumann algebra that is not diffuse must be strongly solid as it contains no diffuse unital subalgebras at all. 
\end{remark}

In \cref{sub:classification_strong_solidity} we characterize strong solidity for graph products $M_{\Gamma}$ of tracial von Neumann algebras $(M_v,\tau_v)$. In \cref{Sect=ClassStrongSolExamples} we then show that for many concrete cases this makes it possible to verify whether the graph product is strongly solid.

\subsection{Strong solidity main result}\label{sub:classification_strong_solidity}
We proof the main result \cref{Thm=MainImplication}.
The overall proof method is similar to \cite[Theorem 4.4] {borstClassificationRightangledCoxeter2023}, where strong solidity was classified for the group von Neumann algebras $\calL(\calW_{\Gamma})$ of right-angled Coxeter groups.
To characterize strong solidity we use the following result concerning amalgamated free products.
\begin{theorem}[Theorem A of \cite{vaesNormalizersAmalgamatedFree2014}]
	\label{thm:prelim:alternatives}
	Let $(N_1,\tau_1)$,$(N_2,\tau_2)$ be tracial von Neumann algebras with a common von Neumann subalgebra $B\subseteq N_i$ satisfying $\tau_1|_{B} = \tau_2|_{B}$, and denote $N:= N_1 \ast_{B} N_2$ for their amalgamated free product. Let $A\subseteq 1_AN1_A$ be a von Neumann algebra  that is amenable relative  to  $N_1$ or $N_2$ inside $N$. Put $P = \Nor_{1_{A}N1_{A}}(A)''$. Then  at least one of the following is true:
	\begin{enumerate}[(i)]
		\item $A\prec_{N} B$;
		\item $P \prec_{N} N_i$ for some $i=1,2$;
		\item $P$ is amenable relative to $B$ inside $N$. 
	\end{enumerate}
\end{theorem}

Furthermore, we use the following results that are rather standard.

\begin{proposition}[Proposition 4.2. in \cite{borstClassificationRightangledCoxeter2023} or Proof of Corollary C in \cite{vaesNormalizersAmalgamatedFree2014}]\label{Prop=EmbedSolid}
	Let $N \subseteq M$ be a von Neumann subalgebra and assume $N$ is strongly solid.
	Let $A \subseteq M$ be a diffuse amenable von Neumann subalgebra and $P = \Nor_M(A)''$ and $z \in P \cap P'$ be a non-zero projection. Assume that $zP \prec_M N$. Then $zP$ has an amenable direct summand.
\end{proposition}

 Recall that a von Neumann algebra $M$ is atomic if any projection in $M$ majorizes a minimal projection.  If $M$ is atomic it is a direct sum of type I factors. We state the following proposition.

\begin{proposition}\label{prop:strong-solid-tensor-product}
    Let $N = N_1\overline{\otimes} N_2$ be a tensor product of finite von Neumann algebras $N_1,N_2$. The following statements hold:
    \begin{enumerate}
        \item \label{Item=StronglySolidImpliesAtomic}Suppose $N_1$ is non-amenable and diffuse and $N$ is strongly solid. Then $N_2$ is atomic;
        \item \label{Item=StronglySolidNonDiffCase} Suppose $N_1$ is non-amenable and $N_2$ is diffuse. Then $N$ is not strongly solid;
        \item \label{Item=StronglySolidAtomicTensor} Suppose $N_1$ is strongly solid and diffuse and $N_2$ is atomic. Then $N$ is strongly solid.
    \end{enumerate}
    \begin{proof}
        1)  Write $N_2 = N_c \oplus N_d$ with $N_c$ either 0 or a diffuse von Neumann algebra and $N_d$ an atomic von Neumann algebra. Assume $N_c \not = 0$. Let $A \subseteq N_c, B \subseteq N_1$ be diffuse amenable von Neumann subalgebras. Then $C := \mathbb{C}1_{N_1} \overline{\otimes} A \oplus B \overline{\otimes} \mathbb{C}1_{N_d} \subseteq N$ is diffuse and amenable. Furthermore, $\Nor_N(C)''$ contains $N_1 \overline{\otimes} A \oplus   B \overline{\otimes}  \mathbb{C}1_{N_d}$ which is non-amenable. This contradicts that $N$ is strongly solid and we conclude that $N_c = 0$.\\
        2)  Take any diffuse amenable subalgebra $A \subseteq N_2$,  for instance we may take $A$ to be a maximal abelian subalgebra. Then $\mathbb{C}1_{N_1} \overline{\otimes} A$ is a diffuse amenable subalgebra of $N$ and $\Nor_N(\mathbb{C}1_{N_1} \overline{\otimes} A)''$ contains $N_1 \overline{\otimes} A$ which is non-amenable. Hence $N$ is not strongly solid. \\
        3) As $N_2$ is atomic we may identify  $N_2$  with  $\bigoplus_{\alpha \in I} \Mat_{n_\alpha}(\mathbb{C})$ where $I$ is some index set and $n_\alpha \in \mathbb{N}_{\geq 1}$. Let $1_{\alpha}$ be the unit of  $\Mat_{n_\alpha}(\mathbb{C})$.  Let $A \subseteq N_1 \overline{\otimes} N_2$ be a diffuse amenable von Neumann subalgebra. Then $1_{\alpha} A \subseteq N_1 \otimes \Mat_{n_\alpha}(\mathbb{C})$. So $\Nor_{N_1 \otimes \Mat_{n_\alpha}(\mathbb{C})}( 1_{\alpha}  A   )''$ is amenable by \cite[Proposition 5.2]{houdayerStronglySolidGroup2010} since $N_1$ is strong solid and diffuse. Since $\Nor_{N}(A)''  = \bigoplus_{\alpha \in I} \Nor_{N_1 \otimes \Mat_{n_\alpha}(\mathbb{C})}( 1_{\alpha} A   )'' $ and direct sums of amenable von Neumann algebras are amenable we conclude that $\Nor_{N}(A)'' $ is amenable. It follows that $N$ is strongly solid.
    \end{proof}
\end{proposition}

We classify atomicity for graph products.
\begin{proposition}\label{Prop=CharacterizeAtomicGraphProduct} Let $\Gamma$ be a  finite simple graph. 
Let $(M_\Gamma, \tau_\Gamma) = \ast_{v, \Gamma} (M_v, \tau_v)$ be a graph product of tracial von Neumann algebras over a simple graph $\Gamma$. Then $M_\Gamma$ is atomic if and only if  $\Gamma$ is complete and each $M_v$ is atomic. 
\begin{proof}
Any subalgebra of an atomic von Neumann algebra is atomic again. It follows that each $M_v$ is atomic. If $\Gamma$ would not be complete then we may pick $v, w \in \Gamma$ not sharing an edge and $(M_v, \tau_v) \ast (M_w, \tau_w) \subseteq M_\Gamma$. However,  $(M_v, \tau_v) \ast (M_w, \tau_w)$ is not atomic by \cite{uedaFactorialityTypeClassification2011}. So $\Gamma$ is complete. Conversely if $\Gamma$ is complete and each $M_v$ is atomic then  $M = \overline{\bigotimes}_{v \in \Gamma} M_v$ is atomic. 
\end{proof}
\end{proposition}

We now classify strong solidity for graph products in terms of conditions on subgraphs. These conditions can be verified in most cases (see \cref{Prop=CharacterizeAtomicGraphProduct}, \cref{Thm=AmenableGraphProduct}, \cref{Prop=ClassifyDiffuse} and \cref{thm:Hecke-algebras:diffuseness-and-amenability}).
\begin{theorem}\label{Thm=MainImplication}
	Let $\Gamma$ be a finite graph and for each $v\in \Gamma$ let $M_{v}$ ($\not=\CC$) be a von Neumann algebra with normal faithful trace $\tau_{v}$. Then $M_{\Gamma}$ is strongly solid if and only if the following conditions are  satisfied:
	\begin{enumerate}
		\item \label{Item=StronglySolid1}  For every vertex  $v\in \Gamma$ the von Neumann algebra $M_v$ is strongly solid;
    \item  \label{Item=StronglySolid3}  For every subgraph $\Lambda\subseteq \Gamma$ with $M_{\Lambda}$ non-amenable, we have that $M_{\Link(\Lambda)}$  is not diffuse;
		\item \label{Item=StronglySolid2}  For every subgraph $\Lambda\subseteq \Gamma$ with $M_{\Lambda}$ non-amenable and diffuse, we have moreover that $M_{\Link(\Lambda)}$ is atomic. 
	\end{enumerate} 
\begin{proof}
    Suppose $M_{\Gamma}$ is strongly solid, we show that conditions \eqref{Item=StronglySolid1}, \eqref{Item=StronglySolid3}  and \eqref{Item=StronglySolid2}  are satisfied. Since strong solidity passes to subalgebras, as follows from its very definition,  we obtain that \eqref{Item=StronglySolid1} is satisfied. Furthermore, suppose $\Gamma_0\subseteq \Gamma$ is a subgraph for which $M_{\Gamma_0}$ is non-amenable.  We have $M_{\Gamma_0 \cup \Link(\Gamma_0)} = M_{\Gamma_0} \overline{\otimes} M_{\Link(\Gamma_0)}$ which is strongly solid being a von Neumann subalgebra of $M_\Gamma$. Hence, \cref{prop:strong-solid-tensor-product}\eqref{Item=StronglySolidNonDiffCase} shows that $M_{\Link(\Gamma_0)}$ cannot be diffuse.  This concludes \eqref{Item=StronglySolid3}. If  $M_{\Gamma_0}$  is diffuse then \cref{prop:strong-solid-tensor-product}\eqref{Item=StronglySolidImpliesAtomic} shows that $M_{\Link(\Gamma_0)}$ is atomic. This concludes \eqref{Item=StronglySolid2}.\\

	We now show the reverse direction. The proof is based on induction to the number of vertices of the graph. The statement clearly holds when $\Gamma=\varnothing$ since in that case $M_{\Gamma} =\CC$ is strongly solid.

	\vspace{0.3cm}
	
	\noindent {\it Induction.} 
	Let $\Gamma$ be a non-empty graph, and assume by induction that \cref{Thm=MainImplication} is proved for any strictly smaller subgraph of $\Gamma$, i.e. with less vertices.
	Assume conditions \eqref{Item=StronglySolid1}, \eqref{Item=StronglySolid3} and \eqref{Item=StronglySolid2} are satisfied for $\Gamma$. Observe that condition \eqref{Item=StronglySolid1}, \eqref{Item=StronglySolid3}  and \eqref{Item=StronglySolid2} are then satisfied for all subgraphs of $\Gamma$ as well. Hence by the induction hypothesis we obtain that $M_{\Gamma_0}$ is strongly solid for all strict subgraphs $\Gamma_0\subsetneq \Gamma$. We shall show that $M_\Gamma$ is strongly solid. Let $A \subseteq M$ be diffuse and amenable and denote $P = \Nor_M(A)''$. We will show that $P$ is amenable.

	Suppose there is $v\in \Gamma$ with $\Star(v) = \Gamma$. Then we can decompose the graph product as $M_{\Gamma} = M_{v}\overline{\otimes} M_{\Gamma\setminus \{v\}}$. Now $M_{v}$ is strongly solid by condition \eqref{Item=StronglySolid1}, and $M_{\Gamma\setminus \{v\}}$ is strongly solid by the induction hypothesis as $\Gamma\setminus \{v\}\subsetneq \Gamma$ is a strict subgraph. 	When both $M_{v}$ and $M_{\Gamma\setminus \{v\}}$ are amenable then  $M_{\Gamma} = M_{v}\overline{\otimes} M_{\Gamma\setminus \{v\}}$ is also amenable, and hence $M_{\Gamma}$ is strongly solid. We can thus assume that $M_{v}$ or $M_{\Gamma\setminus \{v\}}$ is non-amenable.
 If $M_{v}$ is non-amenable we need to separate two cases. 
 \begin{itemize}
     \item If $M_v$ is non-amenable and not diffuse then  by condition  \eqref{Item=StronglySolid3}  neither $M_{\Gamma\setminus \{v\}}$ is diffuse and hence neither is $M_\Gamma = M_v \overline{\otimes} M_{\Gamma\setminus \{v\}}$. Then certainly $M_\Gamma$ is strongly solid by the absence of (unital) diffuse subalgebras. 
     \item If $M_v$ is non-amenable and  diffuse   
  then  by condition  \eqref{Item=StronglySolid2} we obtain that $M_{\Link(v)}$ ($=M_{\Gamma\setminus \{v\}}$) is atomic, so that  by \cref{prop:strong-solid-tensor-product}\eqref{Item=StronglySolidAtomicTensor} we have $M_\Gamma = M_{\Link(v) } \overline{\otimes} M_v$ is strongly solid. 
  \end{itemize}
	The case when $M_{\Gamma\setminus \{v\}}$ is non-amenable can be treated in the same way by swapping the roles of $M_v$ and  $M_{\Gamma\setminus \{v\}}$  in the previous argument. We summarize that our proof is complete in case there is $v \in \Gamma$ with $\Star(v) = \Gamma$.

Now we  assume that for all $v\in \Gamma$ we have $\Star(v)\not=\Gamma$.
	Pick $v \in \Gamma$ and set $\Gamma_1:=\Star(v)$ and  $\Gamma_2 := \Gamma \setminus \{v\}$. By \eqref{Eqn=Amalgam}  we can decompose  $M_\Gamma = M_{\Gamma_1} \ast_{ M_{\Gamma_1 \cap \Gamma_2}} M_{\Gamma_2}$. Moreover, as $\Gamma_1$, $\Gamma_2$ and $\Gamma_{1}\cap \Gamma_2$ are strict subgraphs of $\Gamma$ we obtain by our induction hypothesis that $M_{\Gamma_1}$, $M_{\Gamma_2}$ and $M_{\Gamma_1\cap \Gamma_2}$ are strongly solid.

	Let $z\in P\cap P'$ be a central  projection such that $z P$ has no amenable direct summand. Note that $zP \subseteq \Nor_{z M_\Gamma z}( zA)''$. As $zA$ is amenable, it is amenable relative to $M_{\Gamma_1}$ in $M_{\Gamma}$. Therefore by \cref{thm:prelim:alternatives} at least one of the following three holds.
	\begin{enumerate}
		\item\label{Case1} $z A \prec_{M_\Gamma} M_{\Gamma_1 \cap \Gamma_2}$;
		\item\label{Case2} There is $i \in \{ 1, 2\}$ such that $z P \prec_{M_\Gamma} M_{\Gamma_i}$;
		\item\label{Case3} $z P$ is amenable relative to $M_{\Gamma_1 \cap \Gamma_2}$ inside $M_\Gamma$.
	\end{enumerate}
	We now analyse each of the cases. 
	
	\vspace{0.3cm}

	\noindent {\it Case \eqref{Case2}.} In Case \eqref{Case2} we have that Proposition \ref{Prop=EmbedSolid} together with the induction hypothesis shows that $zP$ has an amenable direct summand in case $z \not = 0$. This is a contradiction so we conclude $z=0$ and hence $P$ is amenable. 
	
	\vspace{0.3cm}
	
	\noindent {\it Case \eqref{Case1}.}  In Case \eqref{Case1}, since $zA\prec_{M_\Gamma} M_{\Gamma_1\cap \Gamma_2}$ but $zA\not\prec_{M_\Gamma} \CC = M_{\varnothing}$, there is a subgraph $\Lambda \subseteq \Gamma_1 \cap \Gamma_2$ such that $zA \prec_{M_\Gamma} M_\Lambda$ but $zA \not \prec_{M_\Gamma} M_{\widetilde{\Lambda}}$ for any strict subgraph $\widetilde{\Lambda}\subseteq \Lambda$. Put $\Lambda_{\emb} := \Lambda \cup \Link(\Lambda)$. Observe that $\Lambda_{\emb}$ contains at least $v$ and $\Lambda$. Furthermore, by \cref{prop:embed}\eqref{Item=Embed2} we obtain that $zP\prec_{M_{\Gamma}} M_{\Lambda_{\emb}}$. If $\Lambda_{\emb}\not=\Gamma$ then $M_{\Lambda_{\emb}}$ is strongly solid by the induction hypothesis. Therefore, in case $z\not=0$ we obtain by \cref{Prop=EmbedSolid} that $zP$ has an amenable direct summand, which is a contradiction. Thus $z=0$, and $P$ is amenable. Hence $M_{\Gamma}$ is strongly solid.
	
	We can thus assume that $\Lambda_{\emb} = \Gamma$.
	Suppose $M_{\Lambda}$ is non-amenable. Again we separate two cases:
 \begin{itemize}
     \item  Assume that  $M_{\Lambda}$ is non-amenable and diffuse. Then by condition \eqref{Item=StronglySolid2} we have that $M_{\Link(\Lambda)}$ is atomic and by \cref{Prop=CharacterizeAtomicGraphProduct} we see that $\Link(\Lambda)$  must be complete.  But as $v\in \Link(\Lambda)$ this implies that $\Link(\Lambda)\subseteq \Star(v) = \Gamma_1$ and thus $\Lambda_{\emb}\subseteq \Gamma_1$. Therefore $\Lambda_{\emb}$ is a strict subgraph of $\Gamma$, a contradiction. So this case does not occur;
     \item   Assume that  $M_{\Lambda}$ is non-amenable and not diffuse. Then by \eqref{Item=StronglySolid3}  $M_{\Link(\Lambda)}$ is not diffuse either. As $M_\Gamma = M_{\Lambda }   \overline{\otimes} M_{\Link(\Lambda)}$ we find that $M_\Gamma$ is not diffuse and thus strongly solid by absence of diffuse (unital) subalgebras.
 \end{itemize}
	Next suppose $M_{\Link(\Lambda)}$ is non-amenable. Again we separate two cases: 
 \begin{itemize}
     \item  Assume that  $M_{\Link(\Lambda)}$ is non-amenable and diffuse. 
Then  $M_{\Link(\Link(\Lambda))} = M_\Lambda$ is atomic by \eqref{Item=StronglySolid2}.   But then $zA \prec_{M_\Gamma} M_{\Lambda}$ with $zA$ diffuse leads to a contradiction;
\item  Assume that $M_{\Link(\Lambda)}$ is non-amenable and not diffuse.  Then by  \eqref{Item=StronglySolid3} also $M_\Lambda$ is not diffuse and so $M_\Gamma = M_{\Lambda }   \overline{\otimes} M_{\Link(\Lambda)}$  is not diffuse and thus strongly solid.
 \end{itemize}
So we are left with the case that  $M_{\Lambda}$ and $M_{\Link(\Lambda)}$ are amenable. In this case, $M_{\Gamma}= M_{\Lambda_{\emb}} =  M_{\Lambda}\overline{\otimes} M_{\Link(\Lambda)}$ is amenable and hence $M_{\Gamma}$ is strongly solid.
	\vspace{0.3cm}
	
	\noindent {\it Remainder of the proof of the main theorem in the situation that  Case \eqref{Case1} and Case \eqref{Case2} never occur.}  
	We first recall that if we can find a single vertex $v$ as above such that we are in case  \eqref{Case1} or \eqref{Case2} then the proof is finished. Otherwise for any vertex $v \in \Gamma$  we are in case  \eqref{Case3}.  So  $zP$ is amenable relative to $M_{\Link(v)}$ inside $M_\Gamma$. As $\bigcap_{v\in \Gamma}\Link(v)\subseteq \bigcap_{v\in \Gamma}\Gamma\setminus\{v\}=\varnothing$ we obtain by iteratively using \cref{Thm=Square} that $zP$ is amenable relative to $\bigcap_{v\in V}M_{\Link(v)} = \CC$, that is $zP$ is amenable. So $z =0$ and we conclude again that $P$ is amenable.  
\end{proof}
\end{theorem} 

\subsection{Classifying strong solidity in specific cases}\label{Sect=ClassStrongSolExamples}
We show that in many concrete cases that one can verify whether or not a graph product $M_{\Gamma}$ is strongly solid.
\cref{Thm=MainImplication} tells us how to decide whether $M_{\Gamma}$ is strongly solid. For this we need to know for each vertex $v$ whether or not $M_v$ is strongly solid. Furthermore, we need to know for each subgraph $\Lambda\subseteq \Gamma$ whether of not $M_{\Lambda}$ is atomic, diffuse, or non-amenable. We observe that in concrete cases we can verify whether $M_{\Lambda}$ is diffuse, atomic or non-amenable. Indeed, atomicity is classified in \cref{Prop=CharacterizeAtomicGraphProduct}. Furthermore, amenability was classifed in \cite{charlesworthStructureGraphProduct2024}. Moreover, in \cite{charlesworthStructureGraphProduct2024} diffuseness was classified under the condition that each vertex algebra $M_v$ contains a unitary element of trace $0$, i.e. a Haar unitary. This in particular applies to the case where $M_v$ is either diffuse or a group von Neumann algebra. We state these results here.

\begin{proposition}[Proposition 6.3 of \cite{charlesworthStructureGraphProduct2024}]    
\label{Thm=AmenableGraphProduct}
	Let $\Gamma$ be a simple graph. For $v\in \Gamma$ let $M_v$ ($\not=\CC$) be a von Neumann algebra with normal faithful state $\varphi_v$.  
    Then the graph product $M_{\Gamma} = *_{v,\Gamma}(M_{v},\varphi_v)$ is amenable if and only if the following conditions hold:
	\begin{enumerate}
		\item \label{Thm:Item=Amen1} Each vertex von Neumann algebra $M_v,v\in \Gamma$ is amenable;
		\item \label{Thm:Item=Amen2} If $v\not=w\in \Gamma$ share no edge, then $\dim M_v =\dim M_w =2$ and $\Link(\{v,w\}) = \Gamma\setminus \{v,w\}$.
	\end{enumerate}
\end{proposition}

\begin{proposition}[Theorem E of \cite{charlesworthStructureGraphProduct2024}]\label{Prop=ClassifyDiffuse}     
 Let $(M_\Gamma, \tau_\Gamma) = \ast_{v, \Gamma} (M_v, \tau_v)$ be a graph product of tracial von Neumann algebras over a simple graph $\Gamma$. Assume that each $M_v, v \in \Gamma$ contains a unitary $u_v$ with $\tau_v(u_v) = 0$.  Then $M_\Gamma$ is diffuse if either (a) there is $v \in \Gamma$ with $M_v$ diffuse; (b) $\Gamma$ is not a complete graph.  
\end{proposition}

In case not every vertex von Neumann algebra contain a unitary of trace 0 the situation becomes more subtle and the analysis becomes significantly more intricate. However, if the vertex von Neumann algebras are 2-dimensional then the results in
\cite{caspersGraphProductKhintchine2021},
\cite{garncarekFactorialityHeckeNeumann2016}, \cite{raumFactorialMultiparameterHecke2023} again yield a classification of diffuseness (and amenability) of graph products.  

\begin{definition}
Suppose that $M_{v, q_v}, q_v \in (0,1]$ is the 2-dimensional Hecke algebra which is the $\ast$-algebra spanned by the unit $1_v$ and an element $T_{v, q_v}$ satisfying the Hecke relation
\[
(T_{v, q_v} - q_v^{\frac{1}{2}}) (T_{v, q_v} + q_v^{-\frac{1}{2}}) = 0, \quad  T_{v, q_v}^\ast = T_{v, q_v}. 
\]
Define the tracial state $\tau_v$ by setting $\tau_v(T_{v, q_v} ) = 0$ and $\tau_v(1_v) = 1$.  For a simple graph $\Gamma$ and ${\bf q} := (q_v)_{v \in \Gamma} \in (0, 1]^{\Gamma}$ we let $M_{\Gamma, {\bf q} } = \ast_{v, \Gamma} (M_v, \tau_{v, q_v})$ be the graph product von Neumann algebra which is called the {\it right-angled Hecke von Neumann algebra}. 
\end{definition}

\begin{remark}
Note that $(M_{v, q_v}, \tau_v)$ is isomorphic to $\mathbb{C}^2$ with tracial state $\tau_{\alpha}(x\oplus y) := \alpha x + (1-\alpha)y$ with $\alpha = \frac{1}{2}\left(1+\sqrt{1-\frac{4}{p_v(q)^2 +4}}\right)$ where $p_v(q) := \frac{q_v -1}{\sqrt{q_v}}\in (-\infty,0]$. Hence a general 2-dimensional von Neumann algebra with a (necessarily tracial) faithful state is of the form $(M_{v, q_v}, \tau_q)$ and Hecke algebras correspond to a general graph product of 2-dimensional von Neumann algebras.  
\end{remark}

Let  $L$ be the graph with 3 points and no edges and $L^+$ be the graph with 3 points and 1 edge between two of the points. 

\begin{theorem}[Theorem A of \cite{raumFactorialMultiparameterHecke2023}, Theorem 6.2 of \cite{caspersGraphProductKhintchine2021}] 
\label{thm:Hecke-algebras:diffuseness-and-amenability}
 Let $\Gamma$ be a  finite simple graph and ${\bf q} := (q_v)_{v \in \Gamma} \in (0, 1]^{\Gamma}$. Then
 \begin{enumerate}
\item  The Hecke von Neumann algebra  $M_{\Gamma, {\bf q} }$ is not diffuse if and only if the sum $\sum_{\ww \in \mathcal{W}_\Gamma } q_{\ww }$,  converges where $q_{\ww} = q_{w_1} \ldots q_{w_n}$ and $\ww = w_1 \ldots w_n$ reduced;
 \item  $M_{\Gamma, {\bf q} }$ is non-amenable if and only if $\mathcal{W}_\Gamma$ is non-amenable if and only if $L$ or $L^+$ is a subgraph of $\Gamma$. 
 \end{enumerate}
 \end{theorem}
  
 Hence, by \cref{Thm=MainImplication} and \cref{Prop=CharacterizeAtomicGraphProduct} and \cref{thm:Hecke-algebras:diffuseness-and-amenability}  the classification of strongly solid right-angled Hecke von Neumann algebras with finitely many generators is complete. Partial results toward this classification had been obtained earlier in \cite{caspersAbsenceCartanSubalgebras2020} and \cite{borstBimoduleCoefficientsRiesz2023}.

\section{Classification of primeness for graph products}
\label{Sect=Primeness}
We start by recalling the definition of primeness. 
\begin{definition}\label{Dfn=Prime}
A II$_1$-factor $M$ is called \textit{prime} if it can not factorize as a tensor product $M = M_1\overline{\otimes} M_2$ with $M_1,M_2$ diffuse.
\end{definition}
We study primeness for graph product $M_{\Gamma} = *_{v,\Gamma}(M_v,\tau_v)$ of tracial von Neumann algebras $M_v$. In \cref{sub:primeness-II1-factors} we prove \cref{thm:primeness-for-graph-products-II1-factors} which characterizes primeness for graph products of II$_1$-factors. In \cref{sub:UPF-results-graph-products} we use this to prove \cref{thm:rigid-graph-decomposition-irreducible-components} concerning irreducible components in rigid graph products. Moreover, we prove \cref{thm:unique-prime-factorization} which establishes UPF-results for the class $\Crigid$.
Last, in \cref{sub:primeness-general-setting} we extend the primeness characterization from \cref{thm:primeness-for-graph-products-II1-factors} to a larger class of graph products.

\subsection{Primeness results for graph products of II$_1$-factors}
\label{sub:primeness-II1-factors}
We prove \cref{lemma:ICC-group-not-embed-in-smaller-graph-extended} which we use in \cref{lemma:irreducible-graph-prime-or-amenable} to give sufficient conditions for a graph product to be either prime or amenable. For graph products of II$_1$-factors we then characterize primeness in \cref{thm:primeness-for-graph-products-II1-factors}.

\begin{lemma}\label{lemma:ICC-group-not-embed-in-smaller-graph-extended}Let $\Gamma$ be a finite simple graph that is irreducible. 
 For $v\in \Gamma$ let $M_{v}$ ($\not=\CC$) be a von Neumann algebra with a normal faithful trace $\tau_v$. 
 Suppose $N\subseteq M_{\Gamma}$ is a diffuse von Neumann subalgebra. The following are equivalent:
	\begin{enumerate}
		\item $N\not\prec_{M_{\Gamma}} M_{\Gamma_0}$ for any strict subgraph $\Gamma_0\subsetneq \Gamma$;
		\item $\Nor_{M_{\Gamma}}(N)''\not\prec_{M_{\Gamma}}M_{\Gamma_0}$ for any strict subgraph $\Gamma_0\subsetneq \Gamma$.
	\end{enumerate}   
	\begin{proof}
	As $N\subseteq \Nor_{M_{\Gamma}}(N)''$, it is clear that $(1)\implies (2)$. We will show that $(2)\implies (1)$.
		
		As $N\subseteq M_{\Gamma}$ is a subalgebra, we have that $N\prec_{M_{\Gamma}}M_{\Gamma}$. Therefore, there is a (minimal) subgraph $\Lambda\subseteq \Gamma$ such that 
		$N\prec_{M_{\Gamma}}M_{\Lambda}$ and $N\not \prec_{M_{\Gamma}} M_{\widetilde{\Lambda}}$ for all strict subgraphs $\widetilde{\Lambda}\subsetneq \Lambda$.  By \cref{prop:embed} \eqref{Item=Embed2} we obtain that $\Nor_{M_{\Gamma}}(N)'' \prec_{M_{\Gamma}}M_{\Lambda_{\emb}}$ where $\Lambda_{\emb} = \Lambda\cup \Link(\Lambda)$. Now by our assumption this implies that $\Lambda_{\emb} = \Gamma$. Now, as $\Gamma$ is irreducible and $\Gamma =  \Lambda\cup \Link(\Lambda)$ we have that $\Lambda$ or $\Link(\Lambda)$ is empty. As $N\not\prec_{M_{\Gamma}}\CC1_{M_{\Gamma}}$ (since $N$ is diffuse) and $N\prec_{M_{\Gamma}}M_{\Lambda}$ we must have that $\Lambda$ is non-empty, and thus that $\Link(\Lambda)$ is empty. Thus $\Lambda = \Gamma$, and this proves the statement.
	\end{proof}
\end{lemma}

\begin{lemma}\label{lemma:irreducible-graph-prime-or-amenable}
Let $\Gamma$ be a  finite irreducible graph with $|\Gamma|\geq 2$. For $v\in \Gamma$ let $M_{v}$ ($\not=\CC$) be a von Neumann algebra with a normal faithful trace $\tau_v$. Suppose the graph product $M_{\Gamma} = *_{v,\Gamma}(M_v,\tau_v)$ is a II$_1$-factor and $M_{\Gamma}\not\prec_{M_{\Gamma}} M_{\Gamma_0}$ for any strict subgraph $\Gamma_0\subsetneq \Gamma$. Then $M_{\Gamma}$ is prime or amenable.
	\begin{proof} Suppose that $M_{\Gamma}$ is not prime, we show it is amenable. As $M_{\Gamma}$ is not prime, we can write $M_{\Gamma} = N_1\overline{\otimes} N_2$ with $N_1,N_2$ both diffuse.
		We observe that $\Nor_{M_{\Gamma}}(N_1)'' = M_{\Gamma}$. Therefore, using our assumption on $M_{\Gamma}$ and applying		\cref{lemma:ICC-group-not-embed-in-smaller-graph-extended} we obtain that $N_1\not\prec_{M_{\Gamma}}M_{\Gamma_0}$ for any strict subgraph $\Gamma_0\subsetneq \Gamma$.
		
		As $N_2$ is diffuse it contains a diffuse amenable von Neumann subalgebra $A\subseteq N_2$. Now observe that 
		$\Nor_{M_{\Gamma}}(A)''$ contains $N_1$ and hence $\Nor_{M_{\Gamma}}(A)''\not\prec_{M_{\Gamma}} M_{\Gamma_0}$ for any strict subgraph $\Gamma_0\subsetneq \Gamma$. Thus, again by 	\cref{lemma:ICC-group-not-embed-in-smaller-graph-extended} we obtain that $A\not\prec_{M_{\Gamma}} M_{\Gamma_0}$ for any strict subgraph $\Gamma_0\subsetneq \Gamma.$

		Let $v\in \Gamma$ and put $\Gamma_1 := \Star(v)$ and $\Gamma_2 := \Gamma\setminus \{v\}$. We can write 
		\begin{align}
			M_{\Gamma} = M_{\Gamma_1} *_{M_{\Link(v)}} M_{\Gamma_2}.
		\end{align}
	As $A$ is amenable relative to $M_{\Gamma_1}$ inside $M_{\Gamma}$ (as $A$ is amenable), we obtain using \cref{thm:prelim:alternatives} that at least one of the following holds:
	\begin{enumerate}
		\item $A\prec_{M_{\Gamma}} M_{\Link(v)}$;
		\item $\Nor_{M_{\Gamma}}(A)'' \prec_{M_{\Gamma}} M_{\Gamma_i}$ for some $i\in\{1,2\}$;
		\item $\Nor_{M_{\Gamma}}(A)''$ is amenable relative to $M_{\Link(v)}$ inside $M_{\Gamma}$.
	\end{enumerate}
	Now as $\Gamma_1,\Gamma_2$ and $\Link(v)$ are strict subgraphs of $\Gamma$ (as $\Gamma$ is irreducible and $|\Gamma|\geq 2$), we obtain that only option (3) is possible. Thus $\Nor_{M_{\Gamma}}(A)''$ is amenable relative to $M_{\Link(v)}$ inside $M_{\Gamma}$. Note that $v\in \Gamma$ was chosen arbitarily. Thus, applying \cref{Thm=Square} repeatedly, and using that $\bigcap_{v\in \Gamma}\Link(v)=\varnothing$, we obtain that $\Nor_{M_{\Gamma}}(A)''$ is amenable relative to  $M_{\varnothing}$ ($= \CC1_{M_{\Gamma}}$) inside $M_{\Gamma}$, i.e. $\Nor_{M_{\Gamma}}(A)''$ is amenable. Hence the subalgebra $N_1\subseteq \Nor_{M_{\Gamma}}(A)''$ is amenable as well. Interchanging the roles of $N_1$ and $N_2$ we obtain that $N_2$ is also amenable, and hence $M_{\Gamma} = N_1\overline{\otimes} N_2$ is amenable.
	\end{proof}
\end{lemma}

We characterize primeness for graph products of II$_1$-factors.
\begin{theorem}\label{thm:primeness-for-graph-products-II1-factors}
    Let $\Gamma$ be a finite simple graph of size $|\Gamma|\geq 2$. For each $v\in \Gamma$ let $M_v$ be a II$_1$-factor. Then the graph product $M_{\Gamma} = *_{v,\Gamma}(M_v,\tau_v)$ is prime if and only if $\Gamma$ is irreducible.
    \begin{proof}
        Take the finite simple graph $\Gamma$ with $|\Gamma|\geq 2$ and the II$_1$-factors $(M_{v},\tau_v)$ for $v\in \Gamma$. By \cite[Theorem 1.2]{caspersGraphProductsOperator2017a} the von Neumann algebra $M_{\Gamma}$ is a factor. Furthermore, by \cref{Lem=Selfembedding} we have that $M_{\Gamma}\not\prec_{M_{\Gamma}} M_{\Gamma_0}$ for any strict subgraph $\Gamma_0\subsetneq \Gamma$.
        Suppose that $\Gamma$ is irreducible. Then by applying \cref{lemma:irreducible-graph-prime-or-amenable} we obtain that $M_{\Gamma}$ is prime or amenable. Since $\Gamma$ is irreducible and has size $|\Gamma|\geq 2$ we obtain that $\Gamma$ is not complete. We then see by \cref{Thm=AmenableGraphProduct} that $M_{\Gamma}$ is non-amenable. Thus $M_{\Gamma}$ is prime, which shows one direction.
        Now suppose $\Gamma$ is reducible, so that we can decompose $\Gamma = \Gamma_1\cup \Gamma_2$ with $\Gamma_1,\Gamma_2\subseteq \Gamma$ non-empty and such that $\Link(\Gamma_1)=\Gamma_2$. But then we can decompose $M_{\Gamma} = M_{\Gamma_1}\overline{\otimes} M_{\Gamma_2}$ as a tensor product and again by \cite[Theorem 1.2]{caspersGraphProductsOperator2017a} $M_{\Gamma_1}$ and $M_{\Gamma_2}$ are II$_1$-factors. Thus $M_{\Gamma}$ is not prime.
    \end{proof}
\end{theorem}

\subsection{Unique prime factorization results}
\label{sub:UPF-results-graph-products}
We prove \cref{thm:rigid-graph-decomposition-irreducible-components} which strengthens the statement of \cref{thm:rigid-graph-decomposition} by showing for irreducible components $\Gamma_0$ that $M_{\Gamma_0}$ is isomorphic to an amplification of $N_{\alpha(\Gamma_0)}$. We then use this result to prove \cref{thm:unique-prime-factorization} which establishes UPF results for the class $\Crigid$.

\begin{theorem}\label{thm:rigid-graph-decomposition-irreducible-components}
    Given a  finite rigid graph $\Gamma$. For each $v\in \Gamma$ let $M_v\in \Cvert$. Let $M_{\Gamma} = *_{v,\Gamma}(M_v,\tau_v)$ be the graph product. Suppose $M_{\Gamma} =*_{w,\Lambda}(N_w,\tau_w)$, with another rigid graph $\Lambda$ and other von Neumann algerbas $N_w\in\Cvert$. Let $\alpha:\Gamma\to \Lambda$ be the graph isomorphism from \cref{thm:rigid-graph-decomposition}. Then for any irreducible component $\Gamma_0\subseteq \Gamma$,  $M_{\Gamma_0}$ is isomorphic to an amplification of  $N_{\alpha(\Gamma_0)}$.
    \begin{proof}
    
    We observe that $M_{\Gamma\setminus \Gamma_0}' \cap M_{\Gamma} = M_{\Gamma_0}$ is non-amenable.
    Hence, by \cref{Thm=KeyAlternatives} we obtain a subgraph $\Lambda_0\subseteq \Lambda$ such that $M_{\Gamma\setminus \Gamma_0} \prec_{M_{\Gamma}} N_{\Lambda_0}$ and
    $\Link_{\Lambda}(\Lambda_0)\not=\varnothing$. Choose $\widetilde{\Lambda}_0\subseteq \Lambda_0$ minimal with the property that $M_{\Gamma\setminus \Gamma_0}\prec_{M_{\Gamma}} N_{\widetilde{\Lambda}_0}$. We show $\widetilde{\Lambda}_{0} = \alpha(\Gamma\setminus \Gamma_0)$.
    By \cref{prop:embed}\eqref{Item=Embed2} we have 
    $N_{\Lambda} = M_{\Gamma} = \Nor_{M_{\Gamma}}(M_{\Gamma\setminus \Gamma_0})'' \prec_{M_{\Gamma}} N_{\Lambda_{\emb}}$ where $\Lambda_{\emb} = \widetilde{\Lambda}_0 \cup \Link_{\Lambda}(\widetilde{\Lambda}_0)$.
    By \cref{Lem=Selfembedding} we conclude $\Lambda_{\emb} = \Lambda$.
    We note for $v\in \Gamma\setminus \Gamma_0$ that $N_{\alpha(v)}\prec_{M_{\Gamma}} M_{\Gamma\setminus \Gamma_0}$ and $M_{\Gamma\setminus \Gamma_0}\prec_{M_{\Gamma}}^s N_{\widetilde{\Lambda}_0}$ by \cref{Lem=StableEmbedding}\eqref{Item=StableEmbedding:condition}. Hence by \cref{Lem=StableEmbedding}\eqref{Item=StableEmbedding:transative} we obtain $N_{\alpha(v)}\prec_{M_{\Gamma}}N_{\widetilde{\Lambda}_0}$. Thus $\alpha(\Gamma\setminus \Gamma_0)\subseteq \widetilde{\Lambda}_0$ by \cref{Lem=Selfembedding}. Put $S = \widetilde{\Lambda}_0 \cap \alpha(\Gamma_0)$. Then $$S\cup \Link_{\alpha(\Gamma_0)}(S) = (\widetilde{\Lambda}_0\cup \Link_{\Lambda}(S))\cap  \alpha(\Gamma_0) \supseteq (\widetilde{\Lambda_0} \cup  \Link_{\Lambda}(\widetilde{\Lambda}_0)) \cap \alpha(\Gamma_0) = \alpha(\Gamma_0).$$ 
    Since the graph $\alpha(\Gamma_0)$ is irreducible, we conclude that $S = \varnothing$ or $S = \alpha(\Gamma_0)$.
    Now, if $S = \alpha(\Gamma_0)$ then $\alpha(\Gamma_0)\subseteq \widetilde{\Lambda}_0$, so that $\Lambda = \alpha(\Gamma_0)\cup \alpha(\Gamma\setminus \Gamma_0) \subseteq \widetilde{\Lambda}_0$ . But since $\widetilde{\Lambda}_0\subseteq \Lambda_0\subseteq \Lambda$ this implies $\Lambda_0 = \Lambda$, which contradicts the fact that $\Link_{\Lambda}(\Lambda_0)\not=\varnothing$. We conclude that $S = \varnothing$ and thus $\widetilde{\Lambda}_0 = \alpha(\Gamma\setminus \Gamma_0)$. 

    We have obtained $M_{\Gamma\setminus \Gamma_0} \prec_{M_{\Gamma_0}} N_{\alpha(\Gamma\setminus \Gamma_0)}$. Taking relative commutants, by \cref{Lem=RelativeCom}, we get $N_{\alpha(\Gamma_0)}\prec_{M_{\Gamma}}M_{\Gamma_0}$.
    Since we are dealing with II$_1$-factors, these embeddings are also with expectation, i.e. $N_{\alpha(\Gamma_0)}\preceq_{M_{\Gamma}} M_{\Gamma_0}$ as in \cite[Definition 4.1]{houdayerUniquePrimeFactorization2017}.
    Thus, since $M_{\Gamma} = M_{\Gamma_0}\overline{\otimes} M_{\Gamma\setminus \Gamma_0}$ we obtain by \cite[Lemma 4.13]{houdayerUniquePrimeFactorization2017} 
non-zero projections $p,q\in M_{\Gamma}$ and a partial isometry $v\in M_{\Gamma}$ with $v^*v =p$ and $vv^* =q$ and a subfactor $P\subseteq qN_{\alpha(\Gamma_0)}q$ so that
\begin{align*}
    qN_{\alpha(\Gamma_0)}q = vM_{\Gamma_0}v^* \overline{\otimes} P& &and& &vM_{\Gamma\setminus \Gamma_0}v^* = P\overline{\otimes} qN_{\alpha(\Gamma\setminus \Gamma_0)}q.
\end{align*}
By \cref{thm:primeness-for-graph-products-II1-factors} we have that $N_{\alpha(\Gamma_0)}$ is prime. Hence $qN_{\alpha(\Gamma_0)}q$ is prime. Thus, since $vM_{\Gamma_0}v^*$ is a II$_1$-factor, we obtain that $P$ is a type I$_n$ factor for some $n\in\NN$. We conclude that $N_{\alpha(\Gamma_0)}$ is isomorphic to some amplification of $M_{\Gamma_0}$. 
    \end{proof}
\end{theorem}

\begin{theorem}\label{thm:unique-prime-factorization}
    Any von Neumann algebra $M\in \Crigid^f$ have a prime factorization inside $\Crigid^f$, i.e. \begin{align}
        M = M_1\overline{\otimes} \cdots \overline{\otimes}M_m,
        \end{align}
        for some $1\leq m< \infty$ and prime factors $M_1,\ldots, M_m\in \Crigid^f$.
    
    Suppose there is another prime factorization of $M$ inside $\Crigid^f$, i.e.
    \begin{align}
        M =N_1\overline{\otimes} \cdots \overline{\otimes} N_n,
    \end{align}
    for another $1\leq n< \infty$ and other prime factors $N_1,\ldots,N_n\in \Crigid^f$. Then $m=n$ and there is a permutation $\sigma$ of $\{1,\ldots, m\}$ such that $M_{i}$ is isomorphic to some amplification of $N_{\sigma(i)}$.
    \begin{proof}
        Since $M\in \Crigid^f$, we can write $M = *_{v,\Gamma}(M_v,\tau_v)$ for some finite rigid graph $\Gamma$ and some $M_v\in \Cvert$ for $v\in \Gamma$. Let $\Gamma_1,\ldots, \Gamma_m   (1\leq m< \infty)$ be the irreducible components of $\Gamma$. Let $\Pi = \{1,\ldots,m\}$ be the complete graph with $m$ vertices and put $M_i = M_{\Gamma_i}$ for $i\in \Pi$.  Then since $\Gamma = \Gamma_{\Pi}$ we have by \cref{remark:graph-product-of-graphs-consistency} that $M = *_{v,\Gamma}(M_v,\tau_v) = *_{i,\Pi}(*_{v,\Gamma_i}(M_{v},\tau_v)) = M_1 \overline{\otimes} \cdots \overline{\otimes} M_m$. Now, for $i\in \Pi$ we have by \cref{thm:primeness-for-graph-products-II1-factors} that $M_i$ is prime since $\Gamma_i$ is irreducible. Note furthermore that $\Gamma_i$ is rigid by \cref{remark:rigid-components} and hence $M_i\in \Crigid^f$. 

        Now since $N_i\in \Crigid^f$ for $i\in\{1,2,\ldots,n\}$, we can write $N_i = *_{v,\Lambda_i}(N_v,\tau_v)$ for some non-empty, rigid graph $\Lambda_i$. We note that $\Lambda_i$ is irreducible by \cref{thm:primeness-for-graph-products-II1-factors} since $N_i$ is prime.         
        Let $\Pi' = \{1,\ldots, n\}$ be a complete graph with $n$ vertices and put $\Lambda := \Lambda_{\Pi}$ which is rigid by \cref{lemma:graph-product-of-rigid-graphs}. Then by \cref{remark:graph-product-of-graphs-consistency} we have $M = N_1\overline{\otimes} \cdots \overline{\otimes} N_n = *_{i,\Pi}(N_i,\tau_i) = *_{i,\Pi}(*_{v,\Lambda_i}(N_{v},\tau_v)) = *_{v,\Lambda}(N_v,\tau_v)$. Hence, we can apply \cref{thm:rigid-graph-decomposition} to obtain a graph isomorphism  $\alpha:\Gamma\to \Lambda$. We note that $\Lambda_1,\ldots, \Lambda_n$ are the irreducible components of $\Lambda$ and that $\Gamma_1,\ldots, \Gamma_m$ are the irreducible components of $\Gamma$. Since $\alpha$ is a graph isomorphism, this implies that $m=n$ and that there is a permutation $\sigma$ of $\{1,\ldots, m\}$ such that
        $\alpha(\Gamma_i) = \Lambda_{\sigma(i)}$. Now, for $1\leq i< m$ we obtain   by \cref{thm:rigid-graph-decomposition-irreducible-components}   a real number $0<t_i<\infty$ such that $M_i = M_{\Gamma_i} \simeq N_{\alpha(\Gamma_i)}^{t_i} = N_{\Lambda_{\sigma(i)}}^{t_i} = N_{\sigma(i)}^{t_i}$.
    \end{proof}
\end{theorem} 

\begin{remark}\label{remark:new-UPF-results}
    In \cref{figure:graph-UPF-result} we give an example of a von Neumann algebra  for which we obtain a unique prime factorization. This example was not yet covered by \cite[Theorem A]{houdayerUniquePrimeFactorization2017} since the graph $\Gamma$ is not complete. The example is also not covered by \cite[Theorem 6.16]{chifanTensorProductDecompositions2018} in case the vertex von Neumann algebras $M_v \in \Cvert$ are not known to be group von Neumann algebras. Examples of such $M_v$ can be found  as von Neumann algebras of free orthogonal quantum groups \cite{VaesVergnioux}  or $q$-Gaussian algebras of finite dimensional Hilbert spaces and $q \in (-1, 1)$ sufficiently far away from 0, see \cite[Remark 4.5]{borstIsomorphismClassGaussian2023} which is essentially proved in \cite{Kuzmin}.  We emphasize that it is not known whether such von Neumann algebras are group von Neumann algebras; we do not make the more definite claim that they cannot be isomorphic to group von Neumann algebras.  
\end{remark}

\begin{figure}[h!]	
		\begin{tikzpicture}[baseline]
			\node at (-2,2) 	  	(La) {$a$};
			\node at (-1.5,1) 	  	(Lb) {$b$};
			\node at (-1.3,0) 	  	(Lc) {$c$};
			\node at (-1.5,-1) 	  	(Ld) {$d$};
                \node at (-2,-2) 	  	(Le) {$e$};
			
			\node at (2,2) 	  	    (Ra) {$f$};
			\node at (1.5,1) 	  	(Rb) {$g$};
			\node at (1.3,0) 	  	    (Rc) {$h$};
			\node at (1.5,-1) 	  	(Rd) {$i$};
                \node at (2,-2) 	  	(Re) {$j$};

			\draw[-] (La) -- node {} (Lb);
                \draw[-] (Lb) -- node {} (Lc);
                \draw[-] (Lc) -- node {} (Ld);
                \draw[-] (Ld) -- node {} (Le);
                \draw[-] (Le) -- node {} (La);

                \draw[-] (Ra) -- node {} (Rb);
                \draw[-] (Rb) -- node {} (Rc);
                \draw[-] (Rc) -- node {} (Rd);
                \draw[-] (Rd) -- node {} (Re);
                \draw[-] (Re) -- node {} (Ra);

                \draw[-] (La) -- node {} (Ra);
                \draw[-] (La) -- node {} (Rb);
                \draw[-] (La) -- node {} (Rc);
                \draw[-] (La) -- node {} (Rd);
                \draw[-] (La) -- node {} (Re);

                \draw[-] (Lb) -- node {} (Ra);
                \draw[-] (Lb) -- node {} (Rb);
                \draw[-] (Lb) -- node {} (Rc);
                \draw[-] (Lb) -- node {} (Rd);
                \draw[-] (Lb) -- node {} (Re);

                \draw[-] (Lc) -- node {} (Ra);
                \draw[-] (Lc) -- node {} (Rb);
                \draw[-] (Lc) -- node {} (Rc);
                \draw[-] (Lc) -- node {} (Rd);
                \draw[-] (Lc) -- node {} (Re);

                \draw[-] (Ld) -- node {} (Ra);
                \draw[-] (Ld) -- node {} (Rb);
                \draw[-] (Ld) -- node {} (Rc);
                \draw[-] (Ld) -- node {} (Rd);
                \draw[-] (Ld) -- node {} (Re);

                \draw[-] (Le) -- node {} (Ra);
                \draw[-] (Le) -- node {} (Rb);
                \draw[-] (Le) -- node {} (Rc);
                \draw[-] (Le) -- node {} (Rd);
                \draw[-] (Le) -- node {} (Re);
                
		\end{tikzpicture}

		\caption{An example of a rigid graph $\Gamma$ is depicted. Put $M_v \in \Cvert$  for $v\in \Gamma$. Then \cref{thm:unique-prime-factorization} obtains for $M_{\Gamma} = *_{v,\Gamma}(M_v,\tau_v)$ the unique prime factorization $M_{\Gamma} = M_{\Gamma_1}\overline{\otimes} M_{\Gamma_2}$, where  $\Gamma_1 = \{a,b,c,d,e\}$ and $\Gamma_2 = \{f,g,h,i,j\}$ are the irreducible components of $\Gamma$. }
     \label{figure:graph-UPF-result}
\end{figure}
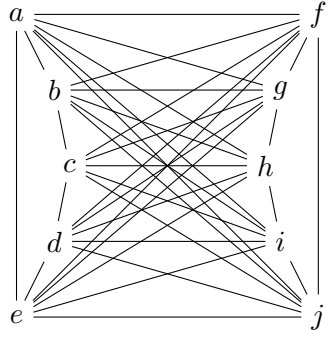

\subsection{Primeness results for other graph products}
\label{sub:primeness-general-setting}
In case the von Neumann algebras $M_v$ are not (all) type II$_1$-factors, it is interesting to know whether the condition $M_{\Gamma}\not\prec_{M_{\Gamma}} M_{\Gamma_0}$ for any strict subgraph
$\Gamma_0\subsetneq \Gamma$, is satisfied. In \cref{lemma:ICC-group-not-embed-in-smaller-graph} we will give sufficient conditions for the property to hold. To prove this, we need the following lemma.
\begin{lemma}\label{lemma:expectation-conjugation}
	Let $\Gamma$ be a simple graph and for $v\in \Gamma$ let $(M_{v},\tau_v)$ be a tracial von Neumann algebra. Let $\Lambda\subseteq \Gamma$ be a subgraph and let $\uu\in \calW_{\Gamma}\setminus \calW_{\Lambda}$. Let $\vv,\vv'\in \calW_{\Gamma}$ be such that every letters at the start of $\vv$ respectively $\vv'$ does not commute with any letters at the end of $\uu^{-1}$ respectively $\uu$. Let $\ww,\ww'\in \calW_{\Gamma}$ with $|\ww|\leq |\vv|$ and $|\ww'|\leq |\vv'|$. Then 
	\begin{align}
		\EE_{M_{\Lambda}}(axb) = 0 &\quad \quad \text{ for } a\in \mathring{M}_{\ww}, x \in \mathring{M}_{\vv^{-1} \uu\vv'}, b\in \mathring{M}_{\ww'}.
	\end{align}
	\begin{proof}
		Let $\uu,\vv,\vv',\ww,\ww'$ be given as stated. Observe by the assumptions on $\vv$ and $\vv'$ that in particular $\vv^{-1}\uu\vv'$ is reduced. 
		Denote 
		\begin{align}
			\calH(\uu) &:=\bigoplus_{\ww_0\in \calW(\uu)} \mathring{\calH}_{\ww_0},	&
			\boldM(\uu) &:= \bigoplus_{\ww_0\in \calW(\uu)} \mathring{\boldM}_{\ww_0}.	
		\end{align}	
		Observe for $y_1\in \lambda(\boldM(\uu^{-1}))$, $y_2\in \mathring{M}_{\uu}$ and $y_3\in \lambda(\boldM(\uu))$ that if we denote $y:=y_1^*y_2y_3$ and write $y = \sum_{\ww\in \calW_{\Gamma}}y_{\ww}$ where $y_{\ww}\in \mathring{M}_{\ww}$, then we have that $y_{\ww} = 0$ whenever $\ww$ does not contain $\uu$ as a subword. Thus, in particular $\EE_{M_{\Lambda}}(y_1^*y_2y_3) = 0$. We will apply this to obtain the result.

		Let $x\in \mathring{\boldM}_{\vv^{-1}\uu\vv'}$ be a pure tensor, and let $x_1\in \mathring{\boldM}_{\vv}$, $x_2\in \mathring{\boldM}_{\uu}$ and  $x_3\in \mathring{\boldM}_{\vv'}$ be s.t.$\lambda(x) = \lambda(x_1)^*\lambda(x_2)\lambda(x_3)$. Let $a\in \mathring{M}_{\ww}$ and  $b\in \mathring{M}_{\ww'}$. 
		Let $\omega \in \calS_{\vv'}$, then we can write $\omega = (\vv_1',\vv_2',\vv_3')$ for some $\vv_1',\vv_2',\vv_3'\in \calW_{\Gamma}$ with $\vv' = \vv_1'\vv_2'\vv_3'$. 

        By \cref{lemma:word_action_non_zero} we have $\eta_{\omega}:=\lambda_{(\vv_1',\vv_2',\vv_3')}(x_3)b\Omega\in \mathring{\calH}_{\vv_0'}$ where $\vv_0' = \vv_1'\vv_3'\ww'$. We show that $\eta_{\omega}\in \calH(\uu)$. In particular, we can assume that $\eta_{\omega}$ is non-zero, so that $\ww'$ starts with $(\vv_3')^{-1}\vv_2'$ and $\vv_0'$ starts with $\vv_1'\vv_2'$. If $\vv_1'\vv_2' =e$ then $\vv_3' = \vv'$, so that $|\vv_3'| + |\vv_3'\ww'| = |\ww'|\leq |\vv'| = |\vv_3'|$ and therefore $\vv_3'\ww' = e$. We then conclude that $\eta_{\omega}\in \mathring{\calH}_{e}\subseteq \calH(\uu)$. Thus, suppose $\vv_1'\vv_2'\not=e$.
	Then $\vv_1'\vv_2'\ww_0'$ ($=\vv_0'$) starts with a letter $v_0'$ at the start of $\vv'$. Now, by the assumption on $\vv'$ we obtain that $v_0'$ does not commute with elements at the end of $\uu$. This implies that $\uu\vv_0'$ is reduced and so $\eta_{\omega}\in \calH(\uu)$. Now, as $\lambda(x_3)\lambda(b)\Omega = \sum_{\omega\in \calS_{\vv'}}\lambda_{\omega}(x_3)\lambda(b)\Omega \in \calH(\uu)$ we obtain that $y_3:=\lambda(x_3)\lambda(b)\in \boldM(\uu)$. In a similar way we obtain $y_1:= \lambda(x_1)\lambda(a)^*\in \boldM(\uu^{-1})$. Hence, putting $y_2 :=\lambda(x_2)$ we obtain that $\EE_{M_{\Lambda}}(\lambda(a)\lambda(x)\lambda(b)) = \EE_{M_{\Lambda}}(y_1^*y_2y_3)=0$. The result now follows by density of $\lambda(\mathring{\boldM}_{\zz})\subseteq \mathring{M}_{\zz}$ for $\zz\in \calW_{\Gamma}$.
    \end{proof}
\end{lemma}

\begin{corollary}
	Let $\Gamma$ be a simple graph, $\Lambda\subseteq \Gamma$ be a subgraph and let $\uu\in \calW_{\Gamma}\setminus \calW_{\Lambda}$. Let $\vv,\vv'\in \calW_{\Gamma}$ be such that every letters at the start of $\vv$ respectively $\vv'$ does not commute with any letters at the end of $\uu^{-1}$ respectively $\uu$. Let $\ww,\ww'\in \calW_{\Gamma}$ with $|\ww|\leq |\vv|$ and $|\ww'|\leq |\vv'|$. Then $\ww\vv^{-1} \uu\vv'\ww'\not\in \calW_{\Lambda}$.
		\begin{proof}
			 For $v\in \Gamma$ let $M_{v} := \calL(\ZZ/2\ZZ)$, so that $M_{\Gamma}= \calL(\calW_{\Gamma})$. Take $a = \lambda_{\ww}$, $x = \lambda_{\vv^{-1}\uu\vv'}$ and $b = \lambda_{\ww'}$. Then \cref{lemma:expectation-conjugation} shows that  $\EE_{M_{\Gamma\setminus \Lambda}}(\lambda_{\ww\vv^{-1}\uu\vv'\ww}) = \EE_{M_{\Lambda}}(axb) = 0$. This means that $\ww_1\vv^{-1}\uu\vv\ww_2 \not\in \calW_{\Lambda}$.  
		\end{proof}
\end{corollary}

\begin{lemma}\label{lemma:ICC-group-not-embed-in-smaller-graph}
	Let $\Gamma$ be a simple graph of size $|\Gamma|\geq 3$ such that for any $v\in \Gamma$, $Star(v)\not=\Gamma$. For $v\in \Gamma$ let $(M_{v},\tau_{v})$ be a von Neumann algebra with a normal faithful trace. Suppose for any $v\in \Gamma$ there is a unitary $u_{v}\in M_{\Gamma}$ with $\tau_{v}(u_v)=0$. 
	Then $M_{\Gamma}\not\prec_{M_{\Gamma}}M_\Lambda$ for any strict subgraph $\Lambda\subsetneq \Gamma$.
	\begin{proof}
		First observe that the Coxeter group $\calW_{\Gamma}$ is icc since $|\Gamma|\geq 3$ and $\Star(v)\not=\Gamma$ for all $v\in \Gamma$. 
		Now let $\Lambda\subsetneq \Gamma$ be a strict subgraph and fix $v\in \Gamma\setminus \Lambda$. As the conjugacy class $\{\vv^{-1} v \vv: \vv\in \calW_{\Gamma}\}$ is infinite, we can for $n\in \NN$ choose $\vv_n\in \calW_{\Gamma}$ such that $|\vv_n^{-1} v\vv_n|\geq 2n+1$. If a letters $s$ commuting with $v$ is at the start of $\vv_n$ then we can replace $\vv_n$ with $\widetilde{\vv}_n := s\vv_n \in \calW_{\Gamma}$ which does not start with $s$ and is such that $\widetilde{\vv}_n^{-1}v\widetilde{\vv}_n= \vv_n^{-1}v\vv_n$. Repeating the argument, we may thus assume that every letter at the start of $\vv_n$ does not commute with $v$. Then in particular $\vv_n^{-1}v\vv_n$ is reduced and $|\vv_n|\geq n$. Let $(v_{n,1},\ldots, v_{n,l_n})$ be a reduced expression for $\vv_n^{-1} v\vv_n$ and define $u_{n} := u_{v_{n,1}}\ldots u_{v_{n,l_n}}\in \mathring{M}_{\vv_n^{-1}\uu\vv_n}$. Then $u_n$ is a unitary and for any $\ww,\ww'\in \calW_{\Gamma}$ with $|\ww|,|\ww'|\leq n$ and $a\in \mathring{M}_{\ww}$, $b\in \mathring{M}_{\ww'}$, we have by \cref{lemma:expectation-conjugation} that
		 \begin{align}\label{eq:conditional-expectation-conjugation-is-zero}
		 	\EE_{M_{\Lambda}}(au_nb) =0.
		 \end{align}
		 We take $x,y\in M_{\Gamma}$ and $\epsilon>0$. We can choose $x_0\in M_{\Gamma}$ of the form $x_0 = \sum_{i=1}^{K_1}x_i$ for some $K_1\geq 1$, $x_i\in \mathring{M}_{\ww_i}$ with some $\ww_i\in \calW_{\Gamma}$,  and with $\|y\|\cdot\|x_0 - x\|_2\leq \epsilon$. We can now also choose $y_0\in M_{\Gamma}$ of the form  $y_0 = \sum_{i=1}^{K_2}y_{i}$ for some $K_2\geq 1$, $y_{i}\in \mathring{M}_{\ww_i'}$, with some $\ww_i'\in W_{\Gamma}$ and $\|x_0\|\cdot\|y_0 - y\|_2\leq \epsilon$. Put $l_1:= \max_{1\leq i\leq K_1} |\ww_{i}|$, $l_2:= \max_{1\leq i\leq K_2} |\ww_{i}'|$ and $l = \max\{l_1,l_2\}$. Let $n\geq l$ so that by \eqref{eq:conditional-expectation-conjugation-is-zero} and linearity we have $\EE_{M_{\Lambda}}(x_0u_ny_0) =0$ and hence
		 \begin{align}
		 	\EE_{M_{\Lambda}}(xu_ny) = \EE_{M_{\Lambda}}((x - x_0)u_ny) + \EE_{M_{\Lambda}}(x_0u_n(y-y_0)).
		 \end{align}
 		Furthermore,
	 	\begin{align}
	 		\|(x-x_0)u_ny\|_{2} &\leq \|x-x_0\|_{2}\cdot\|u_ny\|\leq \epsilon,\\
	 		\|x_0u_n(y-y_0)\|_{2}& \leq \|x_0u_n\|\cdot\|y-y_0\|_{2}\leq \epsilon.
	 	\end{align}
	 	Thus, as the conditional expectation $\EE_{M_{\Lambda}}$ is $\|\cdot\|_{2}$-decreasing (this follows from the Schwarz inequality \cite[Proposition 3.3]{paulsenCompletelyBoundedMaps2002} as $\EE_{M_{\Lambda}}$ is trace-preserving and u.c.p.), we obtain for $n\geq l$ that 
	 	\begin{align}
	 		\|\EE_{M_{\Lambda}}(xu_ny)\|_{2}\leq 2\epsilon.
	 	\end{align} This shows for any $x,y\in M_{\Gamma}$ that $\|\EE_{M_{\Lambda}}(x u_n y)\|_{2}\to 0$ as $n\to \infty$.
	 	By \cref{Dfn=Intertwine}\eqref{Item=Intertwine2} this means that $M_{\Gamma}\not\prec_{M_{\Gamma}}M_{\Lambda}$. 
	\end{proof}
\end{lemma}

\begin{theorem}\label{thm:primeness:irreducible-graph}
Let $\Gamma$ be a irreducible finite simple graph of size $|\Gamma|\geq 3$ and for $v\in \Gamma$, let $M_{v}$ ($\not=\CC$) be a von Neumann algebra with a normal faithful trace $\tau_{v}$ such that there exists a unitary $u_v\in M_{v}$ with $\tau_v(u_v)=0$. Then $M_{\Gamma}$ is a prime factor.
\begin{proof}
    It follows from \cite[Theorem E]{charlesworthStructureGraphProduct2024} and our assumptions that $M_{\Gamma}$ is a II$_1$-factor. Furthermore, by \cref{lemma:ICC-group-not-embed-in-smaller-graph} we have that $M_{\Gamma}\not\prec_{M_{\Gamma}}M_{\Lambda}$ for any strict subgraph $\Lambda\subsetneq \Gamma$. Hence, by \cref{lemma:irreducible-graph-prime-or-amenable} we obtain that $M_{\Gamma}$ is either prime or amenable. Since $\Gamma$ is irreducible and $|\Gamma|\geq 3$ it follows from \cref{Thm=AmenableGraphProduct} that $M_{\Gamma}$ is non-amenable. Hence, $M_{\Gamma}$ is prime.
\end{proof}
\end{theorem}

\begin{theorem}\label{thm:primeness:graph-product}
    Let $\Gamma$ be a simple graph. For $v\in \Gamma$, let $M_{v}$ ($\not=\CC$) be a von Neumann algebra with a normal faithful trace $\tau_{v}$ and assume that $M_{\Gamma} = *_{v,\Gamma}(M_{v},\tau_v)$ is a II$_1$-factor. Then $M_{\Gamma}$ is prime if and only if  there is an irreducible component $\Lambda\subseteq \Gamma$ for which $M_{\Lambda}$ is prime and $M_{\Gamma\setminus \Lambda}$ is finite-dimensional.
    \begin{proof}
        Suppose there is an irreducible component  $\Lambda\subseteq \Gamma$ with
        $M_{\Lambda}$ prime and with $\dim M_{\Gamma\setminus \Lambda}<\infty$. Then the factor $M_{\Gamma} = M_{\Lambda}\overline{\otimes} M_{\Gamma\setminus \Lambda}$ is prime since it is a matrix amplification of $M_{\Lambda}$.

        For the other direction, suppose that $M_{\Gamma}$ is a prime factor. Denote 
        $$\Lambda := \{v\in \Gamma: \Star_{\Gamma}(v)\not=\Gamma \text{ or } \dim M_{v}=\infty\}.$$ If $w\in \Gamma\setminus \Lambda$ then $\Star_{\Gamma}(w) = \Gamma$ and $\dim M_w<\infty$, so $w\in \Link_{\Gamma}(\Lambda)$. Hence $\Link_{\Gamma}(\Lambda) = \Gamma\setminus \Lambda$ and so $M_{\Gamma} = M_{\Lambda}\overline{\otimes} M_{\Gamma\setminus \Lambda}$. Now, since $\Gamma\setminus \Lambda$ is complete, and since $\dim M_v<\infty$ for $v\in \Gamma\setminus \Lambda$ we have that $M_{\Gamma\setminus \Lambda}$ is finite-dimensional. Hence, since $M_{\Gamma}$ is a prime factor also $M_{\Lambda}$ is a prime factor. 
        
        We now show that the graph $\Lambda$ is irreducible so that from $\Link_{\Gamma}(\Lambda) = \Gamma\setminus \Lambda$ it follows that $\Lambda$ is an irreducible component of $\Gamma$. Suppose there is a non-empty subgraph $\Lambda_1\subseteq \Lambda$ s.t. $\Lambda_2 := \Lambda\setminus \Lambda_1$ is non-empty and  $\Link_{\Lambda}(\Lambda_1) = \Lambda_2$. We show a contradiction. We can write $M_{\Lambda} = M_{\Lambda_1}\overline{\otimes} M_{\Lambda_2}$. Hence, by primeness of the factor $M_{\Lambda}$ there is $i\in \{1,2\}$ s.t. $\dim M_{\Lambda_i}<\infty$. Let $v\in \Lambda_i$. Since $\dim M_{\Lambda_i}<\infty$ we have $\dim M_{v}<\infty$. Hence, since $v\in \Lambda$ we have by definition of $\Lambda$ that $\Star_{\Gamma}(v)\not=\Gamma$. Let $w\in \Gamma\setminus \Star_{\Gamma}(v)$. Then $\Star_{\Gamma}(w)\not=\Gamma$ so that $w\in \Lambda$. Furthermore, $w\not\in \Link_{\Gamma}(v)$ so that $w\not\in \Link_{\Lambda}(\Lambda_i) = \Lambda\setminus \Lambda_i$, i.e. $w\in \Lambda_i$. Hence, since the vertices $v,w$ in $\Lambda_i$ share no edge we have $\dim M_{\Lambda_i} = \infty$, which is a contradiction. Thus $\Lambda$ is irreducible.
    \end{proof}
\end{theorem}

\section{Classification of free indecomposability for graph products}
\label{Sect=freely-indecomposable}
In this section we study free-indecomposability for graph product of II$_1$-factors.
In \cref{theorem:free-indecompose} we characterize for graph products of II$_1$-factors (with separable predual) when they can  decompose as tracial free products of II$_1$-factors. In \cref{thm:free-product-decomposition} we combine this result with \cref{thm:rigid-graph-decomposition} to show unique free product decompositions for von Neumann algebras in the class $\Crigid\setminus \Cvert$.  Hereafter, we show that \cref{theorem:free-indecompose} and \cref{thm:free-product-decomposition} really cover new examples. Indeed, in \cref{prop:Cartan-subalgebra} we give sufficient conditions for a graph product to not possess a Cartan-subalgebra, which in \cref{remark:class-anti-free} we use to give examples of freely indecomposable von Neumann algebras $M\in \Crigid\setminus\Cvert$ that are not in the class $\calC_{\text{anti-free}}$ from \cite{houdayerRigidityFreeProduct2016}. 
In \cref{remark:covers_new_ufd} we show that the unique free product decomposition from \cref{thm:free-product-decomposition} also covers new examples.

\begin{theorem}\label{theorem:free-indecompose}
    Let $\Gamma$ be a simple graph of size $|\Gamma|\geq 2$, and for $v\in \Gamma$ let $(M_v,\tau_v)$ be tracial II$_1$-factor with separable predual. Then the graph product $M_{\Gamma}:= *_{v,\Gamma}(M_{v},\tau_v)$ can decompose as a tracial free product $M_{\Gamma} = (M_1,\tau_1)*(M_2,\tau_2)$ of II$_1$-factors  $M_1$,$M_2$ if and only if $\Gamma$ is not connected.
    
    \begin{proof}
    Let $\Gamma$ and $(M_{v},\tau_v)_{v\in \Gamma}$ be given as stated. If $\Gamma$ is not connected then for any connected component $\Gamma_0$ of $\Gamma$ we have $M_{\Gamma} = (M_{\Gamma_0},\tau_1)* (M_{\Gamma\setminus \Gamma_0},\tau_2)$, which shows one direction. 
    
    For another direction suppose that $\Gamma$ is connected. Assume we can write $M_\Gamma =(M_1,\tau_1)*(M_2,\tau_2)$ for some II$_1$-factors $M_1,M_2$.
    Fix $v\in \Gamma$ and by \cite[Proposition 13]{ozawaPrimeFactorizationResults2004} let $N_0\subseteq M_{v}$ be an amenable II$_1$-subfactor with $N_0' \cap M_{\Gamma} = M_v' \cap M_{\Gamma}$. Then $N_0$ is amenable relative to $M_i$ inside $M$ for $i=1,2$. Therefore, by 
    \cref{thm:prelim:alternatives} one of the following holds true:
    \begin{enumerate}
        \item $N_0\prec_{M_{\Gamma}} \CC 1_{M_{\Gamma}}$;\label{Item:free-indecomposability:option1}
        \item $\Nor_{M_{\Gamma}}(N_0)''\prec_{M_{\Gamma}} M_i$ for some $1\leq i\leq 2$;\label{Item:free-indecomposability:option2}
        \item $\Nor_{M_{\Gamma}}(N_0)''$ is amenable relative to $\CC 1_{M_{\Gamma}}$ inside $M_{\Gamma}$.\label{Item:free-indecomposability:option3}
    \end{enumerate}
    Since $N_0$ is diffuse, we can not have \eqref{Item:free-indecomposability:option1}.

  We show that \eqref{Item:free-indecomposability:option2} is also not satisfied. Suppose $\Nor_{M_{\Gamma}}(N_0)''\prec_{M_{\Gamma}} M_1$. We first argue that $\Nor_{M_{\Gamma}}(N_0)''$ unitarily conjugates into $M_1$ through an application of   \cref{lemma:unitary-inclusion-in-star-for-vNa} to the case of a free product of two II$_1$ factors and so the ambient graph in that theorem consists of two points and no edges (this is essentially \cite[Proof of Theorem 3.3]{OzawaKurosh}). In \cref{lemma:unitary-inclusion-in-star-for-vNa}  we further take $Q = N_0$ and note that $N_0$ and $N_0'\cap M_{\Gamma} = M_{v}'\cap M_{\Gamma} = M_{\Link(v)}$ are indeed factors. For $\Lambda$ we take the single vertex corresponding to the first free product component and $\{ \Lambda_j \}_{j \in \mathcal{J}} = \{ \varnothing \}$ so there is only one index in  $\mathcal{J}$.   By assumption $N_0 \prec_{M_{\Gamma}} M_1$  and as $N_0$ is diffuse $N_0 \not \prec_{M_{\Gamma}} M_{\varnothing}$. Then   \cref{lemma:unitary-inclusion-in-star-for-vNa} gives that  $u^* N_0 u\subseteq M_1$  for some unitary $u\in M_{\Gamma}$. Hence      $u^*\Nor_{M_{\Gamma}}(N_0)''u\subseteq M_1$.

    Now take $w\in \Gamma$ arbitrarily. Since $\Gamma$ is connected there is a path $P$ from $v$ to $w$, i.e. $P=(v_0,v_1,\ldots, v_n)$ for some $n\geq 0$ and vertices $v_0,v_1,\ldots, v_n\in \Gamma$ such that $v_i\in \Link(v_{i-1})$ for $1\leq i\leq n$ and such that $v_0 =v$ and $v_n =w$. As $|\Gamma|\geq 2$ we can moreover choose this path such that it has length $n\geq 1$.
    
    For $i\in\{1,\ldots,n\}$, denote $N_i := M_{v_i}$. Then, since $v_i\in \Link(v_{i-1})$ we obtain $N_i\subseteq \Nor_{M_{\Gamma}}(N_{i-1})''$.
    Since $u^*\Nor_{M_{\Gamma}}(N_0)''u\subseteq M_1$ we obtain $u^*N_1u\subseteq M_1$. Then since  $u^*N_1u\not\prec_{M_{\Gamma}}M_{\varnothing}$ (since $u^*N_1u$ is diffuse) we obtain by \cref{prop:embed}\eqref{Item=Embed1b}  with $\Lambda$ and $\{ \Lambda_j \}_{j \in \mathcal{J}} = \{ \varnothing \}$  the same as above that $\Nor_{M_{\Gamma}}(u^*N_1u)''\subseteq M_1$  (note that this also follows from \cite[Theorem 1.1]{ioanaAmalgamatedFreeProducts2008}).
    Now, observe that $\Nor_{M_{\Gamma}}(u^*N_1u) = u^*\Nor_{M_{\Gamma}}(N_1)u$ and hence $\Nor_{M_{\Gamma}}(u^*N_1u)'' = u^*\Nor_{M_{\Gamma}}(N_1)''u$.    
    We thus obtain that $u^*\Nor_{M_{\Gamma}}(N_1)''u\subseteq M_1$. Continuing in this way we obtain $u^*\Nor_{M_{\Gamma}}(N_{i})''u\subseteq M_1$ for all $0\leq i\leq n$. Thus, in particular $u^*M_{w}u\subseteq u^*\Nor_{M_{\Gamma}}(N_{n-1})''u\subseteq M_{1}$. Since $w$ was arbitrary, we obtain that $M_{w}\subseteq uM_1u^*$ for each $w\in \Gamma$. But this implies $M_{\Gamma} = (\bigcup_{w\in \Gamma}M_{w})'' \subseteq uM_1u^*$. Hence $M_{\Gamma} = M_1$, which is a contradiction. We conclude that $\Nor_{M_{\Gamma}}(N_0)''\not\prec_{M_{\Gamma}} M_1$. By symmetry also $\Nor_{M_{\Gamma}}(N_0)''\not\prec_{M_{\Gamma}} M_2$. We obtain that  \eqref{Item:free-indecomposability:option2} is not satisfied.

    We conclude that \eqref{Item:free-indecomposability:option3} is satisfied, i.e. $\Nor_{M_{\Gamma}}(N_0)''$ is amenable. Hence $M_{\Link(v)}\subseteq \Nor_{M_{\Gamma}}(N_0)''$ is amenable as well.  Therefore, by \cref{Thm=AmenableGraphProduct} we obtain that $\Link(v)$ is a clique and that $M_{w}$ is amenable for any $w\in \Link(v)$. We observe that $v\in \Gamma$ was arbitrary, thus for each vertex $z\in \Gamma$ its $\Link(z)$ is a clique. Since $\Gamma$ is connected, it follows that $\Gamma$ is a complete graph   (see Lemma \ref{Lem=CompleteAdd}). Moreover, for any $v\in \Gamma$ choose $z\in \Gamma\setminus \{v\}$ we have $M_{v}\subseteq M_{\Link(z)}$, which shows that $M_{v}$ is amenable. Hence $M_{\Gamma}$ is a tensor product of amenable II$_1$-factors and so $M_{\Gamma}$ is amenable.
    But the amenable II$_1$-factor can not decompose as a free product of type II$_1$-factors. This gives a contradiction and we conclude that $M_{\Gamma}$ can not  decompose as free product of II$_1$-factors.
    \end{proof}
\end{theorem}

\begin{theorem}\label{thm:free-product-decomposition}
     Any von Neumann algebra $M\in \Crigid\setminus \Cvert$ can decompose as tracial free product inside $\Crigid\setminus \Cvert$:
     \begin{align}\label{eq:free-product-decomposition}
         M =*_{i\in I} M_i,
     \end{align}
     for some index set $I$ and for every $i\in  I$ a II$_1$-factor $M_i\in \Crigid\setminus \Cvert$ that can not decompose as any tracial free product of II$_1$-factors. 
     
     Furthermore, suppose $M$ has another free product decomposition:$$M=*_{j\in J}N_j,$$ for another index set $J$ and for every $j\in J$ a II$_1$-factor $N_j\in \Crigid\setminus\Cvert$ that can not decompose as tracial free product of II$_1$-factors. Then $|I|=|J|$ and there is a bijection $\sigma$ between $J$ and $I$ such that for each $j \in J$, $N_j$ is unitarily conjugate to $M_{\sigma(j)}$ in $M$.
    \begin{proof}
        Since $M\in \Crigid\setminus \Cvert$ we can write $M = M_{\Gamma} = *_{v,\Gamma}(M_v,\tau_v)$ for some rigid graph $\Gamma$ of size $|\Gamma|\geq 2$ and some II$_1$-factors $M_v\in \Cvert$. Let $\{\Gamma_i\}_{i\in I}$ be the connected components of $\Gamma$ for some index set $I$, which are rigid by \cref{remark:rigid-components}. We let $\Pi = \{i\}_{i\in I}$ be the graph with $m$ vertices and no edges. We claim that $|\Gamma_i|\geq 2$ for all $i\in \Pi$. Indeed, if $|I|=1$ then $\Pi = \{1\}$ and $\Gamma_1 = \Gamma$ so that $|\Gamma_i| = |\Gamma|\geq 2$ for all $i\in \Pi$. On the other hand, if $|I|\geq 2$ then $\Link_{\Pi}(\Link_{\Pi}(i)) = \Pi\not=\{i\}$ for all $i\in \Pi$, so it follows from \cref{lemma:graph-product-of-rigid-graphs} and rigidity of $\Gamma_{\Pi}\simeq \Gamma$ that $|\Gamma_i|\geq 2$ for all $i\in \Pi$.  
        
        We denote $M_i := M_{\Gamma_i}\in \Crigid$ for $i\in \Pi$. By \cref{thm:rigid-graph-decomposition} and rigidity of $\Gamma_i$ and the fact that $|\Gamma_i|\geq 2$ it follows that $M_i\not\in \Cvert$. Furthermore, since $\Gamma_i$ is connected we obtain by \cref{theorem:free-indecompose} that $M_i$ can not decompose as tracial free product of II$_1$-factors. By \cref{remark:graph-product-of-graphs-consistency} we conclude that
        $M_{\Gamma} = *_{v,\Gamma}(M_{v},\tau_v) = *_{i,\Pi}(M_{\Gamma_i},\tau_i) =*_{i\in I} M_i$
        which shows \eqref{eq:free-product-decomposition}.
       
        Since for every $j\in J$ $N_j\in \Crigid\setminus\Cvert$ we can write $N_j = *_{z,\Lambda_j}(N_{(j,z)},\tau_{(j,z)})$ where $\Lambda_j$ is a rigid graph and $(N_{(j,z)})_{z\in \Lambda_j}$ are II$_1$-factors in $\Cvert$. Observe for $j\in J$ that $|\Lambda_j|\geq 2$ since $N_j\not\in \Cvert$ and that $\Lambda_j$ is connected by \cref{theorem:free-indecompose} since $N_i$ can not decompose as tracial free product of II$_1$-factors. 
        Let $\Pi' = \{j\}_{j\in J}$ be the graph with vertices of all point in $J$ and no edges.
        Then by \cref{remark:graph-product-of-graphs-consistency} we have:
        $$M = *_{j\in J}N_j= *_{j, \Pi'}(*_{v,\Lambda_j}(N_{(j,v)},\tau_{(j,v)})) \simeq *_{w,\Lambda_{\Pi'}}(N_{w},\tau_{w}) = N_{\Lambda_{\Pi'}}.$$ 
        Then since $\Lambda_{\Pi'}$ is rigid by \cref{lemma:graph-product-of-rigid-graphs}, we obtain by \cref{thm:rigid-graph-decomposition} that $\Lambda_{\Pi'} \simeq \Gamma$. The connected components of $\Lambda_{\Pi'}$ respectively $\Gamma$ are $\{\Lambda_j\}_{j\in J}$ respectively $\{\Gamma_i\}_{i\in I}$. Hence $|I|=|J|$. Moreover, \cref{thm:rigid-graph-decomposition} asserts, for some bijection $\sigma$ of between $J$ and $I$, that  $\widetilde{N}_{\Lambda_j}$ ($=N_{j}$) is unitarily conjugate to $M_{\Gamma_{\sigma(j)}}$ ($=M_{\sigma(j)}$) in $M_{\Gamma}$.
    \end{proof}
\end{theorem}

  We give sufficient conditions for absence of Cartan subalgebras in graph products. We note that in \cite{caspersAbsenceCartanSubalgebras2020} absence of Cartan was studied for right-angled Hecke algebras and that in \cite{chifanCartanSubalgebrasNeumann2023} absence of Cartan was fully characterized for von Neumann algebras associated to graph products of groups.

For a non-empty connected graph $\Gamma$ we define its radius as 
\begin{align}\label{definition:radius-graph}
    \Radius(\Gamma):=\inf_{s\in \Gamma}\sup_{t\in \Gamma}\Dist_{\Gamma}(s,t),
\end{align}
where $\Dist_{\Gamma}(s,t)$ denotes the minimal length of a path in $\Gamma$ from $s$ to $t$. Furthermore, we set $\Radius(\Gamma) = 0$ if $\Gamma$ is empty and set $\Radius(\Gamma) = \infty$ if $\Gamma$ is not connected.
\begin{proposition}\label{prop:Cartan-subalgebra}
    Let $\Gamma$ be a simple graph with $\Radius(\Gamma)\geq 3$ and for $v\in \Gamma$ let $M_{v}$ be a II$_1$-factor with normal faithful trace $\tau_v$. Then $M_{\Gamma} = *_{v,\Gamma}(M_v,\tau_v)$ does not possess a Cartan-subalgebra.
    \begin{proof}
        Suppose $M_{\Gamma}$ has a Cartan subalgebra $A\subseteq M_{\Gamma}$. Fix $v\in \Gamma$. Then $M_{\Gamma} = M_{\Star(v)}*_{M_{\Link(v)}}M_{\Gamma\setminus\{v\}}$. Since $A$ is amenable, one of the statements of \cref{thm:prelim:alternatives} must hold.
        Since $\Radius(\Gamma)\geq 3$, we have $\Star(v)\not=\Gamma$. Hence $\Nor_{M_{\Gamma}}(A)'' = M_{\Gamma}\not\prec_{M_{\Gamma}} M_{\Star(v)}$ and $\Nor_{M_{\Gamma}}(A)''  = M_{\Gamma}\not\prec_{M_{\Gamma}} M_{\Gamma\setminus 
    \{v\}}$ by \cref{Lem=Selfembedding}. Thus we must either have $A\prec_{M_{\Gamma}} M_{\Link(v)}$ or $\Nor_{M_{\Gamma}}(A)''$ is amenable relative to $M_{\Link(v)}$ inside $M_{\Gamma}$. 
    Suppose that $A\prec_{M_{\Gamma}} M_{\Link(v)}$ then since $A\not\prec_{M_{\Gamma}} M_{\varnothing}$ we obtain by \cref{prop:embed}\eqref{Item=Embed1b} that $M_{\Gamma} =\Nor_{M_{\Gamma}}(A)''\subseteq M_{\Lambda_{\emb}}$ where $\Lambda_{\emb} = \Link(v) \cup \bigcup_{w\in \Link(v)}\Link(w)$.
    We see that $\Radius(\Lambda_{\emb})\leq 2$ (indeed take as center $v$). Hence, since $\Radius(\Gamma)\geq 3$ we have $M_{\Gamma} = \Nor_{M_{\Gamma}}(A)''\subseteq M_{\Lambda_{\emb}} \subsetneq M_{\Gamma}$, a contradiction. We conclude that $\Nor_{M_{\Gamma}}(A)''$ ($=M_{\Gamma}$) is amenable relative to $M_{\Link(v)}$ in $M_{\Gamma}$. Since $v$ was arbitrary we obtain using \cref{Thm=Square} that $M_{\Gamma}$ is amenable. This is a contradiction since $M_{\Gamma}$ is non-amenable by \cref{Thm=AmenableGraphProduct} (since $\Radius(\Gamma)\geq 3$).  Thus $M_{\Gamma}$ does not have a Cartan subalgebra.
    \end{proof}
\end{proposition}

\begin{remark}\label{remark:class-anti-free}
We argue that we find new classes of finite von Neumann algebras that are freely indecomposable. More precisely we argue that Theorem  \ref{theorem:free-indecompose} covers von Neumann algebras that are not in the class $\calC_{\text{anti-free}}$ from \cite{houdayerRigidityFreeProduct2016}.  Indeed, let $\Gamma$ be a graph with  $\Radius(\Gamma)\geq 3$ (hence $\Gamma$ is irreducible) and for $v\in \Gamma$ let $M_{v}$ be a II$_1$-factor with separable predual and possessing the Haagerup property. Then the II$_1$-factor $M_{\Gamma}$ does not lie in the class $\calC_{\text{anti-free}}$ from \cite{houdayerRigidityFreeProduct2016}. Indeed, (i) $M_{\Gamma}$ is prime by \cref{thm:primeness-for-graph-products-II1-factors}, (ii)  $M_{\Gamma}$ is full (so no property Gamma) by \cite[Theorem E]{charlesworthStructureGraphProduct2024}, (iii)  $M_{\Gamma}$  does not have a Cartan subalgebra by \cref{prop:Cartan-subalgebra}, and (iv)  $M_{\Gamma}$ has the Haagerup property (so no property (T) by \cite[Theorem 3]{connesPropertyNeumannAlgebras1985}) by \cite[Theorem 0.2]{caspersGraphProductsOperator2017a}. If $\Gamma$ is moreover connected and rigid and if $M_v$ lies in $\Cvert$ for each $v\in \Gamma$, then $M_{\Gamma}$ lies in $\Crigid$ and can not decompose as free product of II$_1$-factors.  As a concrete example, take the cyclic graph $\Gamma = \ZZ_{n}$ for some $n\geq 6$ and for each $v\in \Gamma$ let $M_v = \calL(\FF_2)\in \Cvert$ which has the Haagerup property by \cite[Theorem 12.2.5]{brownAlgebrasFiniteDimensionalApproximations2008}. Then $M_{\Gamma}$ is a II$_1$-factor in $\calC_{\Rigid}\setminus \calC_{\text{anti-free}}$ that can not decompose as a (tracial) reduced free product of II$_1$-factors.  
\end{remark}
 
\begin{remark}\label{remark:covers_new_ufd}
    We argue that the unique free product decompositions from \cref{thm:free-product-decomposition} are not covered by \cite{houdayerRigidityFreeProduct2016} nor \cite{dingStructureRelativelyBiexact2024}. Indeed, let $\Gamma$ be a simple graph whose connected components $\Gamma_i$ for $i=1,\ldots, m$ are of the form $\ZZ_{n_i}$ for some $n_i\geq 6$. Observe that $\Gamma$ is rigid. For $v\in \Gamma$ put $M_{v} = \calL(\FF_2)\in \Cvert$. Then \cref{thm:free-product-decomposition} asserts the unique free product decomposition $M_{\Gamma} = M_{\Gamma_1}*\cdots *M_{\Gamma_m}$. Since the factors $M_{\Gamma_i}$ for $i=1,\ldots, m$ are not in the class $\Cantifree$, this result is not covered by \cite{houdayerRigidityFreeProduct2016}. 
    Furthermore, we note for $i=1,\ldots, m$ that the group $*_{v,\Gamma_i}\FF_2$ is properly proximal by \cite[Proposition 3.7]{dingProperProximalityVarious2024} since $\Radius(\Gamma_i)\geq 3$. Hence, also \cite[Corollary 1.8]{dingStructureRelativelyBiexact2024} does not apply.
\end{remark}

 \section{Graph radius rigidity}\label{section:graph-radius-rigidity}
In this section we generalize the ideas from the proof of \cref{theorem:free-indecompose} and show that we can, in certain cases, retrieve the radius of the graph $\Gamma$ from the graph product $M_{\Gamma}$. In \cref{subsection:radius-of-von-Neumann-algebra} we introduce the notion of the radius of a von Neumann algebra. Furthermore, we establish good estimates on $\Radius(M_{\Gamma})$ in terms of the radius of $\Gamma$ whenever the vertex algebras $M_v$ posses the property strong (AO). In \cref{subsection:radius-estimates-groups} we establish similar estimates when the vertex algebras $M_v$ are group von Neumann algebras $\calL(G_v)$ of countable icc groups $G_v$.

\subsection{Radius of von Neumann algebras}\label{subsection:radius-of-von-Neumann-algebra}
We introduce the following definition for a simple graph. 

\begin{definition}
	Let $\Gamma$ be a simple graph and let $\Lambda\subseteq \Gamma$ be a subgraph. For $d\in \ZZ_{\geq 0}$ put $$B_\Gamma(\Lambda;d) = \{v\in \Gamma: \Dist_{\Gamma}(v,w)\leq d \text{ for some } w\in \Lambda\}.$$ 
	which is the closed ball of size $d$ around $\Lambda$. Furthermore, define $B_{\Gamma}(\Lambda,\infty) = \bigcup_{d\geq 1}B_{\Gamma}(\Lambda,d)$.
\end{definition}

We will now introduce a similar definition for von Neumann algebras.
\begin{definition}
	Let $M$ be a  diffuse  von Neumann algebra and $A\subseteq M$ a  diffuse  von Neumann subalgebra. For $d\geq 0$ we define the von Neumann algebra $B_M(A;d)$ inductively. Put $B_M(A;0) =A$ and for $d\geq 1$ define 
	$$B_M(A;d) = \left(\bigcup_{\substack{B\subseteq B_M(A;d-1)\\ \text{ diffuse vNa}}}\Nor_{M}(B)\right)''$$
	Moreover, we also define $$B_{M}(A;\infty) = \left(\bigcup_{d\geq 0}B_{M}(A;d)\right)''$$
\end{definition} 

\begin{remark}
For $n,m\in \ZZ_{\geq 0}$ we have that 
\[
B_{M}(A;n+m) = (B_M(  \: \cdot \: ;1) \circ \ldots \circ   B_M(  \: \cdot \: ;1) )(A)  =  B_{M}(B_{M}(A;n);m),
\]
where we have $n+m$ compositions of taking $B_M(  \: \cdot \: ;1)$.
\end{remark}

Recall that the radius of a graph $\Gamma$ was defined in \eqref{definition:radius-graph} and note that it is equal to the infimum of all $d\in \ZZ_{\geq 0}$ for which there exists a vertex $v\in \Gamma$ with $B_{\Gamma}(v;d) = \Gamma$. In a similar way we can introduce the notion of the radius of a von Neumann algebras. 

\begin{definition}\label{definition:radius-von-Neumann-algebra}
	Let $M$ be a diffuse von Neumann algebra. We define $\Radius(M)$ as the infimum of all $d\in \ZZ_{\geq 0}$ such that there exists a diffuse, amenable subfactor $A\subseteq M$ for which $A'\cap M$ is a non-amenable factor and such that $B_M(A;d) = M$. 
\end{definition}
We remark that the definition of $\Radius(M)$ would be more natural with the relaxation that $A$ can be any diffuse amenable von Neumann subalgebra satisfying $B_{M}(A;d) = M$. However, we need the extra restrictions in order to get appropriate lower bounds on $\Radius(M)$.

\begin{proposition}\label{prop:ball-radius}
	Let $\Gamma$ be a simple graph and let $\Lambda\subseteq \Gamma$ be a subgraph. Let $M_{\Gamma} = *_{v,\Gamma}(M_v,\tau_v)$ be a graph product of II$_1$-factors with separable preduals. Then
	\begin{enumerate}
		\item \label{Item=Radius1} For $d\geq 0$ we have \label{Item=prop:ball-radius:ball_equality} $B_{M_{\Gamma}}(M_{\Lambda};d) = M_{B_{\Gamma}(\Lambda;d)}$;
		\item \label{Item=prop:ball-radius:radius} If $\Gamma$ is not complete then $\Radius(M_{\Gamma})\leq \max\{2,\Radius(\Gamma)\}$.
	\end{enumerate}
	\begin{proof}
		\eqref{Item=Radius1} 
		The statement holds true for $d=0$ since $B_{M_{\Gamma}}(M_\Lambda,0) = M_{\Lambda} = M_{B_{\Gamma}(\Lambda,0)}$. We show the statement for $d=1$.
		Let $A\subseteq M_{\Lambda}$ be amenable and diffuse.
		Then $A\not\prec_{M_{\Gamma}} \CC$. Let $\{\Lambda_j\}_{j\in \calJ}$ be the family $\{\varnothing\}$. Then by \cref{prop:embed}\eqref{Item=Embed1b} we obtain 
		$\qNor_{M_{\Gamma}}(A)''\subseteq M_{\Lambda_{\emb}}$ where $\Lambda_{\emb} = \Lambda\cup \bigcup_{v\in \Lambda}\Link_{\Gamma}(v) = B_{\Gamma}(\Lambda;1)$.
		Hence $B_{M}(M_{\Lambda};1)\subseteq M_{B_{\Gamma}(\Lambda;1)}$. To show equality, take $w\in B_{\Gamma}(\Lambda;1)\setminus\Lambda$ and let $v\in \Lambda$ such that $v$ and $w$ share an edge. Let $A\subseteq M_v$ be an amenable and diffuse. Then $\Nor_{M_{\Gamma}}(A)'' \supseteq M_{\Link(v)}\supseteq M_w$. Hence, $M_w\subseteq B_{M_{\Gamma}}(M_{\Lambda};1)$. Hence, we obtain equality. Now let $d\geq 1$ and suppose the statement holds true for $d-1$. Then
		$$B_{M_{\Gamma}}(M_{\Lambda};d) = B_{M_{\Gamma}}(B_{M_{\Gamma}}(M_{\Lambda};d-1);1) = B_{M_{\Gamma}}(M_{B_{\Gamma}(\Lambda;d-1)};1) = M_{B_{\Gamma}(B_{\Gamma}(\Lambda;d-1);1)} = M_{B_{\Gamma}(\Lambda;d)}$$
		This proves the statement by induction.\\

		\eqref{Item=prop:ball-radius:radius} Put $r = \Radius(\Gamma)$. We know $r\geq 1$ and furthermore we may assume $r<\infty$ since otherwise the statement is trivial.  Let $v\in \Gamma$ such that $B_{\Gamma}(v;r) = \Gamma$. Observe, since $\Gamma$ is not complete, that $v$ can be chosen such that $\Link_{\Gamma}(v)$ is not a clique in $\Gamma$.
		By \cite[Proposition 13]{ozawaPrimeFactorizationResults2004} we may let $A\subseteq M_v$ be a diffuse amenable subfactor for which $A'\cap M_{\Gamma} = M_v'\cap M_{\Gamma} = M_{\Link(v)}$. Thus $A'\cap M_{\Gamma}$ is a non-amenable factor. We show that $B_{M_{\Gamma}}(A;r) = M_{\Gamma}$.
		We see that 
		$$M_{\Link(v)} \subseteq \Nor_{M_{\Gamma}}(A)'' \subseteq B_{M_{\Gamma}}(A;1)\subseteq B_{M_{\Gamma}}(M_v;1)\subseteq M_{B_{\Gamma}(v;1)}$$
		Hence,
		\begin{align}\label{eq:radius-graph-estimate}
			M_{B_{\Gamma}(\Link_{\Gamma}(v);1)} = B_{M_{\Gamma}}(M_{\Link(v)};1)\subseteq B_{M_{\Gamma}}(B_{M_{\Gamma}}(A;1);1) \subseteq B_{M_{\Gamma}}(M_{B_{\Gamma}(v;1)};1) = M_{B_{\Gamma}(v;2)}
		\end{align}
		Now, observe that $B_{M_{\Gamma}}(A;2) = B_{M_{\Gamma}}(B_{M_{\Gamma}}(A;1);1)$ and $B_{\Gamma}(\Link_{\Gamma}(v);1) = B_{\Gamma}(v;2)$. If $r\leq 2$ then $B_{\Gamma}(v;2) = \Gamma$ which shows that $\Radius(M_{\Gamma})\leq 2 = \max\{2,r\}$. Thus assume $r\geq 2$. 		By \eqref{eq:radius-graph-estimate} we obtain 
		\[
          M_{B_{\Gamma}(v;2)} = M_{B_{\Gamma}(\Link_{\Gamma}(v);1)}  \subseteq B_{M_\Gamma}(A;2) \subseteq  M_{B_{\Gamma}(v;2)},
        \]
        and so these sets are all equal.
        Thus we obtain $$B_{M_{\Gamma}}(A;r) = B_{M_{\Gamma}}(B_{M_{\Gamma}}(A;2);r-2) = B_{M_{\Gamma}}(M_{B_{\Gamma}(v;2)};r-2) =M_{B_{\Gamma}(v;r)} = M_{\Gamma}$$
		This shows $\Radius(M_{\Gamma})\leq r = \max\{2,r\}$.\\
	\end{proof}
	
\end{proposition}

\begin{proposition}\label{prop:radius_lower_bound}
	Let $\Gamma$ be a simple graph. Let $M_{\Gamma} = *_{v,\Gamma}(M_v,\tau_v)$ be a graph product of II$_1$-factors $M_v$. Let $K\geq 1$ be a constant. Suppose for any amenable diffuse subfactor $A\subseteq M_{\Gamma}$ with $A'\cap M_{\Gamma}$ a non-amenable factor there is a subgraph $\Lambda\subseteq \Gamma$ with $\Radius(B_{\Gamma}(\Lambda,1))\leq K$ such that $A\prec_{M} M_{\Lambda}$.
	Then $$\Radius(\Gamma) -K \leq \Radius(M_{\Gamma}).$$
	\begin{proof}		
		Denote $R = \Radius(M_{\Gamma})$. We may assume  $R<\infty$.
		Let $A\subseteq M_{\Gamma}$ be an amenable, diffuse subfactor for which $A'\cap M_{\Gamma}$ is a non-amenable factor and for which $B_{M_{\Gamma}}(A;R)=M_{\Gamma}$. By assumption $A\prec_{M_{\Gamma}}M_{\Lambda}$ for some subgraph $\Lambda\subseteq \Gamma$ with $\Radius(B_{\Gamma}(\Lambda;1))\leq K$. Let $\{\Lambda_j\}_{j\in \calJ}$ denote the non-empty familiy $\{\varnothing\}$. Then by 	\cref{lemma:unitary-inclusion-in-star-for-vNa} we obtain a unitary $u\in M_{\Gamma}$ so that $u^*Au\subseteq M_{\Lambda_{\emb}}$ where $\Lambda_{\emb} = B_{\Gamma}(\Lambda,1)$.
		Hence, for $d\geq 0$,  we obtain by using \cref{prop:ball-radius} \eqref{Item=Radius1} in the last equality,        
        $$u^*B_{M_{\Gamma}}(A,d)u = B_{M_{\Gamma}}(u^*Au,d)\subseteq B_{M_{\Gamma}}(M_{B_{\Gamma}(\Lambda;1)};d) = M_{B_{\Gamma}(B_{\Gamma}(\Lambda;1);d)}$$
		Then 
		$$M_{\Gamma} = u^*B_{M_{\Gamma}}(A;R)u \subseteq M_{B_{\Gamma}(B_{\Gamma}(\Lambda;1);R)}$$
		so that $\Gamma = B_{\Gamma}(B_{\Gamma}(\Lambda;1);R)$.
		Therefore we obtain $$\Radius(\Gamma)\leq \Radius(B_{\Gamma}(\Lambda;1)) + R = K+\Radius(M_{\Gamma}),$$
		which completes the proof.
	\end{proof}
\end{proposition}

\begin{theorem}\label{theorem:radius-graph-AO}
		Let $\Gamma$ be a simple graph that is not complete. Let $M_{\Gamma} = *_{v,\Gamma}(M_v,\tau_v)$ be a graph product of II$_1$-factors $M_v$ that satisfy condition strong (AO) and have separable predual. Then
		$$\Radius(\Gamma)-2 \leq \Radius(M_{\Gamma})\leq \max\{2,\Radius(\Gamma)\}$$
		In particular this holds true when $M_{\Gamma}$ is a graph products of hyperfinite II$_1$-factors.
		\begin{proof}
			The upper bound is due to \cref{prop:ball-radius}\eqref{Item=prop:ball-radius:radius}. To obtain the lower bound we show that the condition of \cref{prop:radius_lower_bound} is satisfied with constant $K=2$. Let $A\subseteq M_{\Gamma}$ be amenable and diffuse and such that $A'\cap M_{\Gamma}$ is non-amenable. By \cref{Thm=KeyAlternatives} we obtain $A\prec_{M_{\Gamma}} M_{\Lambda}$ for some non-empty subgraph $\Lambda\subseteq \Gamma$ with $\Link(\Lambda)$ non-empty. Let $v\in \Link(\Lambda)$; so certainly $v \not \in  \Lambda$. Then $\Lambda\subseteq \Link(v)$. Hence, $B_{\Gamma}(\Lambda;1)$ equals $B_{\Gamma}(v,2)$ and has radius at most $2$. 
			This proves the lower bound.
		\end{proof}
\end{theorem}

\begin{remark}\label{remark:distinguish-vNa-using-radius}
    We remark that we can use \cref{theorem:radius-graph-AO} to distinguish von Neumann algebras coming from graph products. Indeed, let $\Gamma$ and $\Lambda$ be two simple graphs with $2\leq \Radius(\Gamma)< \Radius(\Lambda)-2$. Let $R_{\Gamma} = *_{v,\Gamma}(R_v,\tau_v)$ and $R_{\Lambda} = *_{v,\Lambda}(R_v,\tau_v)$ be graph products of hyperfinite II$_1$-factors $R_v$. Then by \cref{theorem:radius-graph-AO}  we obtain $$\Radius(R_{\Gamma}) \leq \Radius(\Gamma) < \Radius(\Lambda) -2\leq \Radius(R_{\Lambda})$$
    Thus, in particular $R_{\Gamma}\not\simeq R_{\Lambda}$.
\end{remark}

\subsection{Radius estimates for graph products groups}	
\label{subsection:radius-estimates-groups}
We now show that the statement of \cref{theorem:radius-graph-AO} also holds true when the vertex von Neumann algebras $M_v$ are group von Neumann algebras $\calL(G_v)$ of countable icc groups (\cref{theorem:radius-graph-groups}).
We state the following definitions.
\begin{definition}
	Let $G$ be a countable discrete group and let $\calS$ be a family of subgroups of $G$. Then a subset $F\subseteq G$ is called \textit{small relative to} $\calS$ if  $$F\subseteq \bigcup_{i=1}^k g_iG_ih_i$$ for some $k\geq 1$, groups $G_1,\ldots, G_k\in \calS$ and elements $g_1,\ldots, g_k,h_1,\ldots, h_k\in G$. 
\end{definition}

\begin{definition}
	Let $G$ be a countable discrete group and let $\calS$ be a family of subgroups of $G$. Let $V\subseteq  \calL(G)$   be a norm bounded subset. 
	We write $$V\subseteq_{\text{approx}}\calL(\calS)$$  if for every $\epsilon>0$ there is a subset $F\subseteq G$ that is small relative to $\calS$ and satisfies for all  $v \in V$  that $\|v - P_{F}(v)\|_{2} \leq \epsilon$ (here $P_{F}:\ell^2(G)\to \ell^2(F)$ denotes the orthogonal projection). 
\end{definition}

The following proposition is similar to \cite[Claim 6.15]{chifanTensorProductDecompositions2018} and follows from the results in \cite{vaesOnecohomologyUniquenessGroup2013}. In the proof we write $(B)_1$ for the closed unit ball of the von Neumann algebra $B$. 
	\begin{proposition}\label{prop:intersection-embeddings-for-groups}
		Let $\Gamma$ be a simple graph and for $v\in \Gamma$ let $G_v$ be a countable icc group.
		Let $G_{\Gamma} = *_{v,\Gamma}G_v$ be the graph product and let $B\subseteq \calL(G_{\Gamma})$ be a von Neumann subalgebra for which $B'\cap \calL(G_{\Gamma})$ is a factor. Let $\{\Lambda_i\}_{i\in I}$ be a collection of subgraphs of $\Gamma$ and let $\Lambda = \bigcap_{i}\Lambda_i$ be their intersection. If $B\prec_{\calL(G_{\Gamma})}\calL(G_{\Lambda_i})$ for all $i$ then $B\prec_{\calL(G_{\Gamma})}\calL(G_{\Lambda})$
		\begin{proof}
			Assume $B\prec_{\calL(G_{\Gamma})} \calL(G_{\Lambda_{i}})$ for $i\in I$. We show $B\prec_{\calL(G_{\Gamma})}\calL(G_{\Lambda})$.  For $i\in I$ we can by
			\cite[Lemma 2.5]{vaesOnecohomologyUniquenessGroup2013} obtain a non-zero projection $q_i\in B'\cap \calL(G_{\Gamma})$ such that for $\calS_{i}:= \{G_{\Lambda_i}\}$ we have
			\[
			(Bq_i)_1 \subset_{{\rm approx}} \calL(\calS_i)
			\]
			Moreover, by \cite[Proposition 2.6]{vaesOnecohomologyUniquenessGroup2013} we may assume $q_i\in \calZ(\Nor_{\calL(G_{\Gamma})}(B)'')$. Note that
			\begin{align}
				q_i \in \calZ(\Nor_{\calL(G_{\Gamma})}(B)'') \cap (B' \cap \calL(G_{\Gamma}))\subseteq \calZ(B' \cap \calL(G_{\Gamma})) = \CC1
			\end{align}
			Thus $q_i = 1$. Denote 
			$$\calS = \{  \bigcap_{i\in I} h_i G_{\Lambda_i} h_i^{-1} \mid h_i \in G_{\Gamma} \text{ for } i\in I \}.$$
			From \cite[Lemma 2.7]{vaesOnecohomologyUniquenessGroup2013} it follows that  
			$(B)_1 \subset_{{\rm approx}} \mathcal{L}(\mathcal{S})$.
			Then from \cite[Proposition 3.4]{antolinTitsAlternativesGraph2013} for each $(h_i)_{i\in I}, h_i \in G_{\Gamma}$ there is a subgraph $\Lambda_{0} \subseteq \Lambda$ and $k \in G_{\Gamma}$ such that 
			\[
			\bigcap_{i\in I} h_i G_{\Lambda_i} h_i^{-1} = k G_{\Lambda_0} k^{-1} \subseteq k G_{\Lambda} k^{-1}. 
			\]
			Thus, putting $\calS_0 = \{G_{\Lambda}\}$ it follows that $(B)_1 \subseteq_{\text{approx}} \calL(\calS_0)$ and hence by \cite[Lemma 2.5]{vaesOnecohomologyUniquenessGroup2013} we obtain $B \prec_{\calL(G_{\Gamma})} \mathcal{L}( G_{\Lambda} )$.  
		\end{proof}
	\end{proposition}

\begin{remark}
\cref{prop:intersection-embeddings-for-groups} allows us to prove the following theorem which holds for a very general class of graph products of group von Neumann algebras. The reason we restrict to group von Neumann algebras is that a version of \cref{prop:intersection-embeddings-for-groups} for more general von Neumann algebras is not known to the authors. 
\end{remark}

	\begin{theorem}\label{theorem:radius-graph-groups}
		Let $\Gamma$ be a simple graph that is not complete. For $v\in \Gamma$ let $G_v$ be a countable icc group.	Let $G_{\Gamma} = *_{v,\Gamma}G_v$ be the graph product. Then
		$$\Radius(\Gamma) - 2 \leq \Radius(\calL(G_{\Gamma}))\leq \max\{2,\Radius(\Gamma)\}.$$
		\begin{proof}
			The upper bound on $\Radius(\calL(G_{\Gamma}))$ follows immediately from \cref{prop:ball-radius} since $\calL(G_{v})$ is a II$_1$-factor for $v\in \Gamma$. To prove the lower bound we show the condition of \cref{prop:radius_lower_bound} is satisfied with $K=2$. Put $M_{v} = \calL(G_v)$ and let $M_{\Gamma} = *_{v,\Gamma}(M_v,\tau_v) = \calL(G_\Gamma)$ be the graph product. Let $R\subseteq M_{\Gamma}$ be an amenable II$_1$-factor for which $R'\cap M_{\Gamma}$ is a non-amenable factor. We need to show that $R\prec_{M_{\Gamma}} M_{\Lambda}$ for some $\Lambda\subseteq \Gamma$ with $\Radius(B_{\Gamma}(\Lambda;1))\leq 2$.
			Let $I$ be the set of all vertices $v$ in $\Gamma$ for which $\Nor_{M_{\Gamma}}(R)''$ is amenable relative to $M_{\Link_{\Gamma}(v)}$ inside $M_{\Gamma}$.			
			By \cref{Thm=Square} we obtain that $\Nor_{M_{\Gamma}}(R)''$ is amenable relative to $M_{\Link_{\Gamma}(I)}$ inside $M_{\Gamma}$. Since $\Nor_{M_{\Gamma}}(R)''$ is non-amenable (as it contains $R'\cap M_{\Gamma}$), we obtain that $\Link_{\Gamma}(I)$ is non-empty. Let $w\in \Link_\Gamma(I)$. Then $ I\subseteq B_{\Gamma}(w;1)$ so that $B_{\Gamma}(I;1)\subseteq B_{\Gamma}(w,2)$. Thus since $w\in B_{\Gamma}(I;1)$ we see that $B_{\Gamma}(I;2)$ has radius at most $2$.
			
			Now let $J\subseteq \Gamma$ be the set of all $v\in \Gamma$ for which 
			$R\prec_{M_{\Gamma}} M_{\Gamma\setminus \{v\}}$. Then since $R'\cap \calL(G_{\Gamma})$ is a factor we obtain by \cref{prop:intersection-embeddings-for-groups} that $R \prec_{M_{\Gamma}} M_{\Gamma\setminus J}$. Now, if $\Gamma\setminus J\subseteq I$ then $R\prec_{M_{\Gamma}} M_{I}$ which shows that we may take $\Lambda = I$.
			Thus assume $\Gamma\setminus J\not\subseteq I$. Take $v\in \Gamma\setminus J$ with $v\not\in I$.
			We can decompose 
			$$M_{\Gamma} = M_{\Star(v)} *_{M_{\Link(v)}} M_{\Link(v)}.$$
			Since $R$ is amenable we get by \cref{thm:prelim:alternatives} that at least one of the following holds true
			\begin{enumerate}
				\item $R\prec_{M_{\Gamma}} M_{\Link(v)}$
				\item $\Nor_{M_{\Gamma}}(R)'' \prec_{M_{\Gamma}} M_{\Star(v)}$ or $\Nor_{M_{\Gamma}}(R)'' \prec_{M_{\Gamma}} M_{\Gamma\setminus \{v\}}$.
				\item $\Nor_{M_{\Gamma}}(R)''$ is amenable relative to $M_{\Link(v)}$ inside $M_{\Gamma}$.
			\end{enumerate}
			Since $v$ is not in $I\cup J$ we must have that $R\prec_{M_{\Gamma}} M_{\Link(v)}$ or $\Nor_{M_{\Gamma}}(R)'' \prec_{M_{\Gamma}} M_{\Star(v)}$. Thus in particular we obtain $R\prec_{M_{\Gamma}} M_{\Star(v)}$. Observe that $B_{\Gamma}(\Star(v);1)$ has radius at most $2$. Hence we may take $\Lambda := \Star(v)$. This finishes the proof. 
		\end{proof}
	\end{theorem}

\printbibliography

\end{document}